\documentclass[12pt]{amsart}

\usepackage[margin=1.25in]{geometry}

\usepackage{cancel}
\usepackage{latexsym, amssymb, amsmath,esint,float}
\usepackage{soul}
\usepackage{amsfonts, graphicx}
\usepackage{graphicx,color}

\usepackage{tikz}
\usepackage{tikz-3dplot}

\newcommand{\be}{\begin{equation}}
\newcommand{\ee}{\end{equation}}
\newcommand{\beq}{\begin{eqnarray}}
\newcommand{\eeq}{\end{eqnarray}}
\usepackage{mathtools}
\mathtoolsset{showonlyrefs}

\usepackage{wasysym,stmaryrd}
\newtheorem{thm}{Theorem}[section]

\newtheorem{lma}[thm]{Lemma}
\newtheorem{prop}[thm]{Proposition}

\newtheorem{defn}[thm]{Definition}
\theoremstyle{remark}
\newtheorem{rem}[thm]{Remark}
\numberwithin{equation}{section}

\newtheorem{claim}{Claim}[section]

\def\be{\begin{equation}}
\def\ee{\end{equation}}
\def\bee{\begin{equation*}}
\def\eee{\end{equation*}}

\newcommand{\ri}{\rightarrow}
\newcommand{\eps}{\varepsilon}

\def\Ric{\text{\rm Ric}}
\def\Rm{\text{\rm Rm}}

\def\p{\partial}

\def\heat{\left(\frac{\p}{\p t}-\Delta_{g(t)}\right)}
\def\tr{\operatorname{tr}}

\def\e{\varepsilon}

\def\a{{\alpha}}
\def\b{{\beta}}

\def\R{\mathbb{R}}

\begin{document}

\title[]
{Higher-dimensional flying wing Steady Ricci Solitons}

\author{Pak-Yeung Chan}
\address[Pak-Yeung Chan]{Department of Mathematics, National Tsing Hua University, Hsin-Chu, Taiwan}
\email{pychan@math.nthu.edu.tw}

 \author{Yi Lai}
\address[Yi Lai]{Department of Mathematics, Rowland Hall, University of California Irvine, CA 92617, USA
}
\email{ylai25@uci.edu}

 \author{Man-Chun Lee}
\address[Man-Chun Lee]{Department of Mathematics, The Chinese University of Hong Kong, Shatin, Hong Kong, China}
\email{mclee@math.cuhk.edu.hk}

\renewcommand{\subjclassname}{
  \textup{2020} Mathematics Subject Classification}

\date{\today}

\begin{abstract}
For any $n\geq 4$, we construct an $(n-2)$-parameter family of steady gradient Ricci solitons with non-negative curvature operator and prescribed by the eigenvalues of Ricci tensor at a 
critical point of the soliton potential. Among them lies an 
$(n-3)$-parameter subfamily of non-collapsed solitons.
These solitons generalized the flying wings constructed by the second named author and produced new examples of steady gradient Ricci solitons with non-negative curvature operator for $n\geq 4$.  Our approach is based on constructing continuous families of Ricci flows smoothing emanating from continuous families of spherical polyhedra which still preserves symmetry. 

This is built upon a new stability result of Ricci flows with scaling invariant estimates. 
As another application of the method, we prove the stability of asymptotically conical expanding solitons constructed by Deruelle under $L^\infty$ perturbation of links.  In particular, the $C^0$-convergence of smooth links implies the smooth convergence of the expanding solitons.
\end{abstract}

\maketitle

\markboth{Pak-Yeung Chan, Yi Lai, Man-Chun Lee}{HIGHER-DIMENSIONAL FLYING WING STEADY RICCI SOLITONS}

\section{Introduction}\label{sec: introduction}
\subsection{Flying wing steady solitons}
The Ricci flow has emerged as a powerful tool in geometric analysis since 
Perelman’s resolution of the Poincar\'e conjecture. 
Singularity analysis in Ricci flow has become a central theme in geometric analysis, 
both for their intrinsic geometric interest and for their potential applications 
in topology. 
A rich source of singularity models comes from \emph{Ricci solitons}, which are self-similar solutions to the Ricci flow and fall into three categories according to whether the flow is shrinking, steady, or expanding.

A steady gradient Ricci soliton is a smooth complete Riemannian manifold $(M,g)$ satisfying
\begin{equation}
    \Ric=\nabla^2 f
\end{equation}
for some smooth function $f$ on $M$ called a potential function.

In dimension two, Hamilton’s cigar soliton is the unique non-flat example \cite{Hamilton1988}. 
In dimension three, the classical examples are the Bryant soliton and $\R\times\textnormal{Cigar}$. Hamilton further conjectured 
the existence of steady solitons asymptotic to sectors with angle 
$\alpha \in (0,\pi)$, known as \emph{flying wings}. The second-named author confirmed this conjecture by constructing a one-paramter $\mathbb Z_2\times O(2)$-symmetric flying wings \cite{Lai2024,lai20253d}.
Using the same method, she also constructed non-collapsed, $\mathbb Z_2\times O(n-1)$-symmetric flying wings
for any dimension $n\geq 4$ \cite{Lai2024}. Recently, Lavoyer-Peachey \cite{LavoyerPeachey2025} generalized the method of the second named author to construct $O(p)\times O(q)$-symmetric $(p+q)$-dimensional flying wings for any pair of integers $p,q\ge2$.
See also for the works of Chan-Conlon-Lai  \cite{ChanConlonLai2024} and Apostolov-Cifarelli \cite{ApostolovCifarelli2025} for K\"ahler flying wings.

In this paper, we widen the class of
flying wings for $ n\ge 4$ by constructing large families of steady solitons with non-negative curvature operator.  In particular, we construct non-negatively curved steady solitons with prescribed Ricci curvature at a 
critical point of $f$ and prescribed symmetric group:

\begin{thm}\label{thm:soliton-polyhedron}
Let $n\ge4$, for any $0\le\lambda_1\le\cdots\le\lambda_{n-1}$ with $\lambda_1+\cdots+\lambda_{n-2}+2\lambda_{n-1}=1$, there exists an $n$-dimensional steady gradient soliton $(M,g,f,p)$ with $\mathrm{Rm}\geq 0$, $R(p)=1$, and the eigenvalues of Ricci curvature at $p$ are equal to $(\lambda_1,\cdots,\lambda_{n-1},\lambda_{n}= \lambda_{n-1})$.   Moreover, if for some $k=1,...,n-1$, and $1= i_1<i_2<\cdots< i_k< i_{k+1}= n$ we have $\lambda_{i_{j}}=\cdots=\lambda_{i_{j+1}-1}$ for $j=1,\cdots,k$, then the steady soliton is 
$O(i_2-i_1)\times\cdots\times O(i_k-i_{k-1})\times O(i_{k+1}-i_k+1)$-symmetric. 
\end{thm}

Apostolov-Cifarelli \cite{ApostolovCifarelli2025} constructed complex $n$-dimensional K\"ahler steady gradient solitons with of $U(d_1+1)\times \cdots U(d_\ell+1)$ symmetry, where $\ell\ge 1$, $d_i\ge 0$ and $\ell+\sum_{i=1}^\ell d_i=n$. They also showed that these solitons have positive sectional curvature when $n=2$.
Theorem \ref{thm:soliton-polyhedron} can be viewed as a Riemannian analogy of Apostolov-Cifarelli's result with additional non-negative curvature property. 
Note that Chan-Conlon-Lai \cite{ChanConlonLai2024} constructed K\"ahler steady gradient solitons with $U(n-1)\times U(1)$-symmetry non-negative curvature operator \cite{ChanConlonLai2024}, it is natural to wonder if there exist  K\"ahler steady gradient solitons having the same symmetries in Apostolov-Cifarelli \cite{ApostolovCifarelli2025} and additional non-negative curvature operator.

 Among the steady solitons constructed in Theorem 1.1, it provides a large family of non-collapsed steady solitons, using a similar proof as \cite[Corollary 3.5]{Lai2024}.
\begin{thm}\label{thm:non-collapsed}
For any $n\ge 4$, and $k=2,\cdots, n-2$, there exists an $(n-k-1)$-parameter family of pairwise different $O(k+1)$-symmetric, non-collapsed $n$-dimensional steady gradient Ricci solitons with {non-negative} curvature operator.
\end{thm}

Note that $O(k+1)$-symmetry implies $O(k)$-symmetry for any $k\in\mathbb N$, and the rotational symmetry is $O(n)$-symmetry for $n$-dimensional manifolds. 
Thus $O(2)$-symmetry is the weakest symmetry. 
In dimension 3, the second named author proved that all steady gradient solitons are $O(2)$-symmetric \cite{Lai2022_O(2)}. Hence, we conjecture that for arbitrary dimension, all steady gradient solitons with positive curvature operator are $O(2)$-symmetric.

In previous works \cite{ApostolovCifarelli2025, Bryant2005, Cao1996, ChanConlonLai2024, Lai2024, lai20253d, LavoyerPeachey2025}, all known steady gradient Ricci solitons of positive sectional curvature have at most $2$ distinct eigenvalues of $\Ric$ at the unique critical point of $f$. With a more judicious choice of $\lambda_i$, Theorem \ref{thm:soliton-polyhedron} yields large families of new examples of steady solitons with positive curvature operator and at least $3$ distinct eigenvalues of $\Ric$ at the unique critical point of $f$ in dimensions four and higher, see Remark \ref{eig}.

In comparison to mean curvature flow,
the non-collapsed  steady gradient solitons in
Theorem \ref{thm:soliton-polyhedron} is analogues of the mean curvature flow entire graphical translators and graphical translators on slabs of bounded width, which are constructed by Hoffman-Ilmanen-Martin-White \cite{HoffmanIlmanenMartinWhite2019}.

In 
dimension three, the uniqueness of 3D flying wings remains open. In dimension 
four, it is conjectured by Haslhofer \cite{Haslhofer2024} that all $\kappa$-solutions (non-collapsed ancient flow with $\Rm\ge0$) are the spherical solution, the cylindrical solution, Perelman's oval,  $\mathbb {CP}^2$, the 4D Bryant soliton, the one-parameter 
family flying wings constructed by the second named author,  or the one-parameter ovals constructed by Haslhofer \cite{Haslhofer2024}.  Motivated by these, we conjecture that all non-collapsed steady gradient Ricci solitons with non-negative curvature operator and cylindrical tangent flow at infinity of the form $\mathbb R^{k}\times\mathbb S^{n-k}$, $n-k\ge2$ is isometric (up to scaling) to one of the examples we constructed in Theorem~\ref{thm:soliton-polyhedron}.

\textit{Sketch of Theorem \ref{thm:non-collapsed}:}
We now briefly outline the ideas behind our construction.
First, we construct a sequence of smooth metrics 
$(\mathbb{S}^{n-1},\{g_{i,\mathbf x}\}_{\mathbf x\in\Omega,i\in\mathbb N})$, parametrized by a $(n-2)$-dimensional simplex $\Omega$ with $(n-1)$ vertices $\mathbf o_1,\cdots,\mathbf o_{n-1}$, such that 
all metrics satisfy $\Rm\ge1$, and the volume uniformly converge to $0$ as $i\to\infty$. 
So by the uniqueness theorem of expanding solitons \cite{Deruelle2016}, this can be lifted to a sequence of smooth expanding solitons 
$\{\mathcal E_{i,\mathbf x}\}_{\mathbf x\in\Omega,i\in\mathbb N}$ on $ \mathbb{R}^{n}$ parametrized by $\Omega$, such that each $\mathcal E_{i,\mathbf x}$ satisfies $\Rm\ge0$, $R=1$ at the critical point, and is asymptotic to the metric cone over $(\mathbb{S}^{n-1},g_{i,\mathbf x})$.  
So it follows by the {same argument in \cite{Lai2024}} that for any sequence $\mathbf x_i\in\Omega$, the expanding solitons converge to a steady soliton with $\Rm\ge0$, and $R=1$ at the critical point.
In particular, the sequence of the expanding solitons $\mathcal E_{i,\mathbf o_1},\cdots,\mathcal E_{i,\mathbf o_{n-1}}$ corresponding to the vertices converge to the $(n-1)$ steady solitons: 
$$\textnormal{Bry}^{n},\;\R\times \textnormal{Bry}^{n-1},\; \cdots,\; \R^{n-3}\times \textnormal{Bry}^{3},\;\R^{n-2}\times \textnormal{Cigar}.$$
Consider the map from the expanding solitons to the  Ricci eigenvalues at the critical point (note there are always two equal eigenvalues by the $O(2)$-symmetry), we obtain a smooth map from the simplex $\Omega$ to an $(n-2)$-dimensional simplex $\Delta$, with vertices the Ricci eigenvalues of the steady solitons corresponding to the vertices of $\Omega$. 
Then we show this map is surjective, thus by taking limits we obtain smooth steady solitons $\{\mathcal S_{\mathbf x}\}_{\mathbf x\in\Delta}$ parametrized by the simplex $\Delta$.
The non-collapsed solitons in Theorem \ref{thm:non-collapsed} is parametrized by a $(n-3)$-face of $\Delta$.

To construct the metrics on $\mathbb{S}^{n-1}$, we first note that the spherical suspensions produces metrics that almost satisfy all our expectation, except the smoothness.

We want to smooth out the cone points {continuously} that inevitably form when taking the spherical suspensions, while keeping the following properties invariant: (i), the metric stays close to the original one in the Gromov-Haudorff sense; (ii), the curvature condition $\Rm\ge1$ still holds; (iii), the symmetry is preserved.
In {\cite{ChanConlonLai2024,ChanLee2025,Lai2024,LavoyerPeachey2025}}, by smoothing the one-variable warping function at two ends we can achieve property (i) and (ii), and property (iii) holds trivially since the symmetry is quite simple there.

 However, it is unclear to us that whether the one variable smoothing technique therein will preserve the symmetry of the singular metric in the spherical suspension case.
 In this work, we adopt a more delicate parabolic smoothing approach, the Ricci flow coming out of non-smooth initial data. 

\subsection{Ricci flow starting from non-smooth initial data}
{
For smooth initial data, the existence, preservation of symmetry, or more generally the uniqueness and stability of Ricci flow with smooth initial data on closed manifold has been extensively studied in the literature.

Given a metric with polyhedral singularity on sphere $\mathbb{S}^{n-1}$ obtained from taking spherical suspension iteratively:
\[g_{0,\b}=g_{0,(\beta_1,\cdots,\beta_{n-1})}:=\beta_{1}^2(dx_{1}^2+\beta_{2}^2\sin^2 x_{1}(dx_{2}^2+(\cdots+\beta_{n-1}^2\sin^2 x_{n-2}dx_{n-1}^2)),\]
where $x_{1}\in[0,\pi],\cdots,x_{n-2}\in[0,\pi],x_{n-1}\in[0,2\pi]$. We might write $g_{0,\b}$ in a inductive way as 
\begin{equation}
g_{0,(\beta_1,\cdots,\beta_{n-1})}= \b_1^2 \left( dx^2_1+ \sin^2 x_1 \cdot g_{0,(\beta_2,\cdots,\beta_{n-1})}\right)
\end{equation}
so that $g_{0,(\beta_2,\cdots,\beta_{n-1})}$ is the link of the spherical suspension and is the singular metrics on $\mathbb{S}^{n-2}$ with polyhedral singularity.
In this work, we want to establish a canonical way to construct continuous Ricci flow smoothing starting from these non-smooth initial metrics with desired properties as $\b$ varies, see Theorem~\ref{thm:const-2} for the detailed statement.

We briefly discuss the idea behind and the subtlety involved.  
By way of induction,  if we have already constructed the Ricci flow smoothing $g_{(\beta_2,\cdots,\beta_{n-1})}(t)$ coming out of the link $g_{0,(\beta_2,\cdots,\beta_{n-1})}$, then we might consider the relatively less singular metric 
\begin{equation}\label{eqn:smoothed-link}
g_{0,(\beta_1,\cdots,\beta_{n-1}),\tau}= \b_1^2 \left( dx^2_1+ \sin^2 x_1 \cdot g_{(\beta_2,\cdots,\beta_{n-1})}(\tau)\right)
\end{equation}
for $\tau>0$, which is a metric with conical singularities. This falls inside the set-up considered by Gianniotis-Schulze  \cite{GianniotisSchulze2018} which allows us to smooth out $g_{0,(\beta_1,\cdots,\beta_{n-1}),\tau}$ by a Ricci flow $g_{(\beta_1,\cdots,\beta_{n-1}),\tau}(t)$ for any $\tau>0$. Using the maximum principle developed by the third named author and Tam \cite{LeeTam2025}, one can show that the Ricci flow $g_{(\beta_1,\cdots,\beta_{n-1}),\tau}(t)$ satisfies all properties we need, after we pass $\tau\to 0^+$. By way of induction, it is not hard to construct Ricci flow $g_{(\beta_1,\cdots,\beta_{n-1})}(t)$ coming out of $g_{0,(\beta_1,\cdots,\beta_{n-1})}$. This will be discussed in Section~\ref{sec:existence}.

The challenging part is to show the continuity of $g_{(\beta_1,\cdots,\beta_{n-1})}(t)$ as $(\beta_1,\cdots,\beta_{n-1})$ varies. This is related to how one can gauge between $g_{(\beta_1,\cdots,\beta_{n-1})}(t)$ and $g_{(\beta'_1,\cdots,\beta'_{n-1})}(t)$ when they are initially close. When the initial data are smooth, it is very often to use the Ricci-harmonic map heat flow between two Ricci flows: $
\partial_t F=\Delta_{g_1(t),g_2(t)}\, F,$
to transform a Ricci flow $g_1(t)$ to a Ricci-DeTurck flow $\hat g_1(t):=(F^{-1}(t))^*g_1(t)$ with respect to another Ricci flow $g_2(t)$, as long as $F(t)$ remains a diffeomorphism. The Ricci-DeTurck flow is a strictly parabolic system where the stability is relatively clear. In our case, it is tempting to mimic the smooth case by working on the regularized level, in the approximation scheme of Gianniotis-Schulze  \cite{GianniotisSchulze2018}. On the regularized level, the initial data is given by gluing in the expanding soliton constructed by Deruelle \cite{Deruelle2016} near the singularities  and thus the tangent flow of the Ricci flow is governed by the expanding Ricci soliton constructed by Deruelle \cite{Deruelle2016}. Deruelle's argument is based on PDE deformation and continuity method. In particular, the stability of those expanders are only known to be true as the link varies in $C^{3,\a}$ topology, while our link of polyhedral singularties varies only in $L^\infty$ topology. Therefore, we need to modify the gluing construction, compatible with $L^\infty$ variation of links.

The ultimate goal is to construct Ricci flows so that the Ricci flow with singular initial data can be gauged with each other whenever the initial singular metrics are close in weak sense, i.e. in $L^\infty$ sense. To achieve this, we construct Ricci flow $g_{(\beta_1,\cdots,\beta_{n-1})}(t)$ coming out of $g_{0,(\beta_1,\cdots,\beta_{n-1})}$ in a slightly different way so that the stability partially holds by construction. In short, the overall idea is to extend the existence using the above mentioned construction at \textit{one} reference point $\hat \b_0:=(\hat \beta_1,\cdots,\hat \beta_{n-1})$ in the parameter space and propagate the existence to other $\b$ using weak stability. More precisely, we fix a reference point $\hat \b_0$ and implement the above mentioned strategy to obtain a reference Ricci flow $g_{(\hat \beta_1,\cdots,\hat \beta_{n-1})}(t)$. Instead of repeating the construction using the same method on a different $\b$, we consider construction using weak stability with respect to $g_{(\hat \beta_1,\cdots,\hat \beta_{n-1})}(t)$.

The weak stability of Ricci flow has been studied extensively and is known to be powerful in studying questions in scalar curvature. We refer readers to \cite{Burkhardt2019,ChuLee2025,DeruelleLamm2017,KochLamm2012,Simon2002} and the reference therein. Recall that for a background Ricci flow $\tilde g(t)$, $g(t)$ is said to be a  solution to the Ricci DeTurck flow with background $\tilde g(t)$ if
\begin{equation}\label{eqn:RDF1}
\left\{
\begin{array}{ll}
\partial_t g_{ij}=-2R_{ij}+\nabla_i V_j+\nabla_j V_i;\\[1mm]
V^k=g^{ij}\left(\Gamma_{ij}^k-\tilde \Gamma_{ij}^k \right).
\end{array}
\right.
\end{equation}
This is a strictly parabolic system of equations and is diffeomorphic to a Ricci flow through the Ricci-DeTurck ODE. In \cite{DeruelleLamm2017}, the weak stability of the Ricci-DeTurck flow with respect to an  expanding soliton $\tilde g(t)$ with nonnegative curvature operator and quadratic curvature decay  has been studied. This is the starting point of this work. We generalize this to background Ricci flow which has curvature bounded from below and satisfies scaling invariant estimates:
\begin{thm}\label{1234}
Suppose $M$ is a complete manifold and $\tilde g(t),t\in [0,T]$ is a smooth Ricci flow  on $M$ such that $\tilde g(t)$ satisfies the following on $(0,T]$ for some $\a>0$:
\begin{enumerate}
\item[(a)] $|\Rm(\tilde g(t))|\leq \a t^{-1}$;
\item[(b)] $\mathrm{inj}(\tilde g(t))\geq \sqrt{\a^{-1}t}$;
\item[(c)] $\mathrm{Rm}(\tilde g(t))\geq -1$;
\end{enumerate}
Then there exists $\bar\e_0(n,\a),\Lambda(n,\a)>0$ so that the following holds: if $g_0$ is a smooth complete metric on $M$ such that $||g_{0}-\tilde g(0)||\leq \e<\bar \e_0$, then it admits a smooth solution to \eqref{eqn:RDF1} on $M\times [0,T\wedge 1]$ such that 
\begin{equation}\label{estt}
    |g(t)-\tilde g(t)|\leq \Lambda \e 
\end{equation}
Furthermore, if $g(t)$ and $\hat g(t)$ are both solutions to \eqref{eqn:RDF1} such that \eqref{estt} holds, then 
\begin{equation}
    \sup_{M\times [0,T\wedge 1]}||g(t)-\hat g(t)||_{L^\infty(M,\tilde g(t))}\leq \Lambda \cdot  ||g(0)-\hat g(0)||_{L^\infty(M,\tilde g(0))}.
\end{equation}
\end{thm}

In particular, this generalizes the stability in \cite{Burkhardt2019,DeruelleLamm2017,KochLamm2012} where the reference Ricci flow $\tilde g(t)$ is assumed to have bounded curvature. Indeed, the metrics are allowed to be non-smooth initially to some extent, which will be important for applications. We refer readers to Section~\ref{sec:stability} for a more general statement of the stability result which allows $L^\infty$ type singularity on the initial data.

By way of induction where Ricci flow smoothing and gauge fixing through Ricci-harmonic map heat flow are established in earlier dimensions, the weak stability allows us to construct Ricci flow $g_{\b,\tau}(t)$ coming out of $g_{0,\b,\tau}$, given in \eqref{eqn:smoothed-link}, through Ricci-DeTurck flow and Ricci-DeTurck ODE, for $\b$ sufficiently close to $\hat \b_0$. The Ricci flow constructed in this way is stable by construction, for all $\tau\to 0$. However unlike the smooth case, the initial data of the Ricci flow (potentially) experienced a slight change of gauge, due to loss of regularity. That said, the Ricci flow constructed can only attain the original initial data with a gauge change, even on the region where the initial data is regular. 

Despite the concern of gauge changing, we connect any $\b$ back to $\hat\b_0$ through an appropriate choice of ray $\gamma$ that lives inside the parameter space with $\gamma(0)=\hat\b_0$ and $\gamma(1)=\b$. Indeed, it is unclear to us whether the choice of ray might lead to a different Ricci flow at the end. Nevertheless, with a fixed choice of $\gamma$ we decompose it into finitely many pieces by considering $\gamma(t_i),\, t_i:=i\, \Delta t$ where $\Delta t$ is (uniformly) small. In this case, the stability of Ricci-DeTurck flow with respect to $g_{\b',\tau}(t)$ allows us to construct Ricci flows $g_{\hat\b,\tau}(t)$ with slight gauge change of their initial metrics, for $\hat \b\in \gamma([t_i,t_{i+1}])$, provided that the Ricci flow has already been constructed at $\b'=\gamma(t_i)$. 

Arguing in this way inductively, we construct Ricci flow $g_{\b,\tau}(t)$ for all $\b$ in the parameter space. In practice, we first carry this out in the $\tau$-level, i.e. we will consider the stability of Ricci-DeTurck flow from $g_{0,(\beta_1,\cdots,\beta_{n-1}),\tau}$, and then let $\tau\to 0$. The gauge fixing using harmonic map heat flow in the earlier dimension will be involved in order to compare different regularized links $g_{(\beta_2,\cdots,\beta_{n-1})}(\tau)$ in $L^\infty$ topology. Similarly whenever a different reference Ricci flow is used, we need to re-fix the gauge of the initial data due to the Ricci-DeTurck ODE in earlier step and gauge fixing in earlier dimension, so that stability with respect to the new Ricci flow applies in the new gauge. Despite the technicality, the stability of the new construction can be achieved easily for $\b$ and $\hat\b$ living along the same ray. The remaining subtlety lies in the stability when $\b$ and $\hat \b$ are close to each other but not necessarily along the same ray. Particularly, the Ricci flows are constructed under a different line of background Ricci flows. We manage to overcome this by tracing the gauge change carefully, and show that Ricci-harmonic map heat flow between $g_\b(t)$ and $g_{\hat \b}(t)$ can still be solved with stability estimates, whenever $\b,\hat\b$ are close. Equivalently, $g_\b(t)$ can be transformed into a Ricci-DeTurck flow with respect to $g_{\hat \b}(t)$ through a gauge change using Ricci-harmonic map heat flow. This will be achieved by careful approximation procedure and will be discussed in detail in Section~\ref{sec:stability-cts}.
}

The new construction of the Ricci flow smoothing is possibly different from the flow constructed by Gianniotis-Schulze \cite{GianniotisSchulze2018}, even in case of isolated singularities. The uniqueness and stability of Ricci flow starting from singular initial metric are still unclear in general. Restricting to the case when the flow is coming out of a cone with smooth link with $\mathrm{Rm}>1$, it was shown by Deruelle \cite{Deruelle2016} that the expanding soliton with positive curvature coming of that cone is unique. On the other hand, the Ricci flow smoothing produced by the new construction scheme can be shown to be an expander with positive curvature, using the result in \cite{ChanLeePeachey2024}. Consequently, two solutions will coincide with each other. 

In particular, thanks to the nature of new construction, this shows that Deruelle's expanders are stable with respect to $L^\infty$ perturbation of the (smooth) link. As a corollary of the stability, we obtain that the $C^0$-convergence of links implies the $C^{\infty}$-convergence of the expanding solitons.
\begin{thm}\label{thm:}
   Let $\{(\mathbb{S}^{n-1},h_i)\}_{i=1}^{\infty}$ and $(\mathbb{S}^{n-1},h_{\infty})$ be a sequence of smooth Riemannian manifolds. Suppose $\Rm(h_i)>1$, for all $i\in\mathbb N\cup\{\infty\}$, and 
   $$\|h_i- h_{\infty}\|_{C^0(\mathbb{S}^{n-1})}\to 0,\qquad\textnormal{as }i\to\infty.$$
   Let $\{(M_i,g_i,f_i,p_i)\}_{i=1}^{\infty}$ be the unique Ricci solitons coming out of the cone over $(\mathbb{S}^{n-1},h_i)$, constructed by Deruelle \cite{Deruelle2016}.
   Then there are smooth diffeomorphisms $\Phi_i$ on $\mathbb{S}^{n-1}$ such that for any $k\in\mathbb N$,
   \[\|\Phi_i^*g_i-g_{\infty}\|_{C^k(\mathbb{S}^{n-1})}\to 0,\qquad\textnormal{as }i\to\infty. \]
\end{thm}
We refer readers to Theorem~\ref{thm:stable-expander} for the quantitative version of the expander's stability.

\subsection*{Acknowledgment} The authors would like to thank Felix Schulze for fruitful discussion related to stability of Ricci flow.
P.-Y. Chan is  supported by the Yushan Young Fellow Program of the Ministry of Education (MOE), Taiwan (MOE-108-YSFMS-0004-012-P1), and by the NSTC grant 113-2115-M-007 -014 -MY2. Y. Lai is  supported by NSF grant DMS-2506832. M.-C. Lee is  supported by Hong Kong RGC grant (Early Career Scheme) of Hong Kong No. 24304222, No. 14300623 and  No. 14304225, NSFC grant No. 12222122 and an Asian Young Scientist Fellowship.

\section{Preliminaries: steady Ricci solitons} 

In this section, we collect some basic material of solitons.  A complete Riemannian metric $(M,g_0)$ is said to be a gradient Ricci soliton if there exists a smooth function $f\in C^\infty_{loc}(M)$ ( which is called a defining function) such that 
\begin{equation}
\Ric+\lambda g_0 =\nabla^2 f
\end{equation}
where $\lambda\in \mathbb{R}$ is a constant. A gradient Ricci soliton is said to be expanding, steady, and shrinking according to the sign  of $\lambda >, =,<0$ respectively. The three distinct types of Ricci solitons correspond to three different type of blow-up solutions of Ricci flow, for instance see \cite{Hamilton1995}. 

In this work, we are preliminarily interested in the steady case, i.e. $\lambda=0$. We might write the Ricci soliton equation as $2\Ric=\mathcal{L}_X g$ where $X=\nabla f$. It was proved by Zhang \cite{Zhang2009} that the vector field $X$ is a complete vector field on $M$. If we let $\Phi_t$ be the one-parameter of diffeomorphisms of $M^n$ generated by $-X$, then $g(t):=\Phi_t^* g_0$ solves the Ricci flow equation: 
\begin{equation}
\left\{
\begin{array}{ll}
\partial_t g(t)=-2\Ric(g(t));\\
g(0)=g_0
\end{array}
\right.
\end{equation}
 on $M\times (-\infty,+\infty)$. For simplicity, we will call $g(t)$ the soliton Ricci flow of $(M,g_0)$. In particular, it follows from the work of Chen \cite{Chen2009} that the scalar curvature satisfies $R(g(t))> 0$ for all $t\in\mathbb{R}$ unless $g_0$ is Ricci-flat.

On the other hand, it was observed by Hamilton \cite{Hamilton1995} that 
\begin{equation}
R(g_0)+|\nabla^{g_0} f|^2=\mathrm{const.}
\end{equation}
on $M$. Hence, we will assume by scaling that 
\begin{equation}
R(g_0)+|\nabla^{g_0} f|^2=1
\end{equation}
on $M$, when $g_0$ is not Ricci-flat.

\section{Ricci flow from spaces with isolated singularities}\label{sec:existence}

In this section, we will construct Ricci flow with curvature bounded from below which is coming out of manifolds with conic singularity. We will ultimately apply the smoothing iteratively to multiply warped product metrics with certain singularities. For manifolds with isolated conic singularity, we consider singularity of the following decay rate.
\begin{defn}\label{defn:deca}
Given $\sigma> 0$ and $L>0$, we say that a metric $g$ is $\sigma$-modeled by $C(X)$ at $p$ if there exists an open set $U_p$ containing $p$, $r_0>0$ and a smooth map $\phi: (0,r_0]\times X\to U_p\setminus \{p\}$ such that $\lim_{r\to 0}\phi(r,x)=p$ for all $x\in X$ and 
\begin{equation}
\sum_{j=0}^4 r^j \left|\nabla^{g_{c,X},j} (\phi^* g-g_{c,X}) \right|_{g_{c,X}} \leq r^\sigma
\end{equation}
as $r\to 0$, where $g_{c,X}:=dr^2+r^2 g_X$ is the cone-metric on $C(X)$.
\end{defn}

When $X=(\mathbb{S}^{n-1},h)$ has $\mathrm{Rm}(h)\geq 1$ and $O(r^\sigma)$ is further relaxed to $o(1)$ as $r\to 0$, it was studied by Gianniotis-Schulze  \cite{GianniotisSchulze2018} who showed that singularity of this form can be smoothed using Ricci flow. We strengthen the decay rate so that the Ricci flow smoothing has a uniform curvature lower bound. 

\begin{thm}\label{thm:RF-existence}
For $n\geq 3$, let $(Z^n,g_0)$ be a smooth compact manifold with isolated singularities $\{p_i\}_{i=1}^N$ and $\mathrm{Rm}(g_0)\geq -\e^2\geq-1$ outside $\{p_i\}_{i=1}^N$. Assume there is $\sigma>0$ such that at each $p_i$, $g_0$ is $\sigma$-modeled by $C(X_i)$ where $X_i=(\mathbb{S}^{n-1},h_i)$ satisfy $\mathrm{Rm}(h_i)\geq  1$ and $\mathrm{sec}(h_i)\not\equiv 1$. Then there exists a  smooth manifold $M$, constants $\a,\Lambda,S>0$ and a Ricci flow $g(t)$ on $M\times (0,S]$ such that 
\begin{enumerate}
    \item[(i)] $\mathrm{Rm}(g(t))\geq -\Lambda \e^2 $;
    \item[(ii)] $|\Rm(g(t))|\leq \a t^{-1}$;
    \item [(iii)]$\mathrm{inj}(g(t))\geq \sqrt{\a^{-1}t}$
    \item[(iv)] $\left(M,d_{g(t)}\right)\to (Z,d_{g_0})$ as $t\to0$ in Gromov-Hausdorff sense;
    \item[(v)] There exists a map $\Psi:Z\setminus\{p_i\}_{i=1}^N\to M$ diffeomorphism onto its image, such that $\Psi^*g(t)\to g_0$ in $C^\infty_{loc}(Z\setminus\{p_i\}_{i=1}^N)$ as $t\to 0$;
    \item[(vi)] If $\mathrm{Rm}(g_0)\geq 1$ outside $\{p_i\}_{i=1}^N$, then $\mathrm{Rm}(g(t))\geq 1$ for $t\in (0,S]$;
    \item[(vii)]If $\e=0$, then we might choose $\a,S$ such that it depends only on $\mathrm{Vol}(Z,g_0)$.
\end{enumerate}
\end{thm}

This is based on the idea in \cite{GianniotisSchulze2018}. For the purpose of curvature lower bound along the flow and discussion of stability, in the rest of this section, we will modify and rebuild their construction. 

\subsection{Expander with $\mathrm{Rm}\geq 0$}\label{sec:expander}

The idea of Gianniotis-Schulze is to first glue-in an  expanding Ricci soliton nearby the conic singularity. To do this, we use the existence result of Deruelle \cite{Deruelle2016}.
\begin{prop}\label{prop:Der}
Suppose $C(X)$ is a cone over a smooth manifold $X:=(\mathbb{S}^m,h)$ with $\mathrm{Rm}(h)> 1$, then there exists a unique Ricci flow $(N,g_N(t),p_N),t\in (0,+\infty)$ coming out of $C(X)$ in the pointed Gromov-Hausdorff sense such that $t\Ric(g_N)+g_N=\nabla^2 f$ for some smooth function $f$, on $N\times (0,+\infty)$.
\end{prop}
\begin{proof}
The existence of a unique expanding Ricci soliton $(N,g_{N,1},f_N,p_N)$ with $0<\mathrm{Rm}(g_{N,1})\leq \a$ for some $\a>0$, follow from the main result of \cite{Deruelle2016}. Let $g_N(t)$ be the induced Ricci flow so that $g_N(1)=g_{N,1}$.  Precisely, we let $\Phi_t,t\in (0,+\infty)$ be the family of diffeomorphism such that 
\begin{equation}
\left\{
\begin{array}{ll}
\partial_t \Phi_t=-t^{-1}\nabla^{g_{N,1}}f_N;\\
\Phi_1=\mathrm{Id}.
\end{array}
\right.
\end{equation}
Then $g_N(t):=t (\Phi_t)^*g_{N,1}$ defines a Ricci flow on $N\times (0,+\infty)$. Moreover it follows from Hamilton-Perelman distance distortion, for instances see \cite[Lemma 3.1]{SimonTopping2021}, that $d_{g_N(t)}$ converges to some distance function $d_0$ on $N$ as $t\to 0$ so that $(N,d_0,p_N)$ is isometric to $C(X)$ as metric space.
\end{proof}

Since $(N,d_0,p_N)$ is isometric to $C(X)$, $C(X)$ inherits a natural smooth structure from $N$ and thus we might without loss of generality assume $g_N(t)$ lives on $C(X)$ and $g_N(t)$ converges smoothly to $g_{c,X}:=dr^2+r^2g_X$ as $t\to 0$ away from $p_N\cong o_{tips}$. Furthermore, $g_N(t)$ satisfies 
\begin{enumerate}
\item[(a)] $\mathrm{Rm}(g_N(t))>0$;
\item[(b)] $\mathrm{AVR}(g_N(t))=v_0>0$;
\item[(c)] $|\Rm(g_N(t))|\leq \a t^{-1}$
\end{enumerate}
on  $N\times (0,+\infty)$, for some $\a,v_0>0$.

We will use $B_c(o,r_1)$ and $A_c(o,r_1,r_2)$ to denote $[0,r_1)\times X/\sim$ and $(r_1,r_2)\times X/\sim$ in $C(X)$. We will also use $B_c(x,r_1)$ to denote the ball with respect to cone metric $g_{c,X}$. Intuitively, we will use the radial function to construct cutoff function on $C(X)$. To make it more compatible with the Ricci flow, we use the smoothed radial from soliton equation. Let $f_N$ be the expanding soliton potential which satisfies $f_N(p_N)=0$ and given the induced Ricci flow $g_N(t)=t\Phi_t^* g_{N,1}$, we let $f(x,t):=t \Phi_t^* f_N(x)$. From solitons equation and Ricci flow estimates, the following is straight forward.
\begin{lma}\label{lma:dist-like}
There exists $C_n>0$ such that the function $$\eta(x,t)=2\sqrt{f(x,t)+C_n\a t}$$ satisfies  
\begin{enumerate}
\item[(i)] $ d_{g_N(t)}(x,p_N)\leq \eta(x,t)\leq  d_{g_N(t)}(x,p_N)+C_n\sqrt{\a t}$;
\item[(ii)] $|\nabla\eta|_{g_N}\leq 1$;
\item[(iii)] $|\nabla^{2,g_N} \eta|+|\partial_t \eta|\leq C_n \a \eta^{-1}$.
\end{enumerate}
\end{lma}

We also need the following estimates which is a consequence of smoothness of link $X$.
\begin{lma}\label{lma:reg-away}
There exists $\hat L,\hat T>0$ depending on the geometry of $X=(\mathbb{S}^{n-1},h)$, such that for all $r>0$ and $(x,t)\in B_c(o,r)^c\times [0,\hat Tr^2]$, we have 
\begin{equation}
\sum_{m=0}^5 r^m|\nabla^m \mathrm{Rm}(g_N)|\leq \hat L r^{-2}.
\end{equation} 
\end{lma}
\begin{proof}
Since $g_N(0)=g_{c,X}=dr^2+r^2g_X$ and $X:=(\mathbb{S}^{n-1},h)$ is smooth, the conclusion follows from \cite[Corollary 3.2]{Chen2009} and the modified Shi's type estimates \cite[Theorem 14.16]{ChowBookII}, see also \cite[Theorem 10.3]{Perelman2002}.
\end{proof}

\subsection{Existence of Ricci flow}

In this subsection, we will construct a Ricci flow coming out of smooth compact manifolds with isolated singularity. We start with constructing smooth approximation of $(Z,g_0)$ with curvature bounded from below in weak sense.

\begin{prop}\label{prop:approx}Under the assumption of Theorem~\ref{thm:RF-existence}, there exists a one parameter family of smooth compact manifold $(M_s,g_{s,0}),s\in (0,1]$ such that 
\begin{enumerate}
\item[(a)] $(M_s,g_{s,0})$ converges to $(Z,d_{g_0})$ in Gromov-Hausdorff sense, as $s\to 0$;
\item[(b)] There exists $r_1,C_1>0$ such that 
\begin{equation*}
\fint_{B_{g_{s,0}}(x,r)} \mathrm{Rm}_-(g_{s,0})\,d\mathrm{vol}_{g_{s,0}}\leq C_1r^{-2+\sigma},
\end{equation*}
for all $s\in (0,1]$, $x\in M_s$ and $r\in (0,r_1]$.
\end{enumerate}
Here $\mathrm{Rm}_-(g):=\inf\{s>0: \mathrm{Rm}(g)+s\cdot \mathrm{Id}>0\}$ denotes the negative part of the lowest eigenvalue of $\mathrm{Rm}$.
\end{prop}
\begin{proof}
Without loss of generality, we might assume $N=1$ since the gluing is performed nearby each isolated singularity only. That said there is only one isolated singularity $p\in Z$ which we assume modeled by $C(X)$. We will assume $\sigma<10^{-2}$ for convenience.

For notation convenience, we will assume $B_c(o,r_0)\subseteq Z$ for some small $r_0>0$. We also let $g_N(t)$ be the Ricci flow coming out of $C(X)$, obtained from Proposition~\ref{prop:Der}, with properties discussed in sub-section~\ref{sec:expander}. Let $\eta$ be the function from Lemma~\ref{lma:dist-like}. 

Fix a smooth non-increasing function $\phi:[0,+\infty)\to [0,1]$ such that $\phi\equiv 1$ on $[0,\frac32]$, vanishes outside $[0,2]$ and satisfies $|\phi''|\leq 10^3$, $|\phi'|^2\leq 10^3\phi$. We let $\phi_s(x):=\phi( \eta(x,s)/ s^{1/4})$ so that as $s\to 0$, $\phi_s$ is compactly supported on $B_c(o,r_0)$. Thus, we might define $(M_s,g_{s,0}):=(Z,g_{s,0})$ where 
\begin{equation}
g_{s,0}:=\phi_s g_N(s)+(1-\phi_s) g_0
\end{equation} 
is a smooth metric. The Gromov-Hausdorff convergence is immediate from the construction. It remains to show the curvature lower bound. 

\begin{claim}\label{claim:rough-bdd-beforeMorry}
There exists $C_2>0$ such that for $s$ sufficiently small,  the metric $g_{s,0}$ satisfies 
\begin{equation}
\mathrm{Rm}_-(g_{s,0})\leq \chi_{A_c(o,s^{1/4},3s^{1/4})}\cdot C_2 s^{-1/2+\sigma/4}+\chi_{Z\setminus B_c(o,3s^{1/4})}.
\end{equation}
\end{claim}
\begin{proof}[Proof of claim.]

Since $\mathrm{Rm}_-(g_0)\leq 1$ by assumption and $g_0$ can be regarded as a $L^\infty$ metric on $M$, it suffices to consider the following  cases: {\bf Case 1.} $x\in B_c(o,s^{1/4})$; {\bf Case 2.} $x\in A_c(o,s^{1/4},3s^{1/4})$; {\bf Case 3.} $x\in A_c(o,3s^{1/4},r_0)$. 
\medskip

If $x\in B_c(o,s^{1/4})$, then
\begin{equation}
\begin{split}
\eta(x,s)&\leq d_{g_N(s)}(x,o)+C_n\sqrt{\a s}\\
&\leq d_{0}(x,o)+C_n\sqrt{\a s}<\frac32 s^{1/4} 
\end{split}
\end{equation}
for $s$ sufficiently small. Here we have used \cite[Lemma 3.1]{SimonTopping2021} and Lemma~\ref{lma:dist-like}. In particular, $g_{s,0}= g_N(s)$ on $B_c(o,s^{1/4})$ so that $\mathrm{Rm}(g_{s,0})>0$. This proves the estimate on {\bf Case 1.} If $x\notin  \overline{B_c(o,3s^{1/4})}$, then for $s$ sufficiently small,
\begin{equation}
\begin{split}
\eta(x,s)&\geq d_{g_N(s)}(x,o)\geq d_{0}(x,o)-C_n'\sqrt{\a s}\\
&\geq 3s^{1/4}-C_n'\sqrt{\a s}>2s^{1/4}.
\end{split}
\end{equation}
Here we have used \cite[Lemma 3.1]{SimonTopping2021} and Lemma~\ref{lma:dist-like}. Thus, $g_{s,0}=g_0$ where we have $\mathrm{Rm}_-(g_{s,0})\leq 1$.  This proves the estimate on {\bf Case 3.}

\medskip
It remains to consider the transition region, $x\in A_c(o,s^{1/4},3s^{1/4})$.  We write
\begin{equation}
\begin{split}
g_{s,0}-g_{c,X}&=\phi_s \left( g_N(s)-g_{c,X}\right)+(1-\phi_s )  (g_0-g_{c,X})\\
&=h_1+h_2.
\end{split}
\end{equation}
We use $\tilde\nabla$ to denote the connection with respect to $g_{c,X}$. All norm below will be measured with respect to $g_{c,X}$, unless specified. In what follows, we will use $L_i$ to denote constant depending only on $h$. All estimates are understood on $A_c(o,s^{1/4},3s^{1/4})$ and for sufficiently small $s>0$.

By Lemma~\ref{lma:reg-away}, we have $|\Rm(g_N(t))|_{g_N(t)}\leq  L_0 s^{-1/2}$ on $A_c(o,s^{1/4},3s^{1/4})\times [0,\hat T s^{1/2}]$ and thus integrating in time yields 
\begin{equation}
|g_N(s)-g_{c,X}|\leq L_1 s^{1/2}.
\end{equation}
Similarly, we have $s^{-1/4}|\tilde\nabla g_N(s)|+|\tilde\nabla^2 g_N(s)|\leq  L_2$. In particular
\begin{equation}
\begin{split}
|\tilde\nabla h_1|&\leq s^{-1/4}|\phi'||\tilde\nabla \eta| |g_N(s)-g_{c,X}|+ \phi_s |\tilde\nabla g_N|\leq L_2 s^{1/4};\\[3mm]
|\tilde\nabla^2 h_1|&\leq s^{-1/2} |\phi''| |\tilde\nabla\eta| |g_N(s)-g_{c,X}|+s^{-1/4} |\phi'| |\tilde\nabla^2\eta| |g_N(s)-g_{c,X}|\\
&\quad +s^{-1/4} |\phi'| |\tilde\nabla \eta||\tilde\nabla g_N|+  \phi_s |\tilde\nabla^2 g_N(s)|\leq L_2
\end{split}
\end{equation}
using Lemma~\ref{lma:dist-like} and above estimates. 

\medskip

On the other hand, it follows from assumption that 
\begin{equation}
|g_0-g_{c,X}|+s^{1/4}|\tilde \nabla(g_0-g_{c,X})|+s^{1/2}|\tilde \nabla^2(g_0-g_{c,X})| \leq L_3s^{\sigma/4}.
\end{equation} 
A similar argument shows that  
\begin{equation}
|\tilde \nabla h_2|^2+|\tilde\nabla^2 h_2|\leq L_4 s^{-1/2+\sigma/4}.
\end{equation}

Using this and $\mathrm{Rm}(g_{c,X})\geq 0$, we see that 
\begin{equation}
\mathrm{Rm}_-(g_{s,0})\leq C_n \sum_{i=1}^2\left( |\tilde \nabla h_i|^2+  |\tilde \nabla^2 h_i|\right)\leq L_5s^{-1/2+\sigma/4}.
\end{equation}
This proves the claim.
\end{proof}

We now show (b) using claim~\ref{claim:rough-bdd-beforeMorry}. Let $x\in Z$ and $r>0$ small, 
\begin{equation}
\begin{split}
&\quad \int_{B_{g_{s,0}}(x,r)} \mathrm{Rm}_-(g_{s,0})\,d\mathrm{vol}_{g_{s,0}}\\
&\leq \mathrm{Vol}_{g_{s,0}}\left(B_{g_{s,0}}(x,r)\setminus B_c(o,3s^{1/4})\right)\\
&\quad +C_2 s^{-1/2+\sigma/4}\cdot \mathrm{Vol}_{g_{s,0}}\left(B_{g_{s,0}}(x,r)\cap A_c(o,s^{1/4},3s^{1/4})\right)\\
&=\mathbf{I}+\mathbf{II}.
\end{split}
\end{equation}

Clearly, 
\begin{equation}\label{eqn:esti-I}
\mathbf{I}\leq \mathrm{Vol}_{g_{s,0}}\left(B_{g_{s,0}}(x,r) \right).
\end{equation}

For $\mathbf{II}$, if $r\leq s^{1/4}$, then 
\begin{equation}\label{eqn:esti-II-1}
\mathbf{II}\leq C_2 r^{-2+\sigma} \cdot  \mathrm{Vol}_{g_{s,0}}\left(B_{g_{s,0}}(x,r) \right).
\end{equation}

In case $r>s^{1/4}$, we observe that since $g_{s,0}$ is uniformly bi-Lipschitz to $g_0$ outside $B_c(o,s^{1/4})$ and $g_{s,0}=g_N(s)$ on $B_c(o,s^{1/4})$, we have
\begin{equation}\label{eqn:esti-II-2}
\begin{split}
\mathbf{II}&\leq C_2 s^{-1/2+\sigma/4}\cdot \mathrm{Vol}_{g_{s,0}}\left(B_{g_{s,0}}(x,r)\cap A_c(o,s^{1/4},3s^{1/4})\right)\\
&\leq L_6 s^{{(-2+\sigma+n)}/4}\leq L_7 r^{-2+\sigma  }\cdot \mathrm{Vol}_{g_{s,0}}\left(B_{g_{s,0}}(x,r) \right).
\end{split}
\end{equation}

The conclusion (b) follows from combining \eqref{eqn:esti-I}, \eqref{eqn:esti-II-1} and \eqref{eqn:esti-II-2}.
\end{proof}

\medskip

Now we are ready to prove the existence of Ricci flow coming out of $(Z,g_0)$.
\begin{proof}[Proof of Theorem~\ref{thm:RF-existence}]
 It follows from the proof of \cite[Theorem 4.3]{CarronRose2021} and Proposition~\ref{prop:approx} that $(M_s,g_{s,0})\in\mathcal{K}_{IC1}(n,f,v)$ in the sense of \cite[Definition 1.2]{Lee2024}, uniformly for $f(t)=C_2(n,Z)\cdot  t^{\sigma/2}$ and some $v>0$, as $s\to 0$. Hence, \cite[Theorem 1.1]{Lee2024} applies to obtain a one parameter family of Ricci flow $g_s(t),t\in [0,S]$ starting from $g_{s,0}$  satisfying
\begin{enumerate}
\item[(a)] $|\Rm(g_s(t))|\leq \a t^{-1}$;
\item[(b)] $\mathrm{Vol}_{g_s(t)}(M_s)\geq v_0$;
\item[(c)] $\mathrm{inj}(g_s(t))\geq \sqrt{\a^{-1}t}$;
\item [(d)] $\mathrm{Rm}(g_s(t))\geq -\a t^{\sigma/2-1}$
\end{enumerate}
on $(0,S]$, for some uniform $v_0,\a,S>0$ independent of $s\to 0$. Here conclusion (c) is from \cite[Proposition 3.2 \& Lemma 3.1]{Lee2024}. This enables us to take sub-sequential limit of $g_{s_i}(t)$ as $s_i\to 0$, by Hamilton's compactness \cite{Hamilton1995} to obtain a Ricci flow $(M,g(t)),t\in (0,S]$. By Proposition~\ref{prop:approx}, \cite[Lemma 4.1]{Lee2024} and estimates above, $(M,g(t))$ converges to $(Z,g_0)$ as $t\to 0$ in the Gromov-Hausdorff sense. By \cite[Theorem 1.6]{DeruelleSchulzeSimon2022} or local regularity away from tips, $g(t)$ converges smoothly to some smooth metric $g(0)$ as $t\to 0$ away from $\{p_i\}_{i=1}^N$. Moreover, $d_{g_0}$ is distance-isometric to $d_{g(0)}$ via a Lipschitz map $\Psi:Z\to M$. By Myers–Steenrod, $\Psi$ is smooth outside $\{p_i\}_{i=1}^N$. We might assume $\Psi=\mathrm{Id}$ for notation convenience and $d_{g(t)}\to d_0=d_{g_0}$.

It remains to improve the curvature lower bound from mildly singular $O(t^{\sigma/2-1})$ to $O(1)$. We assume $N=1$ first.  Let $G(x,t;y,\tau)$ be the heat kernel to the operator $\partial_t-\Delta_{g(t)}-R_{g(t)}$ for $0<\tau<t<S$. By property (a), (c) and Proposition~\ref{prop:heat-kernel}, there is $L(n,\a)>0$ such that 
\begin{equation}\label{eqn:heat-kernel}
G(x,t;y,\tau)\leq \frac{L}{(t-\tau)^{n/2}}\cdot\exp\left(-\frac{d_{g(\tau)}^2(x,y)}{L(t-\tau)} \right)
\end{equation}
for all $x,y\in M$ and $0<\tau<t<S$.

By property (d) and \cite[Proposition 2.2]{BamlerCabezasWilking2019}, the function $\varphi:=e^{2\a \sigma^{-1} t^{\sigma/2}} \mathrm{Rm}_-(g(t))$ satisfies 
\begin{equation}\label{eqn:evo-Rm}
\left(\frac{\partial}{\partial t} -\Delta_{g(t)}-R_{g(t)}\right)\varphi\leq 0
\end{equation}
in the sense of barrier and hence maximum principle implies that for all $x\in M$, $t\in [\tau,S]$ and $\tau\to 0^+$,
\begin{equation}\label{eqn:heat-rep}
\begin{split}
\varphi(x,t)&\leq \int_M G(x,t;y,\tau)\,\varphi(y,\tau)\,d\mathrm{vol}_{y,g(\tau)}\\
&=\left(\int_{B_{g(\tau)}(p,\Lambda\sqrt{\tau})}+\int_{M\setminus B_{g(\tau)}(p,\Lambda\sqrt{\tau})}\right) G(x,t;y,\tau)\,\varphi(y,\tau)\,d\mathrm{vol}_{y,g(\tau)}\\
&=\mathbf{I}+\mathbf{II}
\end{split}
\end{equation}
where $\Lambda>0$ is a large constant to be chosen. 

We now use the fact that $g(\tau)\to g_0$ as $\tau\to 0$ smoothly outside $p\in M$ so that $\limsup_{\tau\to 0}\varphi(\tau)\leq \e^2$, to improve the estimate on $\mathbf{II}$. We need quantitative estimates on the speed of $\limsup_{\tau\to 0}\varphi$. 
\begin{claim}\label{claim:improved-decay}
If $x\in M$, and $r\in (0,1]$ are such that $g_0$ is smooth on $B_{g_0}(x, 2r)$, then for all $\ell>0$, there is $\hat T_\ell(n,\a)>0$ such that for all $t\in [0,\hat T_\ell r^2]$,
\begin{equation}
\varphi(x,t)\leq C_0(n,\a)\e^2+ t^\ell r^{-2\ell-2}.
\end{equation}
\end{claim}
\begin{proof}
Since $r\in (0,1]$, we might assume $r=1$ by considering $r^{-2}g(r^2t)$. Since $\varphi(0)\leq \e^2 r^2$ on $B_{g_0}(x,2)$, we might apply \cite[Theorem 1.1]{LeeTam2022} so that for all $\ell>0$, there is $\hat T_\ell>0$ such that for all $t\in [0,\hat T_\ell]$, 
\begin{equation}
\begin{split}
\varphi(x,t)&\leq t^\ell+\lim_{\tau\to 0^+}\int_{B_{g_0}(x,2)} G(x,t;y,\tau)\varphi(y,\tau) \,d\mathrm{vol}_{g(s)}\\
&\leq t^\ell+ C_0(n,\a)r^2\e^2
\end{split}
\end{equation}
where we have used \eqref{eqn:heat-kernel}. This proves the claim after rescaling.
\end{proof}

Since $\Ric(g(t))\geq -\a t^{\sigma/2-1}$, we have $g(t)\leq  2g(\tau)$ for all $0<\tau<t<S$ by integration, provided that we shrink $S$ if necessary. Therefore, if $d_{g(\tau)}(y,p)=r$, then 
\begin{equation}
\begin{split}
d_{g_0}(y,p)=\lim_{t\to 0} d_{g(t)}(y,p)> \frac12 r.
\end{split}
\end{equation}

It then follows from Claim~\ref{claim:improved-decay} that for all $\tau\in [0,4^{-1}\hat T_\ell r^2]$,
\begin{equation}
\varphi(y,\tau)\leq C_0\e^2 + 4^{\ell+1}\tau^\ell r^{-2\ell-2}.
\end{equation}
We fix some large $\ell$, for example we might choose $\ell=10^3n$.  Therefore, if we choose $\Lambda\geq 2 \hat T_\ell^{-1/2}$, then co-area formula and \eqref{eqn:heat-kernel} imply 
\begin{equation}
\begin{split}
\mathbf{II}&\leq \int_{\Lambda\sqrt{\tau}}^D  \left(\int_{\partial B_{g(\tau)}(p,r)}G(x,t;y,\tau)\varphi(y,\tau) dA_{g(\tau)} \right) \,dr\\
&\leq \int_{\Lambda\sqrt{\tau}}^D( C_0 \e^2+ 4^{\ell+1}\tau^\ell r^{-2\ell-2})\left(\int_{\partial B_{g(\tau)}(p,r)}G(x,t;y,\tau) dA_{g(\tau)} \right) \,dr\\
&\leq C_0\e^2\int_{M\setminus B_{g(\tau)}(p,\Lambda\sqrt{\tau})} G(x,t;y,\tau)\,d\mathrm{vol}_{y,g(\tau)}+\int_{\Lambda\sqrt{\tau}}^D\frac{L4^{\ell+1}\tau^\ell r^{-2\ell-2}}{(t-\tau)^{n/2}} \,dr.
\end{split}
\end{equation}
where $V_{g(\tau)}(r):=\mathrm{Vol}_{g(\tau)}\left(B_{g(\tau)}(p,r)\right)$. 

Since $g(\tau)\leq 2 g_0$ and $d_{g(\tau)}+C_n\sqrt{\a \tau}\geq d_0$, we have $V_{g(\tau)}(r)\leq L_1 r^n$ for $r\geq \sqrt{\tau}$. Using this, we see that 
for some $L_2(n,\a)>0$,
\begin{equation}\label{eqn:II-esti}
\limsup_{\tau\to 0^+}\mathbf{II}\leq L_2\e^2.
\end{equation} 

On the other hand by the rough bound from (d), \eqref{eqn:heat-kernel} and the volume upper bound,
\begin{equation}\label{eqn:I-esti}
\begin{split}
\limsup_{\tau\to 0}\mathbf{I}&\leq  \limsup_{\tau\to 0}\int_{B_{g(\tau)}(p,\Lambda\sqrt{\tau})}\frac{L}{(t-\tau)^{n/2}} \varphi(y,\tau)\,d\mathrm{Vol}_{g(\tau)}\\
&\leq \limsup_{\tau\to 0}\frac{L }{(t-\tau)^{n/2}} \frac\a{\tau^{1-\sigma/2}} L_1\tau^{n/2}=0.
\end{split}
\end{equation} 

By combining \eqref{eqn:I-esti}, \eqref{eqn:II-esti} and \eqref{eqn:heat-rep}, we deduce that 
\begin{equation}\label{eqn:res-Gr}
\varphi(x,t)\leq L_2\e^2
\end{equation}
for all $(x,t)\in M\times (0,S]$. This finishes the case when the initial metric $g_0$ has curvature bounded from below by $-\e^2$. 

We now consider the case of $\mathrm{Rm}(g_0)\geq 1$ outside $\{p\}$. We first observe that by letting $\e\to 0$, we have $\mathrm{Rm}(g(t))\geq 0$ for $t\in (0,S]$. It remains to improve it to uniformly positive. If  we define 
\begin{equation}
v(x,t):=\sup\{s: \mathrm{Rm}(g(x,t))-s\cdot\mathrm{Id}\geq  0\},
\end{equation}
then it satisfies $v\geq 0$ by $\mathrm{Rm}\geq 0$, and $(\partial_t-\Delta_{g(t)}) v\geq 0$ in the sense of barrier by \cite{Hamilton1986}. We follow the idea in the proof of \cite[Lemma 4.2]{ChanConlonLai2024}. Let $K(x,t;y,s)$ be the heat kernel for $\partial_t-\Delta_{g(t)}$. From smooth convergence away from  $p$, for any $\delta,\e>0$, there is $s_0>0$ such that $v\geq 1-\e$ on $B_{g_0}(p,\delta)^c\times (0,s_0)$. By the kernel representation, for all $t\in (0,S]$,
\begin{equation}
\begin{split}
v(x,t)&\geq \lim_{s\to 0^+}\int_{M\setminus B_{g_0}(p,\delta)} K(x,t;y,\tau) \,v(y,\tau)\,d\mathrm{vol}_{g(\tau)}\\
&\geq (1-\e)\cdot \left(1-\lim_{\tau\to 0^+}\int_{B_{g_0}(p,\delta)} K(x,t;y,\tau) \,d\mathrm{vol}_{g(\tau)}\right)
\end{split}
\end{equation}
where we have used $\int_M K(x,t;y,\tau)\,d\mathrm{vol}_{g(\tau)}=1$. Since $R_{g(t)}\geq 0$,  $K(x,t;y,\tau)\leq G(x,t;y,\tau)$ for all $0<\tau<t\leq S$ by maximum principle, it follows from \eqref{eqn:heat-kernel} that 
\begin{equation}
v(x,t)\geq (1-\e)(1-o(1))
\end{equation}
as $\delta\to 0$.
By letting $\e,\delta\to 0$, we see that $v\geq1$ on $M\times (0,S]$ and thus completes the proof when $N=1$. When $N>1$, we decompose $M$ into union of disjoint ball centred at $p_i$ and their complement, the estimate can be carried over using the same heat kernel estimate argument. 

To see that $\a,S$ depends only on the total volume when $\e=0$, it suffices to note that the Ricci flow converges back to $(Z,g_0)$ in the measured Gromov-Hausdorff sense by Colding's result \cite{Colding1997} so that we might apply existence of Ricci flow again from $(M,g(\delta))$ with $\mathrm{Rm}(g(\delta))\geq 0$ with $\mathrm{Vol}(g(\delta))\geq \frac12 \mathrm{Vol}(Z,g_0)$ for $\delta\to 0^+$. This completes the proof.
\end{proof}

\section{Stability of Ricci flows}\label{sec:stability}

In this section, we will discuss the stability of Ricci flow with non-negative curvature. We start with some preliminary on gauge fixing mechanism using Ricci-harmonic map heat flow and Ricci-DeTurck flow.

\subsection{Relation between Ricci flow and Ricci-DeTurck flow}
Given two family of metrics $\hat g(t),\tilde g(t)$ on a complete manifold $M\times [0,T]$. For simplicity, we assume both Ricci flows have (non-quantitative) bounded curvature up to $t=0$. This will be sufficient for our purpose in this work. We consider the short-time solution to the Ricci-harmonic map heat flow on $M\times [0,\hat T]$: 
\begin{equation}
\partial_t F=\Delta_{\hat g(t),\tilde g(t)}\, F.
\end{equation}
Its existence under bounded curvature is standard in the compact case, while the non-compact case was established by Chen-Zhu \cite{ChenZhu2006}.

We consider the case when $g(t)$ and $\tilde g(t)$ are both evolving under Ricci flow. If $F_t$ remain a diffeomorphism on $[0,\hat T]$, then the one parameter family of metric $g(t):=(F_t^{-1})^* \hat g(t)$ defines a solution to the Ricci-DeTurck flow with respect to $\tilde g(t)$ on $M\times [0,\hat T]$: 
\begin{equation}
\left\{
\begin{array}{ll}
\partial_t g_{ij}=-2R_{ij}+\nabla_i V_j+\nabla_j V_i;\\
V^k=g^{ij}\left(\Gamma_{ij}^k-\tilde \Gamma_{ij}^k \right).
\end{array}
\right.
\end{equation}

Locally if we write $g(t)=\tilde g(t)+h(t)$, then using also $\partial_t \tilde g(t)=-2\Ric(\tilde g(t))$, we see that
\begin{equation}\label{eqn:RDF}
\left\{
\begin{aligned}
\partial_t  h_{ij}&=  g^{pq} \tilde \nabla_p\tilde\nabla_q h_{ij}- g^{kl}  g_{ip}\tilde g^{pq} \tilde R_{jkql}- g^{kl}  g_{jp} \tilde g^{pq}\tilde R_{ikql}+2\tilde R_{ij}\\
&\quad +\frac12  g^{kl}  g^{pq}\big(\tilde \nabla_i h_{pk}\tilde \nabla_jh_{ql}+2\tilde \nabla_k h_{jp}\tilde \nabla_q h_{il}-2\tilde \nabla_kh_{jp}\tilde \nabla_l h_{iq}\\[1mm]
&\quad -2\tilde \nabla_jh_{pk}\tilde \nabla_l h_{iq}-2\tilde \nabla_ih_{pk}\tilde \nabla_lh_{qj}\big);\\
g(t)&=\tilde g(t)+h(t),
\end{aligned}
\right.
\end{equation}
 see the computation in \cite[Lemma 2.1]{Shi1989} for example. Here we use the convention $ \mathrm{sec}_{ij}=  R_{ijij}$ in orthogonal frame.

 Conversely given a solution $g(t),t\in (0,T]$ to Ricci-DeTurck flow \eqref{eqn:RDF}, we might consider the ODE: 
 \begin{equation}\label{eqn:RDF-back-RF}
 \left\{
 \begin{array}{ll}
 \partial_t \Psi_t(x)=-W(\Phi_t(x),t);\\[1mm]
 \Psi_T(x)=x
 \end{array}
 \right.
 \end{equation}
 where $W^k=g^{ij}\left(\Gamma_{ij}^k-\tilde\Gamma_{ij}^k\right)$ is a time dependent vector field, then $\hat g(t):=\Psi_t^*g(t)$ is also a solution to the Ricci flow with $\hat g(T)=g(T)$. If $g(t)$ is smooth at $t=0$, we might choose $\Psi_0=\mathrm{Id}$.

\subsection{Main result: weak stability of Ricci-DeTurck flow}\label{subsec:RDF}
  We want to discuss the weak stability of the solution to \eqref{eqn:RDF} when $\tilde g(t)$ satisfies:
\begin{enumerate}
\item[(a)] $|\Rm(\tilde g(t))|\leq \a t^{-1}$;
\item[(b)] $\mathrm{inj}(\tilde g(t))\geq \sqrt{\a^{-1}t}$;
\item[(c)] $\mathrm{Rm}(\tilde g(t))\geq -1$;
\end{enumerate}
on $M\times (0,T]$, for some $\a>0$. In case $\tilde g(t)$ is a gradient expanding Ricci soliton with $\mathrm{Rm}\geq 0$, the weak stability was previously studied by Deruelle-Lamm \cite{DeruelleLamm2017}, using the method developed by Koch-Lamm \cite{KochLamm2012}. We will prove the following more general weak stability of the Ricci flow with curvature lower bound and scaling invariant estimate. This in particular generalizes the previous works \cite{KochLamm2012,DeruelleLamm2017,Burkhardt2019} 
of the weak stability of Ricci flows. For convenience,  $g_0$ is said to be a (uniform) $L^\infty$ metric if it belongs to measurable section of $\mathrm{Sym}_2(T^*M)$ such that $\Lambda^{-1} g \leq g_0\leq \Lambda g$ for some $\Lambda>0$ and smooth metric on $M$. 

\begin{thm}\label{thm:weak-stable-RDF} 
Suppose $M$ is a complete manifold and $\tilde g(t),t\in [0,T]$ is a smooth Ricci flow  on $M$ such that $\tilde g(t)$ satisfies (a)-(c)  on $(0,T]$ for some $\a>0$. Then there exists $\bar\e_0(n,\a),\Lambda(n,\a)>0$ so that the following holds: if $g_0$ is a $L^\infty$ metric on $M$ such that $||g_{0}-\tilde g_0||\leq \e<\bar \e_0$, then it admits a smooth solution to \eqref{eqn:RDF} on $M\times (0,T\wedge 1]$ such that 
\begin{enumerate}
\item[(i)] $|g(t)-\tilde g(t)|\leq \Lambda \e $;
\item[(ii)] $g(t)\to g_0$ in $C^0_{loc}(\Omega)$ where  $\Omega$ is any compact subset in which $g_0$ is continuous.
\end{enumerate}
Furthermore, if $g(t)$ and $\hat g(t)$ are both solutions to \eqref{eqn:RDF} such that (a) holds, then 
\begin{equation}
    \sup_{M\times (0,T\wedge 1]}||g(t)-\hat g(t)||_{L^\infty(M,\tilde g(t))}\leq \Lambda \cdot \limsup_{s\to 0^+}||g(s)-\hat g(s)||_{L^\infty(M,\tilde g(s))}.
\end{equation}
In particular, if the initial metric $g_0$ is continuous, then the solution is unique within the class of solution satisfying (i) and (ii).
\end{thm}
\begin{rem}
    We assume $\tilde g(t)$ in Theorem~\ref{thm:weak-stable-RDF} to be smooth up to $t=0$ so that short-time existence of \eqref{eqn:RDF} from $L^\infty$ metrics is more transparent. Most of the analysis involved still holds even if $\tilde g(t)$ is only smooth for $t>0$.
\end{rem}

This is in spirit analogous to the stability proved by Simon \cite{Simon2002}. The main difference is the absence of quantitative curvature boundedness of $\tilde g(t)$ as $t\to 0$. In what follows, all norms will be computed using the reference Ricci flow $\tilde g(t)$ unless specified.

\subsection{Evolution equation}
We will need a more refined version of evolution inequality  of $h$. The precise leading part of the evolution will play a crucial role in estimation. 
\begin{lma}\label{lma:evo-h-L1}
For any $\gamma>0$, there is $\e_0(\gamma,n)>0$ such that if $g:=\tilde g+h_1$ and $\hat g:=\tilde g+h_2$ are  solutions to \eqref{eqn:RDF} with $|h_i|\leq \e_0$ and $\mathrm{sec}(\tilde g)\geq -1$, then  the function $v:=\sqrt{|h|^2+\sigma}$ for $\sigma\in (0,1)$ where $h:=h_1-h_2=g-\hat g$ satisfies
\begin{equation}
\begin{split}
(\partial_t -\Delta_{\tilde g})v 
&\leq  (1+\gamma)R_{\tilde g}v +C_nR_{\tilde g} |h|^2+C_n v+C_n |h| |\tilde \nabla g|^2+C_n  |\tilde \nabla h|\left(|\tilde\nabla g|+|\tilde\nabla \hat g| \right)\\
&\quad   + \tilde\nabla_k \left[\left( g^{kl}-\tilde g^{kl}\right) \tilde\nabla_l v\right]+\tilde \nabla_p\left(v^{-1} \langle   (g^{pq}-\hat g^{pq})\tilde\nabla_q \hat g ,h\rangle\right).
\end{split}
\end{equation}

Furthermore if $\tilde g=\hat g$, then we might choose $\e_1(n)>0$ such that if $|h_1|\leq \e_1$, then
\begin{equation}
\begin{split}
(\partial_t -\Delta_{\tilde g})v 
&\leq   R_{\tilde g}v +C_nR_{\tilde g} |h|^2+C_n v+C_n  |\tilde \nabla h|^2\\
&\quad   + \tilde\nabla_k \left[\left( g^{kl}-\tilde g^{kl}\right) \tilde\nabla_l v\right].
\end{split}
\end{equation}
\end{lma}
\begin{proof}
This follows directly from careful computation. Since the exact coefficient is important for application, we include it for readers' convenience. 

By \eqref{eqn:RDF} and Uhlenbeck trick, we might assume that $h:=g-\hat g$ satisfies 
\begin{equation}\label{eqn:RDF-2}
\left\{
\begin{aligned}
\partial_t  h_{ij}&=  \left(g^{pq} \tilde \nabla_p\tilde\nabla_q g_{ij}-\hat g^{pq} \tilde \nabla_p\tilde\nabla_q \hat g_{ij}\right)- g^{kl}  g_{ip}\tilde g^{pq} \tilde R_{jkql}- g^{kl}  g_{jp} \tilde g^{pq}\tilde R_{ikql}\\
&\quad  + \hat g^{kl} \hat  g_{ip}\tilde g^{pq} \tilde R_{jkql}+\hat  g^{kl}\hat   g_{jp} \tilde g^{pq}\tilde R_{ikql}+h_{ik}\tilde R^k_j +h_{jl}\tilde R^l_i\\
&\quad +g^{-1}*\hat g^{-1}*h*\hat g^{-1}*\tilde \nabla g*\tilde\nabla g+\hat g^{-1}*g^{-1}*\hat g^{-1}*h*\tilde \nabla g*\tilde\nabla g\\
&\quad +\hat g^{-1}*\hat g^{-1}*\tilde \nabla h*\tilde\nabla g+\hat g^{-1}*\hat g^{-1}*\tilde \nabla \hat g*\tilde\nabla h;\\[1mm]
\partial_t \tilde g_{ij}&=0.
\end{aligned}
\right.
\end{equation}

Therefore, 
\begin{equation}\label{eqn:h-evo}
\begin{split}
\partial_t |h|^2&=2 \langle \partial_t h,h\rangle \\
&\leq  2\langle g^{pq} \tilde \nabla_p\tilde\nabla_q g -\hat g^{pq} \tilde \nabla_p\tilde\nabla_q \hat g  ,h\rangle \\
&\quad +2\left(- g^{kl}  g_{ir}\tilde g^{rs} \tilde R_{jksl}- g^{kl}  g_{jr} \tilde g^{rs}\tilde R_{iksl}\right) h_{pq}\tilde g^{ip}\tilde g^{jq}\\
&\quad +2\left(\hat g^{kl} \hat  g_{ir}\tilde g^{rs} \tilde R_{jksl}+ \hat g^{kl}  \hat g_{jr} \tilde g^{rs}\tilde R_{iksl}\right) h_{pq}\tilde g^{ip}\tilde g^{jq}\\
&\quad +2h_{ik}\tilde R^k_j h_{pq}\tilde g^{ip}\tilde g^{jq}+2h_{jk}\tilde R^k_i h_{pq}\tilde g^{ip}\tilde g^{jq}\\
&\quad +C_n |h|^2 |\tilde \nabla g|^2+C_n |h||\tilde \nabla h|\left(|\tilde\nabla g|+|\tilde\nabla \hat g| \right).
\end{split}
\end{equation}

We  simplify the curvature term by diagonalizing $h_{ij}= \lambda_i\delta_{ij}$ with respect to $\tilde g_{ij}=\delta_{ij}$, then the Ricci term can be controlled by 
\begin{equation}
\begin{split}
&\quad 2h_{ik}\tilde R^k_j h_{pq}\tilde g^{ip}\tilde g^{jq}+2h_{jk}\tilde R^k_i h_{pq}\tilde g^{ip}\tilde g^{jq}\\
&=4\lambda_i^2 (\tilde R_{ii}+n-1)-4(n-1)|h|^2\\
&\leq 2R_{\tilde g}|h|^2+C_n |h|^2
\end{split}
\end{equation}
where we have used $\mathrm{sec}(\tilde g)\geq -1$. Here we have used the fact that $2\Ric(R)\leq \mathrm{scal}(R)$ for curvature tensor $R$ with $\mathrm{sec}(R)\geq 0$, see the proof of \cite[Lemma 3.2]{ChanLee2023}.

We now simplify the highest order term as follows: 
\begin{equation}
\begin{split}
&\quad 2 \langle g^{pq} \tilde \nabla_p\tilde\nabla_q g -\hat g^{pq} \tilde \nabla_p\tilde\nabla_q \hat g  ,h\rangle\\
 &= 2\langle g^{pq} \tilde \nabla_p\tilde\nabla_q h,h\rangle+2\langle (g^{pq}-\hat g^{pq}) \tilde \nabla_p\tilde\nabla_q \hat g,h\rangle\\
 &=g^{pq}\tilde\nabla_p \tilde\nabla_q |h|^2-2g^{pq}\langle \tilde\nabla_p h,\tilde\nabla_q h\rangle\\
 &\quad +2\langle  \tilde \nabla_p\left[ (g^{pq}-\hat g^{pq})\tilde\nabla_q \hat g\right],h\rangle-2\tilde \nabla_p\left (g^{pq}-\hat g^{pq}\right)\cdot \langle  \tilde\nabla_q \hat g,h\rangle\\
 &\leq g^{pq}\tilde\nabla_p \tilde\nabla_q |h|^2-2g^{pq}\langle \tilde\nabla_p h,\tilde\nabla_q h\rangle +2 \tilde \nabla_p\langle   (g^{pq}-\hat g^{pq})\tilde\nabla_q \hat g ,h\rangle+ C_n|h| |\tilde \nabla\hat g| |\tilde \nabla h|.
\end{split}
\end{equation}

Now we compute the evolution equation of $v$ accordingly using that of $|h|^2$. Since $v^2=|h|^2+\sigma$, we have
\begin{equation}
\begin{split}
(\partial_t -\Delta_{\tilde g})v 
&\leq  R_{\tilde g}v +C_n v+C_n |h| |\tilde \nabla g|^2+C_n  |\tilde \nabla h|\left(|\tilde\nabla g|+|\tilde\nabla \hat g| \right)\\
&\quad -v^{-1}g^{pq}\langle \tilde\nabla_p h,\tilde\nabla_q h\rangle  + \tilde\nabla_k \left[\left( g^{kl}-\tilde g^{kl}\right) \tilde\nabla_l v\right]- \tilde\nabla_k g^{kl}\cdot  \tilde\nabla_l v\\
&\quad +v^{-1}|\nabla v|_g^2+v^{-1} \tilde \nabla_p\langle   (g^{pq}-\hat g^{pq})\tilde\nabla_q \hat g ,h\rangle\\
&\quad  +v^{-1}\left(- g^{kl}  g_{ir}\tilde g^{rs} \tilde R_{jksl}- g^{kl}  g_{jr} \tilde g^{rs}\tilde R_{iksl}\right) h_{pq}\tilde g^{ip}\tilde g^{jq}\\
&\quad +v^{-1}\left(\hat g^{kl} \hat  g_{ir}\tilde g^{rs} \tilde R_{jksl}+ \hat g^{kl}  \hat g_{jr} \tilde g^{rs}\tilde R_{iksl}\right) h_{pq}\tilde g^{ip}\tilde g^{jq}.
\end{split}
\end{equation} 

Using \begin{equation}\label{eqn:grad-v}
\begin{split}
|v_i|= \left|\frac{1}{\sqrt{|h|^2+\sigma}} \langle h,\tilde\nabla_i h\rangle\right|\leq |\tilde \nabla_i h|,
\end{split}
\end{equation}
then we see that 
\begin{equation}
\begin{split}
&\quad v^{-1} \tilde \nabla_p\langle   (g^{pq}-\hat g^{pq})\tilde\nabla_q \hat g ,h\rangle\\
&\leq \tilde \nabla_p\left(v^{-1} \langle   (g^{pq}-\hat g^{pq})\tilde\nabla_q \hat g ,h\rangle\right)+ C_n|\tilde\nabla h||\tilde\nabla \hat g|
\end{split}
\end{equation}
so that 
\begin{equation}
\begin{split}
(\partial_t -\Delta_{\tilde g})v 
&\leq  R_{\tilde g}v +C_n v+C_n |h| |\tilde \nabla g|^2+C_n  |\tilde \nabla h|\left(|\tilde\nabla g|+|\tilde\nabla \hat g| \right)\\
&\quad   + \tilde\nabla_k \left[\left( g^{kl}-\tilde g^{kl}\right) \tilde\nabla_l v\right]+\tilde \nabla_p\left(v^{-1} \langle   (g^{pq}-\hat g^{pq})\tilde\nabla_q \hat g ,h\rangle\right)-v^{-1}Q(R)
\end{split}
\end{equation}
where $Q(R)$ denotes the remaining curvature terms, i.e.
\begin{eqnarray*}
    Q(R)&:=&\left(g^{kl}  g_{ir}\tilde g^{rs} \tilde R_{jksl}+g^{kl}  g_{jr} \tilde g^{rs}\tilde R_{iksl}\right) h_{pq}\tilde g^{ip}\tilde g^{jq}\\
    &&-\left(\hat g^{kl} \hat  g_{ir}\tilde g^{rs} \tilde R_{jksl}+ \hat g^{kl}  \hat g_{jr} \tilde g^{rs}\tilde R_{iksl}\right) h_{pq}\tilde g^{ip}\tilde g^{jq}.
\end{eqnarray*}

It remains to control $Q(R)$. By diagonalizing $h$ with respect to $\tilde g$, we have 
\begin{equation}
\begin{split}
Q(R)&=2\lambda_i\tilde R_{ikrl} \left(g^{kl}g_{ir}-\hat g^{kl}\hat g_{ir} \right)\\
&=2\lambda_i\tilde R_{ikrl} \left(-g^{kq}\hat g^{pl}h_{pq}g_{ir}+\hat g^{kl}h_{ir} \right)\\
&=2\lambda_i^2 \tilde R_{ikil}\hat g^{kl}-2\lambda_i \lambda_p \tilde R_{ikrl}g^{kp}\hat g^{pl} g_{ir}\\
&\geq \tilde R_{ikik}\left(\lambda_i^2 +\lambda_k^2-2\lambda_i \lambda_k\right)-C_n \e_0 |\mathrm{Rm}(\tilde g)||h|^2\\
&\geq -C_n|h|^2-C_n \e_0 |\mathrm{Rm}(\tilde g)||h|^2.
\end{split}
\end{equation}
It is also clear that $Q(R)\geq-C_n|h|^2 -C_n|\Rm(\tilde g)||h|^3$, if $\tilde g=\hat g$. The result then follows by combining all inequalities.
\end{proof}

\bigskip

\subsection{Some a-priori estimates}

 To overcome the curvature unboundedness, we need the following deep result of Petrunin \cite{Petrunin2008}.

\begin{prop}\label{prop:Petrunin} 
There exists $C_n>0$ such that the following holds: Suppose $(M^n,g)$ is a complete  manifold with $\mathrm{sec}(g)\geq -\sigma^2$ for some $\sigma>0$, then for all $r\in (0,\sigma^{-1})$ and $x\in M$,  we have 
\begin{equation}
\int_{B_g(x ,r)} |\Rm(g)|\,d\mathrm{vol}_g\leq C_n r^{n-2}.
\end{equation}
\end{prop}
\begin{proof}
For $r\in (0,\sigma^{-1})$, the assumption implies that the metric $\tilde g=r^{-2}g$ satisfies $\mathrm{sec}(\tilde g)\geq -\sigma^2 r^2\geq -1$ and $|\Rm(\tilde g)|\leq C_n\left(2n(n+1)+ \mathrm{scal}(\tilde g)\right )$ for some $C_n>0$.  We then apply \cite[Theorem 1.1]{Petrunin2008} and volume comparison to give 
\begin{equation}
\begin{split}
r^{2-n}\int_{B_{g}(x,r)} |\Rm(g)|\,d\mathrm{vol}_{g}&=\int_{B_{\tilde g}(x,1)} |\Rm(\tilde g)|\,d\mathrm{vol}_{\tilde g}\\
&\leq C_n +C_n\int_{B_{\tilde g}(x,1)}|\mathrm{scal}(\tilde g)|\,d\mathrm{vol}_{\tilde g}\leq C_n,
\end{split}
\end{equation}
for all $x\in M$ and $r\in (0,\sigma^{-1})$. This completes the proof.
\end{proof}

Using Proposition~\ref{prop:Petrunin}, we might obtain a sharper a-priori estimate of $\nabla h$ in space-time Morrey sense, provided $h$ remains close to $0$ along the flow.

\begin{lma}\label{lma:Morrey-grad}
There exists $\e_2(n)>0$ such that the following holds:  Let $(M,\tilde g(t))$ be a smooth solution to the Ricci flow (not necessarily complete) on $[0,T]$ and $x_0\in M$ such that for some $r>0$,
\begin{enumerate}
\item $B_{\tilde g(0)}(x_0,4r)\Subset M$;
\item $|\Rm(\tilde g(t))|\leq \a t^{-1}$ on $M\times (0,T]$;
\item $\mathrm{sec}(\tilde g(t))\geq -r^{-2}$ on $M\times [0,T]$.
\end{enumerate}
If $g(t)$ is a  smooth solution to the Ricci-DeTurck flow with respect to $\tilde g(t)$ on $M\times (0,T]$ such that
\begin{equation}
(1-\e)\tilde g(t)\leq g(t)\leq (1+\e)\tilde g(t)
\end{equation}
on $M\times (0,T]$, for some $0<\e<\e_2$. There exists $1>\tilde T(n,\a)>0$ such that for all $t\in (0,T\wedge \tilde Tr^2]$,
$$\int^t_0 \int_{B_{\tilde g(s)}(x_0,r)}  |\tilde\nabla g|^2 \,d\mathrm{vol}_{\tilde g(s)}ds\leq C_n \e^2r^n.$$
\end{lma}
\begin{proof}
We might assume $r=1$ by scaling.
We follow \cite[Lemma 7.1]{SimonToppingJDG} to construct cut-off function. Let $\phi$ be a smooth non-increasing function on $[0,+\infty)$ such that $\phi\equiv 1$ on $[0,\frac32]$, vanishes outside $[0,2]$ and satisfies $|\phi'|^2\leq 10^3\phi, \phi''\geq -10^3\phi$. Define $
\varphi(x,t):= e^{-10^3t}\phi\left({d_{\tilde g(t)}(x_0,x)+C_n\sqrt{\a t}} \right)$ so that 
\begin{equation}
\left(\frac{\partial}{\partial t}-\Delta_{\tilde g(t)}\right)\varphi\leq 0
\end{equation}
in the sense of barrier, and thus distribution sense \cite[Appendix]{MantegazzaMascellaniUraltsev}. Considering 
\begin{equation}
E(t)=\int_M |h|^2\varphi \,d\mathrm{vol}_{\tilde g(t)},
\end{equation}
where $h:=g-\tilde g$.

By \eqref{eqn:h-evo} with $\hat g=\tilde g$, we see that if $\e_1$ is small enough, then
\begin{equation}
\begin{split}
E'(t)&\leq C_n\e^2\int_M |\Rm(\tilde g)|\varphi \,d\mathrm{vol}_{\tilde g(t)} -C_n^{-1}\int_M |\tilde \nabla h|^2 \varphi\,d\mathrm{vol}_{\tilde g(t)}\\
&\quad +C_n\e^2\int_M  \frac{|\tilde \nabla \varphi|^2}{\varphi }\,d\mathrm{vol}_{\tilde g(t)}  .
\end{split}
\end{equation}

Using Proposition~\ref{prop:Petrunin}, we observe that 
\begin{equation}
\begin{split}
\int_M |\Rm(\tilde g)|\varphi \,d\mathrm{vol}_{\tilde g(t)} &\leq C_n.
\end{split}
\end{equation}

Therefore, if  $\tilde T(n,\a)$ is small enough, then $\varphi\geq \frac12$ on $B_{\tilde g(t)}(x_0,1)$  for all $t\in [0,\tilde T\wedge T]$ and thus volume comparison implies
\begin{equation}
\begin{split}
\int^t_{t'} \int_{B_{\tilde g(s)}(x_0,1)} |\tilde \nabla h|^2 \,d\mathrm{vol}_{\tilde g(s)} ds \leq C_n\e^2+E(t')\leq C_n'\e^2.
\end{split}
\end{equation}
for $0<t'<t$. Result follows by letting $t'\to0$.
\end{proof}

For the discussion of weak stability of Ricci-DeTurck flow,  we also need the following Morrey bound for difference between two solutions.
\begin{lma}\label{lma:improved-Morrey-grad}
There exists $\e_3(n)>0$ such that the following holds:  Let $(M,\tilde g(t))$ be a smooth solution to the Ricci flow (not necessarily complete) on $[0,T]$ and $x_0\in M$ such that for some $r>0$,
\begin{enumerate}
\item[(a)] $B_{\tilde g(0)}(x_0,8r)\Subset M$;
\item[(b)] $|\Rm(\tilde g(t))|\leq \a t^{-1}$ on $M\times (0,T]$;
\item[(c)] $\mathrm{sec}(\tilde g(t))\geq -r^{-2}$ on $M\times [0,T]$.
\end{enumerate}
If $g(t),\hat g(t)$ are two  smooth solutions to the Ricci-DeTurck flow with respect to $\tilde g(t)$ on $M\times [0,T]$ such that
\begin{enumerate}
    \item[(i)]$(1-\e)\hat g(t)\leq g(t)\leq (1+\e)\hat g(t)$;
    \item [(ii)]$(1-\e_3)\tilde g(t)\leq g(t),\hat g(t)\leq (1+\e_3)\tilde g(t)$
\end{enumerate}
on $M\times (0,T]$, for some $0<\e<\e_3$. There exists $\tilde T(n,\a)>0$ such that for all $t\in (0,T\wedge \tilde Tr^2]$,
$$\int^t_0 \int_{B_{\tilde g(s)}(x_0,r)}  |\tilde\nabla (g-\hat g)|^2 \,d\mathrm{vol}_{\tilde g(s)}ds\leq C_n \e^2r^n.$$
\end{lma}
\begin{proof}
We might assume $r=1$ by scaling. Again we let $\varphi$ be the cutoff function constructed in the proof of Lemma~\ref{lma:Morrey-grad}. Considering 
\begin{equation}
E(t)=\int_M |h|^2\varphi \,d\mathrm{vol}_{\tilde g(t)},
\end{equation}
where $h:=g-\hat g$.

By the computation in Lemma~\ref{lma:evo-h-L1}, we see that if $\e_3$ is small enough, then
\begin{equation}
\begin{split}
E'(t)&\leq C_n\e^2\int_M |\Rm(\tilde g)|\varphi \,d\mathrm{vol}_{\tilde g(t)}+C_n \e^2\int_M (|\tilde \nabla g|^2+|\tilde \nabla \hat g|^2) \varphi\,d\mathrm{vol}_{\tilde g(t)}\\
&\quad -C_n^{-1}\int_M |\tilde \nabla h|^2 \varphi\,d\mathrm{vol}_{\tilde g(t)}+C_n\e^2 \int_M \frac{|\tilde\nabla\varphi|^2}{\varphi} d\mathrm{vol}_{\tilde g(t)}
\end{split}
\end{equation}

As in the proof of Lemma~\ref{lma:Morrey-grad},  we use Proposition~\ref{prop:Petrunin} on the curvature term so that for all $0<t'<t$,
\begin{equation}
\begin{split}
&\quad \int^t_s \int_{B_{\tilde g(s)}(x_0,1)} |\tilde \nabla h|^2 \varphi\,d\mathrm{vol}_{\tilde g(s)} ds\\ 
&\leq C_n\e^2 +C_n\e^2\int^t_0 \int_{B_{\tilde g(s)}(x_0,2)}(|\tilde \nabla g|^2+|\tilde \nabla \hat g|^2) \,d\mathrm{vol}_{\tilde g(s)} ds.
\end{split}
\end{equation}
Result follows by combining with estimates from Lemma~\ref{lma:Morrey-grad} and let $t'\to 0$.
\end{proof}

From the rough version of evolution equation of $h$, i.e. \eqref{eqn:RDF}, we also have a scaling invariant smoothing estimate of $h$, provided that $h$ remains small along the evolution.
\begin{lma}\label{lma:RDF-higher-ord} 
There exists $\e_4(n)>0$ such that the following holds:  Let $(N,\tilde g(t))$ be a smooth solution to the Ricci flow (not necessarily complete) on $[0,T]$ and $x_0\in N$ such that 
\begin{enumerate}
\item $B_{\tilde g(0)}(x_0,4r)\Subset N$ for some $r>0$;
\item $|\Rm(\tilde g(t))|\leq \a t^{-1}$ on $N\times (0,T]$. 
\end{enumerate}
If $ g(t)$ and $\hat g(t)$ are two smooth solutions to the Ricci-DeTurck flow with respect to $\tilde g(t)$ on $N^m\times [0,T]$ so that 
\begin{enumerate}
    \item $|g(t)-\hat g(t)|\leq \e$;
    \item $|g(t)-\tilde g(t)|\leq \e_4$
\end{enumerate}
on $N\times [0,T]$ for some $\e<\e_4$. Then for all $k\in\mathbb{N}$, there exists $\hat L_k(n,\a,k)$, $\hat S_k(n,\a,k)>0$ such that for all $(x,t)\in B_{\tilde g(0)}(x_0,r)\times (0,T\wedge  \hat S_kr^2]$,
\begin{equation*}
|\tilde\nabla^k  (g-\hat g)|(x,t)\leq \hat L_k \e t^{-k/2}.
\end{equation*}
\end{lma}
\begin{proof}
This follows from a straightforward modification of \cite[Lemma 2.6]{Lee2025} using \eqref{eqn:RDF-2}, with extra attention to the smallness of $g-\hat g$.
\end{proof}

\subsection{Heat kernel estimates}

We next need the heat kernel estimate for the operator $\partial_t-\Delta_{\tilde g(t)}-R_{\tilde g(t)}$, from the work  \cite{BamlerCabezasWilking2019}.

\begin{prop}\label{prop:heat-kernel}
Suppose $\tilde g(t)$ is a complete Ricci flow on $M\times (0,T]$ such that
\begin{enumerate}
\item[(a)] $|\Rm(\tilde g(t))|\leq \a t^{-1}$;
\item[(b)] $\mathrm{inj}(\tilde g(t))\geq \sqrt{\a^{-1}t}$
\end{enumerate}
for some $\a>0$, then there is $C_1(n,\a)>0$ such that the heat kernel $G(x,t;y,s)$ with respect to operator $\partial_t-\Delta_{\tilde g(t)}-R_{\tilde g(t)}$ satisfies 
\begin{equation}
G(x,t;y,s)\leq \frac{C_1}{(t-s)^{n/2}}\cdot \exp\left(-\frac{d_{\tilde g(s)}(x,y)^2}{C_1(t-s)} \right)
\end{equation}
for $x,y\in M$ and $0< s<t\leq T$. If in addition $\Ric(\tilde g(t))\geq -(n-1)$, then we might assume
\begin{equation}
|\tilde\nabla_{y,s} G(x,t;y,s)|
\leq \frac{C_1}{(t-s)^{(n+1)/2}}\cdot\exp\left(-\frac{d_{\tilde g(s)}(x,y)^2}{C_1(t-s)} \right).
\end{equation}
for $0< s<t\leq T\wedge (n-1)^{-1}\log 2$ and $x,y\in M$.
\end{prop}
\begin{proof}
The heat kernel estimate follows from \cite[Proposition 3.1]{BamlerCabezasWilking2019} when $2s\leq t$. When $t/2<s\leq t$, this follows directly from standard heat kernel estimate after rescaling, for instance see \cite{ChauTamYu2011}. 

It remains to consider the gradient estimate under Ricci lower bound. When $s\in [t/2,t]$, this follows from standard parabolic theory since the curvature is scaling invariant. It suffices to consider $s\in [0,t/2]$. We consider the backward equations: $u(z,\tau)=G(x_0,t_0;z,\frac12 t_0-\tau),\tau\in [0,\frac12 t_0]$ for fixed $(x_0,t_0)\in M\times (0,T]$ and $h(\tau)=\tilde g(\frac12 t_0-\tau)$. Using the backward Ricci flow equation and Bochner formula, we have
\begin{equation}\label{eqn:evo-u-GREEN}
\left\{
\begin{array}{ll}
(\partial_\tau -\Delta_{h(\tau)})u^2=-2|\nabla u|^2;\\[2mm]
(\partial_\tau -\Delta_{h(\tau)}) |\nabla u|^2\leq -2|\nabla^2 u|^2 +4(n-1)|\nabla u|^2
\end{array}
\right.
\end{equation}

Now we construct a cutoff function nearby $y_0$ as follows. Let $\phi$ be a smooth non-increasing function on $[0,+\infty)$ such that $\phi=1$ on $[0,1]$, vanishes outside $[0,2]$ and satisfies $\phi''\geq -10^3\phi,|\phi'|^2\leq 10^3\phi$. For $r> 0$ and a given $y\in M$, we define $\Phi(z,\tau):= \phi(e^{(n-1)\tau}r^{-1}d_{h(\tau)}(z,y))$ so that Laplacian comparison gives
\begin{equation}
-\Delta_{h}\Phi=-\frac{e^{(n-1)\tau}\phi'}{r}\Delta_h d_h(z,y_0)-\frac{e^{2(n-1)\tau}\phi''}{r^2}\leq C_nr^{-2}e^{2(n-1)\tau}
\end{equation}
in the sense of barrier. On the other hand, 
\begin{equation}
\partial_\tau \left(e^{(n-1)\tau}d_{h(\tau)}(y,z_0)\right)\geq 0
\end{equation}
using $\partial_\tau h=2\Ric(h)\geq -2(n-1)h$. Thus, we have 
\begin{equation}\label{eqn:cutoff}
\left\{
\begin{array}{ll}
\left(\partial_\tau -\Delta_{h(\tau)}\right)\Phi\leq  C_nr^{-2}e^{2(n-1)\tau};\\[2mm]
|\nabla \Phi|^2\leq C_nr^{-2}e^{2(n-1)\tau}\Phi
\end{array}
\right.
\end{equation}
in the sense of barrier. 

We let $\Omega:=B_{\tilde g(t_0)}(y,4r)$. Using $\partial_t \tilde g=-2\Ric(\tilde g)\leq 2(n-1)\tilde g$, we see that for $z\in \partial \Omega$ and $t\in [0,t_0/2]$,
\begin{equation}\label{eqn:boundary-dist}
d_{\tilde g(t)}(y,z)\geq e^{-(n-1)(t_0-t)}d_{\tilde g(t_0)}(y,z)\geq 4e^{-(n-1)t_0} r\geq 2r
\end{equation}
if $t_0 \leq (n-1)^{-1} \log 2$. In particular, $\Phi=0$ on $\partial\Omega\times [0,t_0/2]$.

\bigskip

Now we are ready to apply the maximum principle on $\Omega\times [0,T]$. Consider the test function $F(z,\tau):=\Phi \tau |\nabla u|^2+ \Lambda u^2$ which satisfies 
\begin{equation}
\begin{split}
&\quad \left(\partial_\tau -\Delta_{h(\tau)}\right) F\\
&\leq -2\Lambda |\nabla u|^2+\Phi |\nabla u|^2+\Phi \tau\left(-2|\nabla^2 u|^2 +4(n-1)|\nabla u|^2 \right)\\
&\quad -2\tau\langle \nabla\Phi,\nabla |\nabla u|^2\rangle +\tau|\nabla u|^2 \cdot \left(\partial_\tau -\Delta_{h(\tau)}\right)\Phi \\
&\leq  \left(-2\Lambda+1+2(n-1)t_0+C_n t_0 r^{-2}e^{(n-1)t_0} \right)|\nabla u|^2 
\end{split}
\end{equation}
on $[0,t_0/2]$, where we have used Cauchy inequality, \eqref{eqn:evo-u-GREEN} and \eqref{eqn:cutoff}. In particular, if we choose $\Lambda=C_ne^{n-1} r^{-2}t_0+4(n-1)t_0+1$, then 
\begin{equation}
\left(\partial_\tau -\Delta_{h(\tau)}\right) F<0.
\end{equation}

Therefore if we insist $t_0\leq  (n-1)^{-1}\log 2$, then maximum principle implies that on $[0,t_0/2]$, $F$ either attains its maximum on $\partial \Omega\times [0,t_0/2]$ where $\Phi=0$ thanks to \eqref{eqn:boundary-dist}, or $\Omega\times \{0\}$. By evaluating at $z=y$, we conclude that for all $\tau\in [0,t_0/2]$ where  $t_0\leq   (n-1)^{-1}\log 2$, 
\begin{equation}
\begin{split}
\sqrt{\tau}|\nabla u|(y,\tau)&\leq \sqrt{\Lambda} \left( \sup_{\Omega\times \{0\}} u+ \sup_{\partial\Omega\times [0,t_0/2]} u \right).
\end{split}
\end{equation}

In particular if we take $r=1$, then we have 
\begin{equation} 
|\tilde\nabla_{y,s} G(x_0,t_0;y,s)| \leq \frac{C_2}{t_0^{(n+1)/2}}
\end{equation} 
for $s\in [0,t_0/2]$, provided that $t_0\leq (n-1)^{-1}\log 2$. This also shows the case when $d_{\tilde g(0)}(x,y)\leq \sqrt{t_0}$ thanks to Ricci lower bound. 

We now choose $r$ suitably to draw conclusion when  $d_{\tilde g(0)}(x,y)\geq2 C_3\sqrt{t_0}$ for large $C_3(n,\a)>0$. By distance distortion \cite[Lemma 3.1]{SimonTopping2021}, $d_{\tilde g(s)}(x,y)\geq C_3\sqrt{t_0}$ for $s\in [0,t_0/2]$. In this case, we choose $r=10^{-3} C_3\sqrt{t_0}$. If $(z,\tau)\in \partial \Omega\times [0,t_0/2]$, then
\begin{equation}
d_{\tilde g(s)}(x_0,y)^2\leq 4d_{\tilde g(s)}(x_0,z)^2+4d_{\tilde g(s)}(y,z)^2\leq C_4 t_0+4d_{\tilde g(s)}(x_0,z)^2
\end{equation}
using distance distortion again. Thus, 
\begin{equation}
\begin{split}
u(z,\tau)&= G(x_0,t_0;z,s)\leq  \frac{C}{t_0^{n/2}}\cdot\exp\left( -\frac{d^2_{\tilde g(s)}(x_0,y)}{Ct_0} \right).
\end{split}
\end{equation} 

The upper bound when $\tau=0$ is similar. 
\end{proof}

\subsection{Proof of weak stability}

In this sub-section, we will prove Theorem~\ref{thm:weak-stable-RDF}. For later purpose, we let $\gamma$ be a fixed constant to be specified and denotes $\e_5(n):=\min\{\e_1,\e_2,\e_3,\e_4\}$ where $\e_i$ are the constants from  Lemma~\ref{lma:evo-h-L1} Lemma~\ref{lma:Morrey-grad}, Lemma~\ref{lma:improved-Morrey-grad} and Lemma~\ref{lma:RDF-higher-ord}, respectively. We first use the kernel representation to give some estimate to the Ricci-DeTurck flows. We give a localized form which might be useful in the future.
\begin{lma}\label{lma:kernel-repres}
Suppose  $\Phi$ is a smooth function compactly supported on $\Omega\times [0,T]$ where $\Omega\Subset M$. Let $g(t),\hat g(t)$ are two smooth solutions to \eqref{eqn:RDF} on $(0,T]$ so that 
\begin{equation}
     |g-\tilde g|,|\hat g-\tilde g|\leq \e_5
\end{equation}
and $\mathrm{sec}(\tilde g(t))\geq -1$  on $\Omega\times (0,T]$, then for $(x_0,t_0)\in \Omega\times (0,T\wedge 1]$,
\begin{equation*}
\begin{split}
&\quad (\Phi |h|)(x_0,t_0)\\
&\leq C_n\cdot \limsup_{t'\to 0}\int_M G(y,t') (\Phi |h|)(y,t') \,d\mathrm{vol}_{\tilde g(t')}+C_n\gamma\int^{t_0}_0 \int_M  G \Phi |h||\Rm(\tilde g)|\; d\mathrm{vol}_{\tilde g(s),y}ds\\
&\quad + C_n \int^{t_0}_0 \int_M  G |h|\cdot (\partial_s \Phi)_++( |h||\tilde\nabla G|+G|\tilde\nabla h|)\cdot |\tilde\nabla \Phi| \; d\mathrm{vol}_{\tilde g(s),y}ds \\
&\quad + C_n\int^{t_0}_0 \int_M   |h| \left( \Phi |\tilde\nabla G|+ |\tilde\nabla\Phi| G\right)\left(| \tilde\nabla h|+|\tilde \nabla \hat g| \right)\; d\mathrm{vol}_{\tilde g(s),y}ds\\
&\quad +C_n\int^{t_0}_0\int_M  G\Phi\cdot \left[|\Rm(\tilde g)| |h|^2+ |h| |\tilde \nabla g|^2+  |\tilde \nabla h|\left(|\tilde\nabla g|+|\tilde\nabla \hat g| \right)\right]
\; d\mathrm{vol}_{\tilde g(s),y}ds.
\end{split}
\end{equation*}
where $h:=g-\hat g$ and $G(y,s):=G(x_0,t_0;y,s)$ is the heat kernel of $\partial_t-\Delta_{\tilde g(t)}-R_{\tilde g(t)}$. Furthermore, $\gamma$ can be chosen to be $0$ if $\hat g(t)=\tilde g(t)$.
\end{lma}
\begin{proof}
Let $v:=\sqrt{|h|^2+\sigma}$ be the function considered in Lemma~\ref{lma:evo-h-L1} where $\tilde g=\hat g$, and $\hat  v=e^{-C_1t}v$. Then by the representation formula, 
\begin{equation}
\begin{split}
(\Phi \hat v)(x_0,t_0)&=\lim_{s\to t_0^-} \int_M G(x_0,t_0;y,s) (\Phi\hat  v)(y,s) \,d\mathrm{vol}_{\tilde g(s),y}\\
&\leq \limsup_{t'\to 0} \int_M G(y,t') (\Phi \hat v)(y,t') \,d\mathrm{vol}_{\tilde g(t'),y}\\
&\quad +\int^{t_0}_0 \int_M \partial_sG\cdot  \Phi \hat v+G\Phi \cdot \partial_s \hat v-G\Phi \hat v R_{\tilde g} \,\,d\mathrm{vol}_{\tilde g(s),y}ds\\
&\quad +\int^{t_0}_0 \int_M G \hat v\cdot  \partial_s \Phi \; d\mathrm{vol}_{\tilde g(s),y}ds:=\mathbf{I}+\mathbf{II}+\mathbf{III}
\end{split}
\end{equation} 
where we use $G$ to denote $G(x_0,t_0;\cdot,\cdot)$ for notation convenience.

It suffices to estimate $\mathbf{II}$. By the heat kernel equation and Lemma~\ref{lma:evo-h-L1}, we have 
\begin{equation}
\begin{split}
\mathbf{II}&=\int^{t_0}_0 \int_M -\Delta_{\tilde g} G \cdot \hat v\Phi + G\Phi \cdot \Delta_{\tilde g}\hat v +\gamma G\Phi \hat v R_{\tilde g} \; d\mathrm{vol}_{\tilde g(s),y}ds\\
&\quad +C_n\int^{t_0}_0\int_M  G\Phi\cdot  \left[|\Rm(\tilde g)| |h|^2+C_n |h| |\tilde \nabla g|^2+C_n  |\tilde \nabla h|\left(|\tilde\nabla g|+|\tilde\nabla \hat g| \right)\right] 
\; d\mathrm{vol}_{\tilde g(s),y}ds\\
&\quad   +\int^{t_0}_0\int_M \left[\tilde\nabla_k \left[\left( g^{kl}-\tilde g^{kl}\right) \tilde\nabla_l v\right]+\tilde \nabla_p\left(v^{-1} \langle   (g^{pq}-\hat g^{pq})\tilde\nabla_q \hat g ,h\rangle\right)\right]\cdot G\Phi\; d\mathrm{vol}_{\tilde g(s),y}ds\\
&=\mathbf{IV}+\mathbf{V}+\mathbf{VI}.
\end{split}
\end{equation}

By Stokes' Theorem,  we have
\begin{equation}
\begin{split}
\mathbf{IV}&=\int^{t_0}_0 \int_M \langle \hat v\tilde\nabla G-G\tilde\nabla \hat v,\tilde\nabla\Phi\rangle+\gamma G\Phi \hat v R_{\tilde g} \; d\mathrm{vol}_{\tilde g(s),y}ds\\
&\leq \int^{t_0}_0 \int_M (\hat v|\tilde\nabla G|+G|\tilde\nabla \hat v|)\cdot |\tilde\nabla \Phi| +\gamma G\Phi \hat v R_{\tilde g} \; d\mathrm{vol}_{\tilde g(s),y}ds.
\end{split}
\end{equation}

Similarly, we use stoke Theorem and \eqref{eqn:grad-v} to deduce  
\begin{equation}
\begin{split}
\mathbf{VI}&=\int^{t_0}_0 \int_M (G\tilde\nabla_k \Phi+\Phi\tilde\nabla_kG )\cdot  (\tilde g^{kl}-g^{kl}) \tilde\nabla_l \hat v\cdot d\mathrm{vol}_{\tilde g(s),y}ds\\
&\quad -\int^{t_0}_0 \int_M (G\tilde\nabla_p \Phi+\Phi\tilde\nabla_pG )\cdot  \left(v^{-1} \langle   (g^{pq}-\hat g^{pq})\tilde\nabla_q \hat g ,h\rangle\right) \; d\mathrm{vol}_{\tilde g(s),y}ds\\
&\leq C_n\int^{t_0}_0 \int_M   |h| \left( \Phi |\tilde\nabla G|+ |\tilde\nabla\Phi| G\right)\left(| \tilde\nabla h|+|\tilde \nabla \hat g| \right)\; d\mathrm{vol}_{\tilde g(s),y}ds.
\end{split}
\end{equation}

Since $1\leq e^{C_1t}\leq C_n$ for $t\leq 1$, by letting $\sigma\to 0$, we prove the Lemma.
\end{proof}

We now prove the global stability of Ricci-DeTurck flow. 
\begin{proof}[Proof of Theorem~\ref{thm:weak-stable-RDF}]

We first prove the uniform existence. Let $\Lambda$ be a large constant to be chosen. In the following, we will use $C_i$ to denote any constants depending only on $n,\a$. We first prove quantitative existence.

We first consider the case when $M$ is compact. By the work of Simon \cite{Simon2002} (see also \cite{KochLamm2012}), it admits a short-time solution to \eqref{eqn:RDF} on $M\times (0,S]$ for some $S\in (0,T\wedge1]$ such that $\delta:=||h||_{L^\infty(M\times [0,S])}<\e_5$ on $[0,S]$ where $h:=g(t)-\tilde g(t)$. We assume $S$ is the maximal so that it holds. We want to show that $S$ is uniformly bounded from below and $\delta\leq \Lambda \bar\e_0$ for some uniform $\Lambda>0$, if $\bar\e_0$ is sufficiently small. By Lemma~\ref{lma:kernel-repres} with $\Phi\equiv 1$ and $\hat g\equiv \tilde g$, for all $(x,t)\in M\times (0,S]$ we have 
\begin{equation}
\begin{split}
|h|(x,t)&\leq  \limsup_{t'\to0}\int_M G(x,t;y,t') |h|(y,t') d\mathrm{vol}_{\tilde g(t')}\\
&\quad +C_n \int^{t}_{t/2} \int_M  \left(|\tilde\nabla G||h||\tilde\nabla h|+G |h|^2 |\Rm(\tilde g)|+ G  |\tilde\nabla h|^2\right)\; d\mathrm{vol}_{\tilde g(s),y}ds\\
&\quad +C_n \int^{t/2}_{0} \int_M  \left(|\tilde\nabla G||h||\tilde\nabla h|+G |h|^2 |\Rm(\tilde g)|+ G  |\tilde\nabla h|^2\right)\; d\mathrm{vol}_{\tilde g(s),y}ds\\
&=\mathbf{I}+\mathbf{II}+\mathbf{III}.
\end{split}
\end{equation}

It follows from Proposition~\ref{prop:heat-kernel} that $\mathbf{I}\leq C_1\e$ while 
$
\mathbf{II}\leq C_2 \delta^2$ using curvature estimate of $\tilde g(t)$, Lemma~\ref{lma:RDF-higher-ord} and Proposition~\ref{prop:heat-kernel}. For $\mathbf{III}$, we decompose it into $\mathbf{III}_1+\mathbf{III}_2+\mathbf{III}_3$ and consider them one by one.

For $\mathbf{III}_3$, we use Lemma~\ref{lma:Morrey-grad} with $r= \sqrt{t}$ and covering argument from Ricci lower bound so that
\begin{equation}\label{eqn:Morrey-spt}
\int^{t}_0\int_{B_{\tilde g(s)}(x,\sqrt{t})} |\tilde\nabla h|^2 d\mathrm{vol}_{\tilde g(s)}ds\leq C_3 \delta^2t^{n/2}
\end{equation}
for all $(x,t)\in M\times (0,S]$. Using Proposition~\ref{prop:heat-kernel}, co-area formula and Stoke Theorem, for $s\in [0,t/2]$,
\begin{equation}
\begin{split}
\int_M G|\tilde \nabla h|^2 \,d\mathrm{vol}_{\tilde g(s)} 
&\leq \int^\infty_0 \frac{C_4r}{t^{n/2+1}}\cdot\exp\left(-\frac{r^2}{C_4t} \right)\left(\int_{B_{\tilde g(s)}(x,r)}|\tilde\nabla h|^2\right) \,dr\\
&\leq \sum_{k=1}^n \frac{C_5k}{t^{n/2}}\cdot\exp\left(-\frac{k^2}{C_5} \right)\left(\int_{B_{\tilde g(s)}(x,k\sqrt{t/2})}|\tilde\nabla h|^2\right).
\end{split}
\end{equation}

By covering argument and volume non-collapsed condition, we might cover $B_{\tilde g(s)}(x,k\sqrt{t/2})\subseteq \bigcup_{j=1}^{N_k} B_{\tilde g(s)}(x_j,\sqrt{t/2})$ such that $N_k\leq C_6k^ne^{Ck\sqrt{t}}$. Together with \eqref{eqn:Morrey-spt} yields
\begin{equation}
\begin{split}
\mathbf{III}_3&\leq \sum_{k=1}^\infty\sum_{j=1}^{N_k} \frac{C_5k}{t^{n/2}}\cdot\exp\left(-\frac{k^2}{C_5} \right)\left(\int^{t/2}_0\int_{B_{\tilde g(s)}(x_j,\sqrt{t/2})}|\tilde\nabla h|^2d\mathrm{vol}_{\tilde g(s)} ds\right)\\
&\leq C_7 \delta^2.
\end{split}
\end{equation}

For $\mathbf{III}_2$, we use Proposition~\ref{prop:Petrunin} instead of \eqref{eqn:Morrey-spt} and argue similarly to $\mathbf{III}_3$ to show that $\mathbf{III}_2\leq C_8 \delta^2$. Finally for $\mathbf{III}_1$, we use gradient estimate from Proposition~\ref{prop:heat-kernel}, \eqref{eqn:Morrey-spt} and covering argument above again to show that 
\begin{equation}
\begin{split}
\mathbf{III}_1 &\leq C_n\delta\int^{t/2}_0\int^\infty_0 \frac{C_4r}{t^{(n+3)/2}}\cdot\exp\left(-\frac{r^2}{C_4t} \right)\left(\int_{B_{\tilde g(s)}(x,r)}|\tilde\nabla h|\right) \,dr\\
&\leq \sum_{k=1}^\infty\sum_{j=1}^{N_k} \frac{C_8\delta k}{t^{(n+1)/2}}\cdot\exp\left(-\frac{k^2}{C_5} \right)\left(\int^{t/2}_0\int_{B_{\tilde g(s)}(x_j,\sqrt{t/2})}|\tilde\nabla h|d\mathrm{vol}_{\tilde g(s)} ds\right)\\
&\leq C_9 \delta^2
\end{split}
\end{equation}

Hence we have 
\begin{equation}\label{eqn:est-sup-h-ref}
\delta\leq \mathbf{I}+\mathbf{II}+\mathbf{III}\leq C_1\e+C_{10}\delta^2.
\end{equation}

Hence if $S$ is such that $C_{10}\delta< \frac12$, then we have  $\delta\leq 2C_1 \e $. In particular if we choose $\bar\e_0\leq (10C_1C_{10})^{-1}$, then we must have $S=T\wedge 1$ and $||h||_{L^\infty,M\times [0,S]}\leq 2C_1 \e$. This proves the quantitative existence in the compact case.

The case when $(M,\tilde g(t))$ is complete non-compact with bounded curvature up to $t=0$ is similar to the compact case. We focus on the case when $\tilde g(t)$ has no uniform bounded curvature up to $t=0$. We let $\phi_i$ be a sequence of smooth cut-off function supported on $B_{\tilde g(0)}(x_0,i)$ which exhausts $M$. Let $\sigma_i\to 0$ be a sequence of time and we approximate $g_0$ by $g_{i,0}:=\phi_i g_0+(1-\phi_i)\tilde g(\sigma_i)$ which has bounded curvature. We choose $\sigma_i$ so that $|g_{i,0}-\tilde g(\sigma_i)|_{\tilde g(\sigma_i)}\leq 2\e$. By the work of Simon \cite{Simon2002} (see also the work of Shi \cite{Shi1989}), there exists a short-time solution $g_i(t)$ to \eqref{eqn:RDF} with respect to $\tilde g(\sigma_i+t)$ starting from $g_{i,0}$. Using the same analysis as in the compact case, we see that $g_i(t)$ exists up to $T\wedge 1$ and satisfies $|g_i(t)-\tilde g(\sigma_i+t)|\leq 4C_1\e$, provided that $\bar\e_0\leq (20C_1 C_{10})^{-1}$. This also enables us to let $i\to +\infty$ to obtain the global solution $g(t)$ by interior estimates of Ricci-DeTurck flow, see \cite[Lemma 2.4]{Simon2002} or \cite[Proposition 2.2]{ChuLee2025} for example. This proves the quantitative existence in all cases. The $C^0_{loc}$ convergence as $t\to 0$ follows from the proof \cite[Theorem 5.2]{Simon2002} where the argument is purely local. 

\bigskip

Now, given two Ricci-DeTurck flows $g(t)$ and $\hat g(t)$ with respect to $\tilde g(t)$. We want to show that the difference $h:=g-\hat g$ will be persisted up to $t=S=T\wedge 1$, under \eqref{eqn:RDF}. The computation is almost identical to estimating $|g-\tilde g|$. We denote $\sigma:=||g-\hat g||_{L^\infty,M\times (0,S]}$. By Lemma~\ref{lma:kernel-repres}, we conclude 
\begin{equation}
    \begin{split}
        \sigma&\leq \sup_{(x,t)\in M\times (0,S]}\limsup_{t'\to0} \int_M G(x,t;y,t') |h|(y,t') d\mathrm{vol}_{\tilde g(t')}\\
        &\quad +C_n\gamma \int^{S}_0\int_M G|h| |\Rm(\tilde g)| \,d\mathrm{vol}_{\tilde g(s)}ds\\
        &\quad + C_n \int^S_0\int_M G|\Rm(\tilde g) ||h|^2+ G|h||\tilde \nabla g|^2\,d\mathrm{vol}_{\tilde g(s)}ds\\
        &\quad +C_n\int^S_0\int_MG|\tilde\nabla h| (|\tilde \nabla g|+|\tilde\nabla \hat g|)+G|h||\tilde\nabla G|  (|\tilde \nabla h|+|\tilde\nabla \hat g|)\,d\mathrm{vol}_{\tilde g(s)}ds\\
        &=\mathbf{A}+\mathbf{B}+\mathbf{C}+\mathbf{D}.
    \end{split}
\end{equation}

We observe that most of the terms can be handled as in the derivation of \eqref{eqn:est-sup-h-ref}, by using Proposition~\ref{prop:Petrunin}, Lemma~\ref{lma:improved-Morrey-grad} and Lemma~\ref{lma:Morrey-grad}. In particular with $|g-\tilde g|\leq \Lambda \bar\e_0$, the same analysis with simple Cauchy inequalities yield 
\begin{equation}
\mathbf{C}+\mathbf{D}\leq C_{11}\sigma(\sigma+ \bar \e_0) \leq C_{12}\bar\e_0\sigma.
\end{equation}
where we have also used $\sigma\leq 2\bar\e_0$.

We use Proposition~\ref{prop:Petrunin} again to control $\mathbf{B}$ as in the estimate of $\mathbf{II}+\mathbf{III}$. Hence, we conclude that
\begin{equation}
\begin{split}
\sigma\leq \mathbf{A}+C_{12}\bar \e_0 \sigma +C_{13}\gamma \sigma.
\end{split}
\end{equation}

By choosing $\gamma=(4C_{13})^{-1}$ and 
$$\bar\e_0(n,\a):=\frac12 \min\{(20C_1C_{10})^{-1},(4C_{12})^{-1},\e_0(n,\gamma)\},$$
where $\e_0$ is obtained from Lemma~\ref{lma:evo-h-L1}, we conclude that 
\begin{equation}\label{eqn:heat-esti-a.e.}
  \frac12  ||g-\hat g||_{L^\infty,M\times (0,S]}\leq\sup_{(x,t)\in M\times (0,S]} \limsup_{t'\to0}\int_M G(x,t;y,t') |h|(y,t') d\mathrm{vol}_{\tilde g(t'),y}.
\end{equation}
Result follows using Proposition~\ref{prop:heat-kernel}.
\end{proof}

For application, we will work on a more general situation that  the reference Ricci flow $\tilde g(t)$ is coming out of a sphere with polyhedral singularities in the Gromov-Hausdorff sense. And the Ricci-DeTurck flow is initially only $L^\infty(M)$ and continuous away from a singularity $\mathcal{S}$ of high codiemsion. In this case, we still have uniqueness and stability analogous to smooth case, as a simple consequence of the heat kernel estimate.
\begin{thm}\label{thm:unique-general}
Suppose $(M^n,\tilde d,x_0)$ is a pointed complete metric space and $\tilde g(t)$ is a Ricci flow on $M^n\times (0,S]$ such that (a)-(c) in sub-section hold and 
\begin{enumerate}
    \item $(M,d_{\tilde g(t)},x_0)\to (M,\tilde d_0,x_0)$ in the distance sense, as $t\to0$;
    \item There exists a subset $\mathcal{S}\subseteq M$ such that $M\setminus \mathcal{S}$ is open, dense subset with full $\tilde d_0$-Hausdorff measure and $\tilde g(t)$ converges to some continuous metric $\tilde g(0)$ locally uniformly outside $\mathcal{S}$ as $t\to 0$.
\end{enumerate}
If $\Lambda,\bar\e_0$ are the constants from Theorem~\ref{thm:weak-stable-RDF}
and $g_i(t),t\in (0,S],i=1,2$ are Ricci-DeTurck flows with respect to $\tilde g(t)$ such that for $i=1,2$,
\begin{enumerate}
    \item[(i)] $\sup_{M\times (0,S]}|g_i-\tilde g|\leq \Lambda\bar\e_0$ and;
    \item[(ii)] $g_i(t)\to g_{i,0}$ in $C^0_{loc}(M\setminus \mathcal{S})$ as $t\to 0$, for some metric $g_{i,0}$ in $C^0_{loc}(M\setminus \mathcal{S})$,
\end{enumerate}
then we have 
\begin{equation}
\sup_{M\times (0,S]} |g_1-g_2|\leq \Lambda  ||g_{1,0}-g_{2,0}||_{L^\infty}.
\end{equation}
 In particular, the Ricci-DeTurck flow with respect to $\tilde g(t)$ is unique within the class of solutions satisfying (i) and (ii), with the same initial metric $g_0\in C^0_{loc}(M\setminus \mathcal{S})$.
\end{thm}
\begin{proof}
We first note that the Gromov-Hausdorff convergence can be improved to measured Gromov-Hausdorff convergence, thanks to assumption (c) and Colding's volume convergence Theorem \cite{Colding1997}. 

Denote $h=g_1-g_2$.  Thanks to (i), \eqref{eqn:heat-esti-a.e.} still holds so that 
\begin{equation}\label{eqn:kernel}
\frac12 \cdot \sup_{M\times (0,S]} |g_1(t)-g_2(t)|\leq \sup_{(x,t)\in M\times (0,S]} \limsup_{s\to0}\int_M G(x,t;y,s) |h|(y,s) d\mathrm{vol}_{\tilde g(s)}.
\end{equation}

We now examine the right hand side carefully. We split the integral as
\begin{equation}
    \begin{split}
        \int_M G(x,t;y,s) |h|(y,s) d\mathrm{vol}_{\tilde g(s)}&= \left( \int_{M
        \setminus\mathcal{S}}+\int_{\mathcal{S}} \right) G(x,t;y,s) |h|(y,s) d\mathrm{vol}_{\tilde g(s)}\\
        &=\mathbf{I}+\mathbf{II}.
    \end{split}
\end{equation}

By \cite[Lemma 3.1]{SimonTopping2021}, the Hausdorff measure $\mathcal{H}^{n-1}_{\tilde d_0}$ of $\tilde d_0$ satisfies $\mathcal{H}_{d_{\tilde g(t)}}\leq \mathcal{H}_{\tilde d_0}$ for all $t\in (0,S]$. Using this, the heat kernel estimate from Proposition~\ref{prop:heat-kernel} and the assumed estimate of $h$, 
\begin{equation}
    \begin{split}
       \limsup_{s\to 0^+} \mathbf{II}&\leq 0
    \end{split}
\end{equation}
since $\mathcal{H}_{\tilde d_0}(\mathcal{S})=0$.

On the other hand, thanks to the convergence of $\tilde g(t),g_1(t)$ and $g_2(t)$ outside $\mathcal{S}$ and Proposition~\ref{prop:heat-kernel}, dominated convergence Theorem yields
\begin{equation}
    \limsup_{s\to 0}\mathbf{I}\leq C_2(n,\a) ||g_{1,0}-g_{2,0}||_{L^\infty(M,\tilde g_0)}
\end{equation}
and hence $\limsup_{s\to 0}(\mathbf{I}+\mathbf{II})\leq C_2(n,\a) ||g_{1,0}-g_{2,0}||_{L^\infty(M,\tilde g_0)}$. This completes the proof by combining this with \eqref{eqn:kernel} by enlarging $\Lambda$ if necessary. 
\end{proof}

\begin{rem}
By \cite[Lemma 3.1]{SimonTopping2021}, the distance function $\tilde d_0$ generates the same topology as $M$. 
\end{rem}

\section{Smoothing singular metrics on spheres and stability} \label{sec:stability-cts}

In this section, we consider metrics with polyhedral singularity on sphere. 
For any $\b:=(\beta_1,\cdots,\beta_{n-1})\in (0,1]^{n-1}$, we define
\[g_{0,(\beta_1,\cdots,\beta_{n-1})}:=\beta_{1}^2(dx_{1}^2+\beta_{2}^2\sin^2 x_{1}(dx_{2}^2+(\cdots+\beta_{n-1}^2\sin^2 x_{n-2}dx_{n-1}^2)\cdots),\]
where $x_{1}\in[0,\pi],\cdots,x_{n-2}\in[0,\pi],x_{n-1}\in[0,2\pi]$, to be a metric on $\mathbb S^{n-1}$, which is smooth away from a subset of (Hausdorff) dimension at most $n-3$. 
In particular, $g_{0,(1,\cdots,1)}$ is the standard spherical metric on $\mathbb S^{n-1}$.
Moreover, these metrics satisfy $\Rm\ge 1$ on the smooth part, by Lemma~\ref{lma:Rm-sphere-warp}. In case some $\b_i<1$, the resulting metric $g_{0,(\beta_1,\cdots,\beta_{n-1})}$ is singular. The main result in this section provides a continuous smoothing as $\b$ varies. 

\begin{thm}[Link metrics on $\mathbb S^{n-1}$]\label{thm: links}
For each $n\ge3$, and $j\ge1$,
there exists a sequence of  continuous maps $\mathcal S_j$ from $[0,1]^{n-1}$ to the space of all smooth metrics on $\mathbb S^{n-1}$, such that for any $(\beta_1,\cdots,\beta_{n-1})\in [0,1]^{n-1}$ that:
\begin{enumerate}
    \item\label{i:Rm ge 1}  $\mathcal S_j(\beta_1,\cdots,\beta_{n-1})$ satisfies $\Rm \geq1$;
    \item\label{i:GH-close} The metric space induced by $\mathcal S_j(\beta_1,\cdots,\beta_{n-1})$ is $j^{-1}$-close in the Gromov-Hausdorff sense to the metric space induced by 
    $g_{0,(\beta_1,\cdots,\beta_{n-1})}$.
    \item\label{i:symmetry} If for some $k\leq n-1$ and
     $1\le i_1< i_2<\cdots< i_k\le n-1$ we have $\beta_j=1$ for $j\neq i_{\ell}$ for each $\ell=1,...,k$, then
    the metric $\mathcal S_j(\beta_1,\cdots,\beta_{n-1})$ is $O(i_1-1)\times O(i_2-i_1)\times O(i_3-i_2)\times \cdots \times O(i_k-i_{k-1})\times O(n+1-i_k)$-symmetric, where $O(i_1-1)=O(0)=1$ by convention if $i_1=1$.
    \end{enumerate}
\end{thm}

Here the continuity on the space of metrics in $\mathbb{S}^{n-1}$ is with respect to Cheeger-Gromov topology. We will split the proof of Theorem~\ref{thm: links} for a better presentation. This will be based on constructing Ricci flow starting from $g_{0,(\beta_1,\cdots,\beta_{n-1})}$. The Ricci flow smoothing constructed below will be non-smooth at $t=0$. In the following, a Ricci flow $g_{\b}(t)$ is said to be coming out of $g_{0,\b}$ if $d_{g_{\b}(t)}$ converges to a distance metric $d_{0,\b}$ as $t\to 0^+$ such that $(\mathbb{S}^{n-1},d_{0,\b})$ is isometric to $(\mathbb{S}^{n-1},g_{0,\b})$ as a metric space.

\subsection{Ricci flow smoothing of $\mathcal{S}_j$ by stability}

The most straight forward approach to construct $\mathcal{S}_j$ is to use Theorem~\ref{thm:RF-existence} iteratively. From the nature of construction, the stability should be inherited from the stability of the expanders with respect to the links. Since the variation of link is in $L^\infty$ topology, it is unclear why the resulting singularity model is continuous. We combine the construction using stability.

\begin{thm}\label{thm:const-2}
For all $j>0$ and $n\geq 3$, there exists $\a(n,j),S(n,j)>0$ such the following holds: For all $\b=(\b_1,...,\b_{n-1})\in [j^{-1},1]^{n-1}$, there is a one parameter family of metrics $ g_\b(t)$ on $\mathbb{S}^{n-1}\times (0, S]$ coming out of $g_{0,\b}$ and a homeomorphism $\psi_{0,\b}$ such that 
\begin{enumerate}
\item[(i)] $1\leq  \mathrm{Rm}(g_{(\b_1,...,\b_{n-1})}(t))\leq \a t^{-1}$, for $n\geq 4$;
\item[(ii)] $\mathrm{inj}(g_{(\b_1,...,\b_{n-1})}(t))\geq \sqrt{\a^{-1} t}$, for $n\geq 4$;
\item[(iii)] There is $\delta(n,j)>0$ such that if $\a,\b \in [j^{-1},1]^{n-1}$ with $|\a-\b|<\delta$, then there exists a time-dependent diffeomorphism $\Psi_{\a,\b}:\left(\mathbb{S}^{n-1},g_\b(t)\right)\to \left( \mathbb{S}^{n-1},g_\a(t)\right)$ for $t\in (0,S]$ such that
\begin{enumerate}
    \item $\lim_{t\to 0}\Psi_{\a,\b}(t)=(\psi_{0,\a})^{-1}\circ \psi_{0,\b}$ in the sense of map; 
    \item The pull-back of Ricci flow is continuous in the sense that:
    For all $\e>0$, there is a $\delta'=\delta'(j,n,\e)$ such that if $|\a-\b|<\delta'$, then $$\sup_{\mathbb{S}^{n-1}\times (0,S]}||\left(\Psi_{\a,\b}(t)\right)^*g_\a(t)-g_\b(t)||_{g_\b(t)}<\e.$$
\end{enumerate}
\item[(iv)] There exists a set $\mathtt{S}_\b$ such that $\mathbb{S}^{n-1}\setminus \mathtt{S}_\b$ is open dense subset of full $(n-1)$-dimensional Hausdorff measure w.r.t. $d_{0,\b}:=\lim_{t\to 0^+}d_{g_{\b}(t)}$, whose existence follows from $(i)$. Furthermore, on $\mathbb{S}^{n-1}\setminus \mathtt{S}_\b$, $\psi_{0,\b}$ is smooth local diffeomorphism satisfying $g_\b(0)=(\psi_{0,\b})^*g_{0,\b}$, and is bi-Lipschitz from $(\mathbb{S}^{n-1},d_{0,\b})$ to $\mathbb{S}^{n-1}$;
\item [(v)] $\psi_{0,(j^{-1},...,j^{-1})}=\mathrm{id}$.
\end{enumerate}
\end{thm}
\begin{rem}
    When $n\geq 4$, $\Psi_{\a,\b}(t)$ and $g_\b(t)$ will be chosen to be solution to the Ricci-harmonic map heat flows and Ricci flows respectively.
\end{rem}
\begin{proof}
We prove it by induction on $n$. For any $x\in \R^{n-1}$ and $r>0$, we will write $V_r(x):=\{z\in \R^{n-1}: |z-x|\leq r\}$. 
Since the singularity is only conic when $n=3$, Theorem~\ref{thm:const-2} follows directly using ODE construction in \cite{Lai2024}, see also \cite{ChanLee2025}.

For $n\geq 4$, we assume that the conclusion holds on $\mathbb{S}^{n-2}$ and show that the same is true on $\mathbb{S}^{n-1}$. That said, we have constructed one parameter family of metrics 
$\bar g_{(\beta_2,\cdots,\beta_{n-1})}(t),t\in (0,S']$ 
on sphere $\mathbb{S}^{n-2}\times (0,S']$, a corresponding homeomorphism $\bar\psi_{0,(\beta_2,\cdots,\beta_{n-1})}$ and $\mathtt{S}_{(\beta_2,\cdots,\beta_{n-1})}$, satisfying (iii)-(v) in Theorem \ref{thm:const-2}.  We will then show the existence of a family of metrics on $\mathbb{S}^{n-1}\times (0,S']$ and homeomorphisms satisfying (i)-(v) in the same theorem. The quantities with bar are referring those in dimension $n-2$. For example, when $\b=(\b_1,\b_2,...,\b_{n-1})$, then $\bar \b=(\b_2,...,\b_{n-1})\in \R^{n-2}$.

\bigskip

For $\b,\hat\b\in [j^{-1},1]^{n-1}$ where $|\b-\hat\b|$ is small,  denote
\begin{equation}
    \left\{
    \begin{array}{ll}
         g_{0,\b}:=\b_1^2\left(dx_1^2+\sin^2 x_1 \cdot \bar g_{0,(\b_2,...,\b_{n-1})} \right)  \\[1mm]
         g_{0,\b,\tau}:=\b_1^2\left(dx_1^2+\sin^2 x_1 \cdot \bar g_{(\b_2,...,\b_{n-1})}(\tau) \right)\\[1mm]
         g_{0,\b,\hat\b,\tau}:=(F_{\b,\hat\b,\tau})^* g_{0,\b,\tau}
    \end{array}
    \right.
\end{equation}
where $F_{\b,\hat \b,\tau}:\mathbb{S}^{n-1}\to\mathbb{S}^{n-1}$ is given by
\begin{equation}\label{Fdef}
F_{\b,\hat \b,\tau}(x_1,x_2,...,x_{n-1}):=\left(x_1,\bar\Psi_{\b,\hat \b}(\tau)(x_2,...,x_{n-1})\right)
\end{equation}
and $\bar\Psi_{\b,\hat\b}(\tau)$ is the diffeomorphism of $\mathbb{S}^{n-2}$ with $\bar\Psi_{\b,\hat\b}(0)=(\bar\psi_{0,\b})^{-1}\circ \bar\psi_{0,\hat\b}$, which is obtained from induction. For $\tau>0$, $g_{0,\b,\tau}$ is smooth outside the tips, i.e. $\mathtt{Tips}:=\{o,\pi\}\times \mathbb{S}^{n-2}/\sim$ and it is of singularity type in sense of definition~\ref{defn:deca}. Here the beta with  a hat is referring to the gauge we use to compare the Ricci flows.

We note here that as $\tau\to 0$, we only have 
\begin{equation}\label{eqn:initial-gauged-ind}
  g_{0,\b,0}:=\lim_{\tau\to 0} g_{0,\b,\tau}=\b_1^2\left(dx_1^2+\sin^2 x_1 \cdot (\bar\psi_{0,(\b_2,...,\b_{n-1})})^* g_{0,(\b_2,...,\b_{n-1})} \right)
\end{equation}
on $ \mathbb{S}^{n-1}\setminus \left(\mathtt{Tips}\,\bigcup\,  [0,\pi]\times \mathtt{S}_{(\beta_2,\cdots,\beta_{n-1})}/\sim \right)$.

We fix $\hat\b_0:=(j^{-1},...,j^{-1})$, the corner to be our starting point. By considering the metric $g_{0,\hat \b_0,\tau}$ which is of isolated singularity for $\tau>0$, there exists a Ricci flow $g_{\hat\b_0,\tau}(t)$ coming out of $g_{0,\hat\b_0,\tau}$ satisfying (i) and (ii) for some $\a>0$ depending only on $n,j$ by Theorem~\ref{thm:RF-existence}. Furthermore Hamilton's compactness \cite{Hamilton1995-2} allow us to pass $\tau_i\to 0$ and obtain $g_\b(t)$ on $\mathbb{S}^{n-1}\times (0,S]$. 
 The sequence $\tau_i\to 0$ will be fixed from now on. For notation convenience, in what follows $\tau\to 0$ will mean $\tau_i\to 0$.
After identifying $g_{0,\hat\b_0}$ using distance convergence as in Theorem~\ref{thm:RF-existence}, we might assume $g_{\hat \b_0}(t)\to g_{0,\hat \b_0}$ as $t\to 0$ locally smoothly away from $ \mathtt{S}_{\hat \b_0}:=\mathtt{Tips}\,\bigcup\,[0,\pi]\times \mathtt{S}_{(j^{-1},\cdots,j^{-1})}/\sim $. For $\tau>0$, we also denote $\mathtt{S}_{\hat\b_0,\tau}:=\mathtt{Tips}$ where $g_{\hat\b_0,\tau}(0)=g_{0,\hat\b_0,\tau}$ is singular.

\bigskip

We fix the $\a$ from above and Theorem~\ref{thm:RF-existence}, which in turn depends on the non-collapsing among all $g_{0,\b}$, i.e. the lower bound $\mathrm{Vol}(\mathbb{S}^{n-1},g_{0,(j^{-1},...,j^{-1})})$.  We want to propagate the existence of $g_{\hat\b_0,\tau}(t)$ to all $\b\in [j^{-1},1]^{n-1}$. We need an inductive construction algorithm. 
\begin{claim}\label{claim:const-RDF} 
There exists $\delta_0>0$ depending only on $j$ such that the following holds: Suppose for $\hat\b\in [j^{-1},1]^{n-1}$ and $\tau>0$, we have already constructed a Ricci flow $g_{\hat \b,\tau}(t),t\in (0,S]$ coming out of $g_{0,\hat\b,\tau}$ and there exist a homeomorphism $G_{0,\hat\b,\tau}$ and a finite set $\mathtt{S}_{\hat\b,\tau}$ such that 
\begin{enumerate}
\item  $G_{0,\hat \b,\tau}$ is a local diffeomorphism away from $\mathtt{S}_{\hat \b,\tau}$
\item $g_{\hat \b,\tau}(t)$ satisfies the conclusion (i) and (ii), for $\a>0$;
\item $g_{\hat \b,\tau}(t)\to (G_{0,\hat \b,\tau})^*g_{0,\hat \b,\tau}$ in $C^\infty_{loc}(\mathbb{S}^{n-1}\setminus  \mathtt{S}_{\hat\b,\tau})$ as $t\to 0$,
\end{enumerate}
 then for $\b\in V_{\delta_0}(\hat\b)$, there exists a Ricci-DeTurck flow $\hat g_{\b,\hat \b,\tau}(t)$ with respect to $g_{\hat\b,\tau}(t)$ coming out of $g_{0,\b,\tau}$ and a time-dependent diffeomorphism $\Phi_{\b,\hat\b,\tau},t\in (0,S]$ such that $g_{\b,\hat \b,\tau}=(\Phi_{\b,\hat\b,\tau})^*\hat g_{\b,\hat\b,\tau}$ is a solution to the Ricci flow for $t\in (0,S]$ and
\begin{enumerate}
\item $g_{\b,\hat \b,\tau}(t)$ satisfies (i) and (ii) with the same $\a>0$;
\item $\Phi_{\b,\hat\b,\tau}(t)$ is the Ricci-DeTurck ODE solution with respect to $\hat g_{\b,\hat\b,\tau}$ and $g_{\hat \b,\tau}(t)$ such that $\Phi_{\b,\hat\b,\tau}(S)=\mathrm{Id}$;
\item  $ (1-\bar\e_0)^{1/3} g_{\hat\b,\tau}(t)\leq \hat g_{\b,\hat\b,\tau}(t)\leq (1+\bar\e_0)^{1/3} g_{\hat\b,\tau}(t)$ for $t\in (0,S]$ where $\bar\e_0(n,\a)$ is the constant from Theorem~\ref{thm:weak-stable-RDF};
\item $\Phi_{\b,\hat\b,\tau}(t)\to \Phi_{0,\b,\hat \b,\tau}$ as $t\to 0^+$ in $C^\infty_{loc}(\mathbb{S}^{n-1}\setminus \mathtt{S}_{\hat\b,\tau})\cap C^0(\mathbb{S}^{n-1})$;
\item $g_{\b,\hat \b,\tau}(t)\to (\Phi_{0,\b,\hat\b,\tau})^*\circ (G_{0,\hat\b,\tau})^*\circ (F_{\b,\hat\b,\tau})^* g_{0,\b,\tau}$ in $C^\infty_{loc}(\mathbb{S}^{n-1}\setminus  \mathtt{S}_{\b,\hat\b,\tau})$ as $t\to 0$, where 
\begin{equation}
 \mathtt{S}_{\b,\hat\b,\tau}:= \Phi_{0,\b,\hat\b,\tau}^{-1}\left( \mathtt{S}_{\hat\b,\tau}\cup G_{0,\hat\b,\tau}^{-1}(\mathtt{Tips}) \right).
\end{equation}
Here $\mathtt{Tips}$ is referring to the singularity of $F_{\b,\hat\b,\tau}^* g_{0,\b,\tau}$.
\end{enumerate}
\end{claim}

\begin{proof}[Proof of claim~\ref{claim:const-RDF}]   
Let $\delta_0$ be a constant to be determined. For $\b\in [j^{-1},1]^{n-1}\cap V_{\delta_0}(\hat \b)$, we will construct Ricci flow by regularizing its link as in construction of $g_{\hat\b_0,\tau}(t)$ using $g_{(\beta_2,\cdots,\beta_{n-1})}(t)$. The main difference is in this case, we will rely on construction of Ricci-DeTurck flow from stability.  By induction hypothesis, if $\delta_0$ is small enough, then we have a time-dependent diffeomorphism $\bar\Psi_{\b,\hat \b}$ of $\mathbb{S}^{n-2}$ such that 
\begin{equation}\label{eqn:stability-fromind}
    |(\bar\Psi_{\b,\hat\b}(\tau))^* \bar g_{(\beta_2,\cdots,\beta_{n-1})}(\tau)-\bar g_{(\hat\beta_{2},\cdots,\hat\beta_{n-1})}(\tau)|<\frac12 \bar\e_1
\end{equation}
for all $\tau\in (0,S']$, where $\hat\b=(\hat \b_{1},\hat \b_{2},\cdots,\hat \b_{n-1})$. Here $\bar\e_1$ is small enough such that $(1+\Lambda\bar\e_1)\leq (1+\bar\e_0)^{1/3}$ and $(1-\Lambda\bar\e_1)\geq (1-\bar\e_0)^{1/3}$ where $\Lambda,\bar\e_0$ are from Theorem~\ref{thm:weak-stable-RDF}.

\bigskip

The pull-back metric $g_{0,\b,\hat \b,\tau}:=F_{\b,\hat \b,\tau}^*g_{0,\b,\tau}$ is a $L^\infty$ metric on $\mathbb{S}^{n-1}$ with isolated singularity at $\mathtt{Tips}$ for $\tau>0$ in which
\begin{equation}\label{eqn:metric-equ-b-hatb}
    \left(1-\bar\e_1\right)\cdot g_{0,\hat \b,\tau}\leq g_{0,\b,\hat \b,\tau}\leq \left(1+ \bar\e_1\right)\cdot  g_{0,\hat \b,\tau}
\end{equation}
by \eqref{eqn:stability-fromind}, if we shrink $\delta_0$ further.  Suppose $\mathtt{S}_{\hat\b,\tau}\cup G_{0,\hat\b,\tau}^{-1}(\mathtt{Tips})=\{p_i\}_{i=1}^N$, we define
$$\phi_\sigma(x):=\max_{i=1,\dots, N.}\phi(d_{g_{\hat \b,\tau}(\sigma)}(x,p_i)/\sigma^{1/4})$$
where $\sigma\to 0$ and $\phi$ is a fix smooth non-increasing function such that $\phi\equiv 1$ on $[0,3/2]$ and vanishes outside $[0,2]$. Analogous to the proof of Theorem~\ref{thm:RF-existence}, we consider the $L^\infty$ metric (singularity only due to the regularity of distance function):
\begin{equation}\label{eqn:initial-data-2approx}
g_{0,\b,\hat \b,\tau,\sigma}:=\phi_\sigma\cdot  g_{\hat \b,\tau}(\sigma)+(1-\phi_\sigma)\cdot (G_{0,\hat\b,\tau})^* g_{0,\b,\hat \b,\tau}.
\end{equation}
as $\sigma\to 0$, for $\tau>0$. It can be seen that $g_{0,\b,\hat \b,\tau,\sigma}$ converges in $C^\infty_{\text{loc}}$ to $(G_{0,\hat\b,\tau})^* g_{0,\b,\hat \b,\tau}$ away from $\mathtt{S}_{\hat\b,\tau}\cup G_{0,\hat\b,\tau}^{-1}(\mathtt{Tips})$ as $\sigma\to 0$, fixing $\tau>0$.
\bigskip

As in the proof of Theorem~\ref{thm:RF-existence}, by pseudolocality and \eqref{eqn:metric-equ-b-hatb}  we might assume 
\begin{equation}\label{eqn:lipsc-initial}
    \left(1- \bar\e_1\right) g_{\hat \b,\tau}(\sigma)\leq g_{0,\b,\hat\b,\tau,\sigma}\leq \left(1+\bar\e_1\right) g_{\hat \b,\tau}(\sigma)
\end{equation}
as $\sigma\to 0$, for fixed $\tau>0$. Thus we might apply Theorem~\ref{thm:weak-stable-RDF} to $g_{0,\b,\hat\b,\tau,\sigma}$ to obtain a Ricci-DeTurck flow $\hat g_{\b,\hat\b,\tau,\sigma}(t)$ with respect to $g_{\hat\b,\tau}(\sigma+t)$ starting from $g_{0,\b,\hat\b,\tau,\sigma}$ such that 
\begin{equation}\label{eqn:lipsc}
   \left (1-\Lambda\bar\e_1\right) g_{\hat\b,\tau}(\sigma+t)\leq \hat g_{\b,\hat\b,\tau,\sigma}(t)\leq \left(1+\Lambda\bar\e_1\right) g_{\hat\b,\tau}(\sigma+t)
\end{equation}
on $\mathbb{S}^{n-1}\times (0,S]$.

We construct the canonical Ricci flow by using the Ricci-DeTurck ODE as follows. We let $\Psi_{\b,\hat \b,\tau,\sigma}$ be the time-dependent diffeomorphism given by solving the ODE:
\begin{equation} \label{eqn:RDF-ODE}
 \left\{
 \begin{array}{ll}
 \partial_t \Phi_{\b,\hat \b,\tau,\sigma}(x,t)=-W_{\b,\hat \b,\tau,\sigma}\left(\Phi_{\b,\hat \b,\tau,\sigma}(x,t),t\right);\\[1mm]
  \Phi_{\b,\hat \b,\tau,\sigma}(x,S)=x;\\[1mm]
  (W_{\b,\hat \b,\tau,\sigma})^k= (\hat g_{\b,\hat\b,\tau,\sigma})^{ij}\left[\Gamma_{ij}^k(\hat g_{\b,\hat\b,\tau,\sigma}(t))-\Gamma_{ij}^k(g_{\hat\b,\tau}(\sigma+t))\right]
 \end{array}
 \right.
 \end{equation}
 
Then $ g_{\b,\hat\b,\tau,\sigma}(t):=(\Phi_{\b,\hat \b,\tau,\sigma}(t))^* \hat g_{\b,\hat\b,\tau,\sigma}(t)$ defines a Ricci flow on $\mathbb{S}^{n-1}\times (0,S]$ such that $ g_{\b,\hat\b,\tau,\sigma}(S)=\hat g_{\b,\hat\b,\tau,\sigma}(S)$. 

Thanks to Theorem~\ref{thm:unique-general} and Lemma~\ref{lma:RDF-higher-ord}, we might let $\sigma\to 0$ without taking subsequence on $\hat g_{\b,\hat\b,\tau,\sigma},\Phi_{\b,\hat \b,\tau,\sigma}$ for $t\in (0,S]$ such that the limiting solutions $\hat g_{\b,\hat\b,\tau},\Psi_{\b,\hat \b,\tau}$ exists on $\mathbb{S}^{n-1}\times (0,S]$. This also uniquely determines the limiting Ricci flow  $g_{\b,\hat\b,\tau}:=(\Phi_{\b,\hat \b,\tau})^*\hat g_{\b,\hat\b,\tau}$ with $g_{\b,\hat\b,\tau}(S)=\hat g_{\b,\hat \b,\tau}(S)$.  From \eqref{eqn:lipsc}, it is clear that 
\be\label{eqn:RDF-tau-pict}
        (1-\Lambda\bar\e_1) g_{\hat\b,\tau}(t)\leq \hat g_{\b,\hat\b,\tau}(t)\leq (1+\Lambda\bar\e_1) g_{\hat\b,\tau}(t)
\ee
for $t\in (0,S]$.  By Lemma~\ref{lma:RDF-higher-ord}, we also know $\Phi_{\b,\hat\b,\tau}(t)$ converges to some continuous map $\Phi_{0,\b,\hat\b,\tau}$ as $t\to 0$. By interior estimate, see \cite[Proposition 2.2]{ChuLee2025} for example, $\Phi_{0,\b,\hat\b,\tau}$ is smooth outside $\Phi_{0,\b,\hat\b,\tau}^{-1}\left(\mathtt{S}_{\hat\b,\tau}\cup G_{0,\hat\b,\tau}^{-1}(\mathtt{Tips})\right)$ where the convergence as $t\to 0$ is in local smooth topology. 
 Consequently, $g_{\b,\hat \b,\tau}(t)$ converges in $C^\infty_{\text{loc}}$ to $(\Phi_{0,\b,\hat\b,\tau})^*(G_{0,\hat\b,\tau})^* g_{0,\b,\hat \b,\tau}$ as $t\to 0$ away from $ \Phi_{0,\b,\hat\b,\tau}^{-1}\left(\mathtt{S}_{\hat\b,\tau}\cup G_{0,\hat\b,\tau}^{-1}(\mathtt{Tips})\right)$. Furthermore we have 
\begin{equation}\label{eqn:degree}
    \begin{split}
        &\quad d_{g_{\hat \b,\tau}(s)} \left(\Phi_{\b,\hat \b,\tau}(x,t),\Phi_{\b,\hat \b,\tau}(x,s) \right)\\
        &\leq d_{g_{\hat \b,\tau}(t)} \left(\Phi_{\b,\hat \b,\tau}(x,t),\Phi_{\b,\hat \b,\tau}(x,s) \right)+C_0(n,\a) \sqrt{t}\\
        &\leq \int^t_s |\partial_z \Phi_{\b,\hat \b,\tau}(x,z)|_{g_{\hat\b,\tau}(t)}\,dz+C_0\sqrt{t}\\
        &\leq \int^t_s |\partial_z \Phi_{\b,\hat \b,\tau}(x,z)|_{g_{\hat \b,\tau}(z)}\,dz+C_0\sqrt{t}\leq C_1(n,\a)\sqrt{t},
    \end{split}
\end{equation}
for $0<s<t$. We now claim that $\Phi_{0,\b,\hat\b,\tau}$ is injective so that the new singular set still consists of isolated singularities. Suppose $\Phi_{0,\b,\hat\b,\tau}(x)=\Phi_{0,\b,\hat\b,\tau}(y)$ for some $x,y$, then \eqref{eqn:degree} implies
\begin{equation}
    \begin{split}
&\quad  d_{\hat g_{\b,\hat\b,\tau}(t)} \left( \Phi_{\b,\hat\b,\tau}(x,t),\Phi_{\b,\hat\b,\tau}(y,t) \right) \\
&\leq  (1+\Lambda\bar\e_1)^{1/2}\cdot d_{g_{\hat\b,\tau}(t)} \left( \Phi_{\b,\hat\b,\tau}(x,t),\Phi_{\b,\hat\b,\tau}(y,t) \right)\\ 
      &\leq  (1+\Lambda\bar\e_1)^{1/2} \cdot  \limsup_{s\to 0}d_{g_{\hat\b,\tau}(s)} \left( \Phi_{\b,\hat\b,\tau}(x,t),\Phi_{\b,\hat\b,\tau}(y,t) \right)
      \leq  C_2(n,\a)\sqrt{t}\\
    \end{split}
\end{equation}
while the left hand side can be bounded from below by 
\begin{equation}
    \begin{split}
        d_{\hat g_{\b,\hat\b,\tau}(t)}\left( \Phi_{\b,\hat\b,\tau}(x,t),\Phi_{\b,\hat\b,\tau}(y,t) \right) &=d_{g_{\b,\hat\b,\tau}(t)}\left( x,y\right) \\
        &\geq (t/S)^{\e/2} d_{g_{\b,\hat\b,\tau}(S)}\left( x,y\right)
    \end{split}
\end{equation}
using $\Ric(\hat g_{\b,\hat\b,\tau}(t))\geq -\e t^{-1}$ from  Lemma~\ref{lma:RDF-higher-ord} where $\e\to 0$ as $\bar\e_1\to 0$. If $\bar\e_1$ is small enough, then it forces $x=y$ by letting $t\to 0$, and thus shows  injective. The surjective follows from the fact that $\Phi_{\b,\hat\b,\tau}(t)$ is surjective for all $t\in (0,S]$ and the convergence. This shows that $\Phi_{0,\b,\hat\b,\tau}$ is a homeomorphism.

Finally, we claim that $g_{\b,\hat \b,\tau}(t)$ has $\mathrm{Rm}\geq 1$, which in turn implies (i). This follows from minor modification to the proof of (vi) in Theorem~\ref{thm:RF-existence}. From Lemma~\ref{lma:RDF-higher-ord}, we might assume $\mathrm{Rm}_-\leq \e t^{-1}$ since $ g_{\hat \b,\tau}(t)$ has $\mathrm{Rm}\geq 1$, where $\e\to 0$ as $\delta_0\to 0$. Similar to \eqref{eqn:evo-Rm}, the function $\varphi:=t^{-C_n\e}\mathrm{Rm}_-$ satisfies 
\begin{equation}
\left(\frac{\partial}{\partial t}-\Delta_{g_{\b,\hat\b,\tau}(t)}-R_{g_{\b,\hat\b,\tau}(t)} \right)\varphi\leq 0
\end{equation}
and thus for all $0<s<t$
\begin{equation}
\varphi(x,t)\leq \int_{\mathbb{S}^{n-1}} G(x,t;y,s) \varphi(y,s)\,d\mathrm{vol}_{g_{\b,\hat\b,\tau}(s)}.
\end{equation}

We might argue almost identically as in \eqref{eqn:res-Gr}, using the fact that the singularity is isolated, $n-1\geq 3$, $\Ric(g_{\hat \b,\tau})\geq 0$ so that from \eqref{eqn:RDF-tau-pict} the volume of geodesic ball of $g_{\b,\hat\b,\tau}$ is Euclidean like, we have 
\begin{equation}
\limsup_{s\to 0}\mathbf{I}\leq \sum_{i=1}^m\limsup_{s\to 0}\frac{1}{s^{C_n\e}} \int_{B_{g_{\b,\hat\b,\tau}(s)}(q_i,\Lambda\sqrt{s})}\frac{C }{(t-s)^{n/2}} \mathrm{Rm}_-(y,s)\,d\mathrm{vol}_{g_{\b,\hat\b,\tau}(s)}=0,
\end{equation}
as long as $2C_n\e<n-3$, see also \cite{LeeTam2025}. Here $ \mathtt{S}_{\b,\hat\b,\tau}=\{q_i\}_{i=1}^m$ are the isolated  singularities of $g_{0,\b,\hat\b,\tau}$.  
Since $\mathrm{Rm}(g_{0,\b,\hat\b,\tau})\geq 1$ away from the tips, by Lemma~\ref{lma:Rm-sphere-warp}. The remaining argument is identical, showing that $\mathrm{Rm}\geq 0$. Now we might appeal the proof of (vi) in Theorem~\ref{thm:RF-existence} to show that indeed $\mathrm{Rm}(t)\geq 1$. This also fixes $\delta_0$. 

Since we have shown $\Rm(g_{\b,\hat\b,\tau}(t))\geq 1$, by the same argument as in the proof of (vii) in Theorem~\ref{thm:RF-existence}, we see that the Ricci flow $g_{\b,\hat\b,\tau}(t)$ satisfies (i) and (ii), for the same $\a>0$ depending only on $\mathrm{Vol}(\mathbb{S}^{n-1},g_{0,(j^{-1},...,j^{-1})})$. Furthermore, it follows from \cite[Lemma 3.1]{SimonTopping2021} and $$ d_{\hat g_{\b,\hat\b,\tau}(t)}\left( \Phi_{\b,\hat\b,\tau}(x,t),\Phi_{\b,\hat\b,\tau}(y,t) \right)
   =    d_{g_{\b,\hat\b,\tau}(t)}(x,y)$$ 
   for all $x,y\in \mathbb{S}^{n-1}$ and $t\in (0,S]$, that the limiting continuous map $\Phi_{0,\b,\hat\b,\tau}$ is a bi-H\"older homeomorphism. It then follows from the uniqueness and the stability of the ODE that $\Phi_{0,\b,\hat\b,\tau}$ is indeed a local diffeomorphism outside the finite set $   \mathtt{S}_{\b,\hat\b,\tau}:=\Phi_{0,\b,\hat\b,\tau}^{-1}\left(\mathtt{S}_{\hat\b,\tau}\cup G_{0,\hat\b,\tau}^{-1}(\mathtt{Tips})\right)$.  This completes the proof of Claim \ref{claim:const-RDF}.
\end{proof}

\bigskip

We now use Claim~\ref{claim:const-RDF} to construct Ricci flow coming out of $g_{0,\a}$ for all $\a\in [j^{-1},1]^{n-1}$. We first describe how $g_\a(t)$ is constructed. Define $\mathcal{Z}_k:=[j^{-1},1]^{n-1}\cap V_{k\delta_0}(\hat \b_0)$ which denotes the set of points $k\delta_0$ apart from $\hat\b_0$ in the ``positive orientation".  For any $\a\in [j^{-1},1]^{n-1}$, we let $N_\a$ be the minimum $k$ such that 
$\a\in \mathcal{Z}_{k+1}$ and we let $\hat \a_k:=\hat\b_0+ k \delta_0 \frac{\a-\hat\b_0}{|\a-\hat\b_0|}\in \mathcal{Z}_k$ for $0\leq k\leq N_\a$ so that $\hat\a_k \in V_{\delta_0}(\hat\a_{k-1})$ for $1\leq k\leq N_\a$ and $\hat\a_0=\hat\b_0$. We have already constructed a Ricci flow $g_{0,\tau}(t):=g_{\hat \a_0,\tau}(t)=g_{\hat\b_0,\tau}(t)$ which converges to $g_0(t)$ as $\tau\to 0$.

 If $N_\a=0$, $\a\in \mathcal{Z}_0$ so that we can use stability with reference Ricci flow $g_{0,\tau}(t)$ to find a Ricci-DeTurck flow $\hat g_{\a,\hat\a_0,\tau}(t)$ with respect to $g_{0,\tau}(t)$, a finite set $\mathtt{S}_{\a,\tau}$, a time-dependent diffeomorphism $\Phi_{\a,\hat\a_0,\tau}(t)$ and Ricci flow $g_{\a,\tau}(t):=\left(\Phi_{\a,\hat\a_0,\tau}(t)\right)^* \hat g_{\a,\hat\a_0,\tau}$, for $t\in (0,S]$. If $N_\a\geq 1$, we use induction to find a sequence of reference Ricci flows. 
 
 More precisely,  we use Claim~\ref{claim:const-RDF} by setting $\hat\beta= \hat\a_0$, $G_{0, \hat\a_0,\tau}=\mathrm{Id}$ and $\mathtt{S}_{\hat\a_0,\tau}:=\mathtt{S}_{\hat\beta_0,\tau}$, $g_{\hat \beta, \tau}(t)=g_{0,\tau}(t)=g_{\hat \hat\a_0, \tau}(t)$ to show that there exist a Ricci-DeTurck flow $\hat g_{\hat \a_1,\hat\a_0,\tau}(t)$ with respect to $g_{0,\tau}(t)=g_{\hat \a_0,\tau}(t)$, a finite set $\mathtt{S}_{\hat \a_1,\tau}:=\mathtt{S}_{\hat \a_1,\hat\a_0,\tau}$, and a time-dependent diffeomorphism $\Phi_{\hat\a_1,\hat\a_0,\tau}$ for $t\in (0,S]$ such that $g_{\hat \a_1,\tau}(t)=\left(\Phi_{\hat\a_1,\hat\a_0,\tau}(t)\right)^*\hat g_{\hat \a_1,\hat\a_0,\tau}(t)$ is a Ricci flow for $t\in (0,S]$ and 
\begin{equation}
    (1-\bar\e_0)^{1/3} g_{\hat\a_0,\tau}(t)\leq \hat g_{\hat\a_1,\hat\a_0,\tau}(t)\leq  (1+\bar\e_0)^{1/3} g_{\hat\a_0,\tau}(t)
\end{equation}
for $t\in (0,S]$. Moreover, $\Phi_{\hat\a_1,\hat\a_0,\tau}\to \Phi_{0,1,\tau}$ in $ C^\infty_{loc} (\mathbb{S}^{n-1}\setminus  \mathtt{S}_{\hat\a_1,\hat\a_0,\tau})\cap C^0 (\mathbb{S}^{n-1})$ as $t\to 0$ where $\Phi_{0,1,\tau}$ is homeomorphism on $\mathbb{S}^{n-1}$ and local diffeomorphism on $\mathbb{S}^{n-1}\setminus \mathtt{S}_{\hat \a_1,\hat\a_0,\tau}$. We also have $g_{\hat \a_1,\tau}(t)\to (\Phi_{0,1,\tau})^*\circ (F_{\hat\a_1,\hat\a_0,\tau})^*g_{0,\hat \a_1, \tau}$ in $C^\infty_{loc}(\mathbb{S}^{n-1}\setminus \mathtt{S}_{\hat \a_1,\hat\a_0,\tau})$ as $t\to 0$, where $F_{\hat\a_1,\hat\a_0,\tau}$ is defined as in \eqref{Fdef}. To proceed, we invoke Claim~\ref{claim:const-RDF} again, by setting $\hat\beta= \hat\a_1$, $G_{0, \hat\a_1,\tau}=F_{\hat\a_1,\hat\a_0,\tau}\circ \Phi_{0,1,\tau}$ and $\mathtt{S}_{\hat\a_1,\tau}:=\mathtt{S}_{\hat \a_1,\hat\a_0,\tau}$, to get a Ricci flow $g_{\hat\a_2,\hat\a_1,\tau}(t)$, a Ricci DeTurck flow $\hat g_{\hat\a_2,\hat\a_1,\tau}(t)$,  diffeomorphisms $\Phi_{\hat\a_2,\hat\a_1,\tau}(t)$ converges to a homeomorphism $\Phi_{0,2,\tau}$ as $t\to 0$, and the finite singular set $\mathtt{S}_{\hat \a_2,\hat\a_1,\tau}$, satisfying the corresponding properties in Claim~\ref{claim:const-RDF}.
In the same spirit, we process inductively when $N_\a\geq 1$. We define (finite) sequence of maps and sets inductively as follows: 
\begin{itemize}
\item 
$F_{0,\tau}:=\mathrm{Id}$, $g_{0,k,\tau}:=g_{0,\hat\a_k,\tau}$, $\mathtt{S}_{0,\tau}:=\mathtt{S}_{\hat\a_0,\tau}$; 
\item $F_{k,\tau}:=F_{\hat \a_k,\hat \a_{k-1},\tau}\circ F_{k-1,\tau}$ for $1\leq k\leq N_\a$;
\item $F_{\a,\tau}:=F_{\a,\hat\a_{N_\a},\tau}\circ F_{N_\a-1,\tau}$;
\item $g_{0,\tau}(t):=g_{\hat \a_0,\tau}(t)$.
\end{itemize}

Applying Claim~\ref{claim:const-RDF} inductively, we obtain finite sequences of Ricci flows $g_{k,\tau}(t)$, Ricci-DeTurck flows $\hat g_{k,\tau}(t)$ with respect to $g_{k-1,\tau}(t)$,  time-dependent diffeomorphisms $\Phi_{k,\tau}(t)$ and sets $\mathtt{S}_{k,\tau}$ such that for $1\leq k\leq N_\a$, 
\begin{itemize}
\item $\hat g_{k,\tau}(t):=\hat g_{\hat\a_k,\hat\a_{k-1},\tau}(t)$ is the Ricci-DeTurck flow with respect to $g_{k-1,\tau}(t)$;
\item  $\Phi_{k,\tau}(t):=\Phi_{\hat\a_k,\hat \a_{k-1},\tau}(t)$ is the Ricci-DeTurck ODE solution with respect to $\hat g_{k,\tau}$ and $g_{k-1,\tau}$;
\item $g_{k,\tau}(t)\equiv \left(\Phi_{k,\tau}(t)\right)^*\hat g_{k,\tau}(t)$ is a Ricci flow satisfying (i) and (ii);
\item $(1-\bar\e_0)^{1/3} g_{k-1,\tau}(t)\leq \hat g_{k,\tau}(t)\leq  (1+\bar\e_0)^{1/3} g_{k-1,\tau}(t)$; 
\item  $\mathtt{S}_{k,\tau}:= \mathtt{S}_{\hat\a_k,\hat \a_{k-1},\tau}$;
\item $\Phi_{k,\tau}(t)\to \Phi_{0,k,\tau}$ in $C^\infty_{loc} (\mathbb{S}^{n-1}\setminus \mathtt{S}_{k-1,\tau})\cap C^0 (\mathbb{S}^{n-1})$ as $t\to 0$, {where $\Phi_{0,k,\tau}$ is a homeomorphism on $\mathbb{S}^{n-1}$ and local diffeomorphism on $\mathbb{S}^{n-1}\setminus \mathtt{S}_{k,\tau}$};
\item $g_{k,\tau}(t)\to(\Phi_{0,k,\tau})^*\circ\dots \circ (\Phi_{0,1,\tau})^* \circ (F_{k,\tau})^* g_{0,k,\tau}$ in $C^\infty_{loc} (\mathbb{S}^{n-1}\setminus \mathtt{S}_{k,\tau})$ as $t\to 0$.
\end{itemize} 

And finally we have 
\begin{itemize}
\item $\hat g_{\a,\tau}(t):=\hat g_{\a,\hat\a_{N_\a},\tau}(t)$ is the Ricci-DeTurck flow with respect to $g_{N_\a,\tau}(t)$;
\item $\Phi_{\a,\tau}(t):=\Phi_{\a,\hat \a_{N_\a},\tau}(t)$ is the Ricci-DeTurck ODE solution with respect to $\hat g_{\a,\hat\a_{N_\a},\tau}(t)$ and $g_{N_\a,\tau}(t)$;
\item $g_{\a,\tau}(t)\equiv \left(\Phi_{\a,\tau}(t)\right)^*\hat g_{\a,\tau}(t)$ is a Ricci flow satisfying (i) and (ii);
\item $(1-\bar\e_0)^{1/3} g_{N_\a,\tau}(t)\leq \hat g_{\a,\tau}(t)\leq  (1+\bar\e_0)^{1/3} g_{N_\a,\tau}(t)$; 
\item  $\mathtt{S}_{\a,\tau}:= \mathtt{S}_{\a,\hat \a_{N_\a},\tau}$;
\item $\Phi_{\a,\tau}(t)\to \Phi_{0,\a,\tau}$ in $C^\infty_{loc} (\mathbb{S}^{n-1}\setminus \mathtt{S}_{N_\a,\tau})\cap C^0 (\mathbb{S}^{n-1})$ as $t\to 0$, {where $\Phi_{0,\a,\tau}$ is a homeomorphism on $\mathbb{S}^{n-1}$ and local diffeomorphism on $\mathbb{S}^{n-1}\setminus \mathtt{S}_{\a,\tau}$};
\item $g_{\a,\tau}(t)\to (\Phi_{0,\a,\tau})^*\circ (\Phi_{0,N_\a,\tau})^*\circ\dots \circ (\Phi_{0,1,\tau})^* \circ (F_{\a,\tau})^* g_{0,\a,\tau}$ in $C^\infty_{loc} (\mathbb{S}^{n-1}\setminus \mathtt{S}_{\a,\tau})$ as $t\to 0$.
\end{itemize} 

\bigskip

After passing to further sub-sequence (since there is only finitely many $\hat\a_k$), we define
\begin{itemize}
    \item $g_{\a}(t):=\lim_{\tau_i\to 0}g_{\a,\tau_i}(t)$, $\hat g_\a(t):=\lim_{\tau_i\to 0}\hat g_{\a,\tau_i}(t)$;
    \item $g_{k}(t):=\lim_{\tau_i\to 0}g_{k,\tau_i}(t)$, $\hat g_k(t):=\lim_{\tau_i\to 0}\hat g_{k,\tau_i}(t)$;
    \item $F_{\a,0}:=\lim_{\tau_i\to 0} F_{\a,\tau_i}$
    \item $F_{k,0}:=\lim_{\tau_i\to 0} F_{k,\tau_i}$;
    \item $F_{\a,\hat \a_{N_\a}, 0}:=\lim_{\tau_i\to 0} F_{\a,\hat \a_{N_\a}, \tau_i}$;
    \item $F_{\hat\a_k,\hat\a_{k-1},0}:=\lim_{\tau_i\to 0}F_{\hat\a_k,\hat\a_{k-1},\tau_i}$;
    \item $\Phi_{\a,0}(t):=\lim_{\tau_i\to 0} \Phi_{\a,\tau_i}(t)$
    \item $\Phi_{k,0}(t):=\lim_{\tau_i\to 0} \Phi_{k,\tau_i}(t)$;
    \item $\Phi_{0,k,0}:=\Phi_{k,0}(0)$, $\Phi_{0,\a,0}:=\Phi_{\a,0}(0)$;
    \item $\psi_{0,\hat \b_0}:=\mathrm{id}$;
    \item $\psi_{0,\a}:= \Phi_{0,1,0}\circ \dots \circ \Phi_{0,N_\a,0}\circ \Phi_{0,\a,0}$;
\end{itemize}

Then $g_\a(t)$ and $\psi_{0,\a}$ satisfy (i), (ii) and (v).  Moreover, 
\begin{equation}\label{eqvest}
 (1-\bar\e_0)^{1/3} \cdot \Phi_{\a,0}(t)^* g_{N_\a}(t)\leq g_\a(t)=\Phi_{\a,0}(t)^* \hat g_{\a}(t) \leq (1+\bar\e_0)^{1/3}\cdot \Phi_{\a,0}(t)^* g_{N_\a}(t). 
\end{equation}

The singular set $\mathtt{S}_{\a}$ cannot be defined by taking a limit of $\mathtt{S}_{\a,\tau}$ only, this is because $\mathtt{S}_{\a,\tau}$ only corresponds to the conical singularity with regularized links in which their limit only captures the top stratum in the singularity. But we might define $\mathtt{S}_\a$ by induction, analogous to $\mathtt{S}_{\a,\tau}$, by using the homeomorphism obtained from $\tau_i\to 0$ instead of the intermediate homeomorphism $\Phi_{0,\a,\tau}$, etc. 
More precisely, define
\begin{itemize}
\item $\Sigma_0:=$ singularities of $g_{0,\hat\b_0,0}=$ singularities of $g_{0,\hat\b_0}$;
\item $\Sigma_k:=$ singularities of $(F_{\hat\a_k,\hat\a_{k-1},0})^*g_{0,\hat\a_k,0}$ for $1\le k\le N_\a$;
\item $\Sigma_\a:=$ singularities of $(F_{\a,\hat \a_{N_\a},0})^*g_{0,\a,0}$;
\item $\mathtt{S}_0=\mathtt{S}_{\hat \b_0}=\Sigma_0=$ singularities of $g_{0,\hat\b_0}$;
    \item $\mathtt{S}_1= \Phi_{0,1,0}^{-1}\left(\mathtt{S_0}\cup \Sigma_1\right)$;
    \item $\mathtt{S}_k= \Phi_{0,k,0}^{-1}\left(\mathtt{S}_{k-1}\cup \left(F_{k-1,0}\circ\Phi_{0,1,0}\circ\cdots\circ\Phi_{0,k-1,0}\right)^{-1}\left(\Sigma_k\right)\right)$, for $2\leq k\leq N_\a$;
    \item $\mathtt{S}_\a=\Phi_{0,\a,0}^{-1}\left(\mathtt{S}_{N_\a}\cup \left(F_{N_\a,0}\circ\Phi_{0,1,0}\circ\cdots\circ\Phi_{0,N_\a,0}\right)^{-1}\left(\Sigma_\a\right)\right)$.
    \end{itemize}

\bigskip

Then outside $\mathtt{S}_\a$, $\psi_{0,\a}$ is smooth by local regularity of Ricci-DeTurck flow and satisfies
\begin{equation}
    \begin{split}
        g_\a(0)=(\psi_{0,\a})^* \circ (F_{\a,0})^* g_{0,\a,0}.
    \end{split}
\end{equation}

Recall from \eqref{eqn:initial-gauged-ind} that we have 
\begin{equation}\label{eqn:g_0a0}
    g_{0,\a,0}:=\a_1^2\left(dx_1^2+\sin^2 x_1 \cdot (\bar \psi_{0,(\a_2,...,\a_{n-1})})^* g_{0,(\a_2,...,\a_{n-1})} \right)
\end{equation}
and 
\begin{equation}\label{eqn:F-alp-form}
\begin{split}
   F_{\a,0}(x)=\lim_{\tau\to0}F_{\a,\tau}(x)&=  \lim_{\tau\to0} F_{\a,\hat\a_k,\tau}\circ\dots F_{\hat \a_1,\hat \a_0,\tau}(x)\\
   &=  \lim_{\tau\to0} \left(x_1,  \bar\Psi_{\a,\hat\a_{N_\a}}\circ \dots\bar\Psi_{\hat\a_1,\hat\a_0}(\tau)(x_2,...,x_{n-1}) \right)\\
   &=\left(x_1,  \bar\Psi_{\a,\hat\a_{N_\a}}\circ \dots\bar\Psi_{\hat\a_1,\hat\a_0}(0)(x_2,...,x_{n-1}) \right)\\
   &=\left(x_1,  \bar\psi_{0,(\a_2,...,\a_{n-1})}^{-1}(x_2,...,x_{n-1}) \right)
\end{split}
\end{equation}
thanks to induction hypotheses, so that outside $\mathtt{S}_\a$, we have 
\begin{equation}
\begin{split}
    g_\a(0)&=(\psi_{0,\a})^* \left[ \a_1^2\left(dx_1^2+\sin^2 x_1 \cdot \bar g_{0,(\a_2,...,\a_{n-1})} \right)\right]=(\psi_{0,\a})^* g_{0,\a}.
    \end{split}
\end{equation}

We denote distance $d_{0,\a}:=\lim_{t\to 0^+}d_{g_\a(t)}$ and $d_{0,k}:=\lim_{t\to 0^+} d_{g_k(t)}$ which exist by  \cite[Lemma 3.1]{SimonTopping2021}. We next claim that  $ \mathtt{S}_\a$ is of zero $(n-1)$-Hausdorff measure with respect to distance $d_{0,\a}:=\lim_{t\to 0^+}d_{g_\a(t)}$. By construction, $\Phi_{0,\a,0}$ is a $(1+\bar \e_0)^{1/2}$ bi-Lipschitz map from $(\mathbb{S}^{n-1},d_{0,\a})$  to $(\mathbb{S}^{n-1},d_{0,N_\a})$.
Hence, 
\begin{equation}
    \mathcal{H}^{n-1}_{d_{0,\a}}(\mathtt{S}_\a)\leq    2 \left[\mathcal{H}^{n-1}_{d_{0,N_\a}}(\mathtt{S}_{N_\a})+\mathcal{H}^{n-1}_{d_{0,N_\a}}\left(\left(F_{N_\a,0}\circ\psi_{0,N_\a}\right)^{-1}\left(\Sigma_\a\right) \right)  \right].
\end{equation}

Since each $\Phi_{0,k,0}$ is bi-Lipschitz, by induction it suffices to show that $\mathcal{H}^{n-1}_{d_{0,0}}\left( F_{N_\a,0}^{-1}(\Sigma_\a)\right)=0$. From the precise formula of $F_{N_\a,0}$ from \eqref{eqn:F-alp-form} and its bi-Lipschitz property  from induction hypothesis (iv), it remains to show  $\mathcal{H}^{n-1}_{(F_{N_\a,0})_*d_{0,0}}\left(\Sigma_\a\right)=0$. Using  the warped product structure and \eqref{eqn:F-alp-form}, we have  
\begin{equation}
\begin{split}
&\quad \mathcal{H}^{n-1}_{(F_{N_\a,0})_*d_{0,0}}\left(\Sigma_\a\right)\\
&\leq C\cdot \mathcal{H}^{n-2}_{d_{0,\overline{N_\a}}}\left(\bar{\mathtt{S}}_{\bar{\a}_{N_\a}}\cup \bar{\Psi}_{\bar\a,\bar \a_{N_\a}}^{-1}(\bar{\mathtt{S}}_{\bar \a})\right)+\mathcal{H}^{n-1}_{(F_{N_\a,0})_*d_{0,0}}\left(F^{-1}_{\a, \hat\a_{N_\a},0}(\text{singularities of } g_{0,\a,0})\right)\\
&\leq C\left[ \mathcal{H}^{n-2}_{d_{0,\overline{N_\a}}}\left(\bar{\mathtt{S}}_{\bar{\a}_{N_\a}}\right)+\mathcal{H}^{n-2}_{d_{0,\bar \a}}\left( \bar{\mathtt{S}}_{\bar \a}\right)\right]+C\cdot \mathcal{H}^{n-1}_{d_{0,0}}\left(F^{-1}_{\a,0}(\text{singularities of } g_{0,\a,0})\right),
\end{split}
\end{equation}
where $\bar{\mathtt{S}}_{\bar{\a}}$ denotes the singular set in the earlier dimension with $\bar \a:=(\a_2,...,\a_{n-1})$. By induction hypothesis (iv), the first two terms on the R.H.S. vanish. 

Using \eqref{eqn:F-alp-form} and \eqref{eqn:g_0a0}, the last term in the inequality above can be controlled by 
\begin{equation}
    \begin{split}
        \mathcal{H}^{n-1}_{d_{0,0}}\left(F^{-1}_{\a,0}(\text{singularities of } g_{0,\a,0})\right)&\leq C \cdot   \mathcal{H}^{n-2}_{d_{0,\bar 0}}\left(( {\bar\psi_{0,\bar\a}})(\text{singularities of } (\bar\psi_{0,\bar\a}^* g_{0,\bar\a})\right)\\
        &=C \cdot   \mathcal{H}^{n-2}_{d_{0,\bar 0}}\left({\bar\psi_{0,\bar\a}}( \bar{\mathtt{S}}_{\bar \a})\right)\\
        &\leq C \cdot   \mathcal{H}^{n-2}_{d_{0,\bar \a}}\left(  \bar{\mathtt{S}}_{\bar \a}\right)=0.
    \end{split}
\end{equation}
where $d_{0,\bar 0}$ denotes the distance limit of $\bar g_{(j^{-1},...j^{-1})}(t)$ as $t\to 0$.  In the last inequality, we have used, from the induction assumption (iv), that $\bar\psi_{0,\bar\a}$ is bi-Lipchitz. This proves the claim. By induction on index $k$, this also shows that the limiting Ricci-DeTurck flow, Ricci flow and Ricci-DeTurck ODE are independent of the choice of sub-sequence by the uniquness result in Theorem~\ref{thm:unique-general}.

It remains to construct $\Psi_{\a,\b}$ in (iii). It suffices to construct Ricci-harmonic map heat flow between $g_\a(t)$ and $g_\b(t)$ with good estimates, whenever $|\a-\b|$ is small. We first establish the existence of the Ricci-harmonic map heat flow between $g_\b(t)$ and $g_\a(t)$ for all $\a,\b\in [j^{-1},1]^{n-1}$.

{
\begin{claim}There exists $\delta(n,j)>0$ so that if $|\a-\b|<\delta$ and $\a,\b\in [j^{-1},1]^{n-1}$, then there exists a solution to the Ricci-harmonic map heat flow 
\begin{equation}\label{eqn:harmonicmapheatflow}
    \left\{
    \begin{array}{ll}
        \partial_t \Psi_{\a,\b} =\Delta_{g_\b(t),g_\a(t)}\Psi_{\a,\b} ,\;\;\text{for}\; t\in (0,S];\\[1mm]
         \lim_{t\to 0}\Psi_{\a,\b}(t):=\psi_{0,\a}^{-1}\circ \psi_{0,\b}.
    \end{array}
    \right.
\end{equation}
Moreover, for any $\e>0$,  there exists small positive constant $\delta'(j,n,\varepsilon)>0$ such that if $|\a-\b|<\delta'$, then  $\Psi_{\a,\b}$ satisfies
\begin{equation}
(1-\e)g_\b(t) \leq \Psi_{\a,\b}^*\,g_\a(t)\leq (1+\e) g_\b(t).
\end{equation}
on $(0,S]$.
\end{claim}
}
\begin{proof}[Proof of claim]
We prove it by showing the claim  with induction, that for $\a,\b\in \mathcal{Z}_k$  for all $0\leq k\leq N$, where $N$ is the smallest possible nonnegative integer such that
$[j^{-1},1]^{n-1}=\mathcal{Z}_N$ and $\mathcal{Z}_0\subseteq \mathcal{Z}_1\cdots \subseteq \mathcal{Z}_N$. When $k=0$. We let $t_i\to 0^+$ and consider the equation:
\begin{equation}
    \left\{
    \begin{array}{ll}
        \partial_t \Psi_{\a,\b,i} =\Delta_{g_\b(t),g_\a(t)}\Psi_{\a,\b,i} ,\;\;\text{for}\; t\in [t_i,S] \\[1mm]
         \Psi_{\a,\b,i}(t_i)=\Phi_{\a,0}^{-1}(t_i)\circ \Phi_{\b,0}(t_i).
    \end{array}
    \right.
\end{equation}
Since $\mathbb{S}^{n-1}$ is compact and the initial data is smooth, it admits a short-time solution and remains a diffeomorphism. Furthermore, $\left(\Psi_{\a,\b,i}(t)\right)_*g_\b(t)$ is a solution to the Ricci-DeTurck flow. At $t=t_i$, by \eqref{eqvest} with $N_\a=0$,
\begin{equation}
    \begin{split}
       (1-\bar\e_0)  g_{\b}(t_i)
       &\leq (1-\bar\e_0)^{1/3} (\Phi_{\b,0}(t_i))^* g_{\hat \a_0}(t_i)\\
       &\leq  \left(\Phi_{\a,\b,i}(t_i)\right)^* g_\a(t_i)=(\Phi_{\b,0}(t_i))^* \hat g_{\a}(t_i)\\
        &\leq (1+\bar\e_0)^{1/3} (\Phi_{\b,0}(t_i))^* g_{\hat \a_0}(t_i)\\
        &\leq (1+\bar\e_0)  g_{\b}(t_i)
    \end{split}
\end{equation}
so that we might apply Theorem~\ref{thm:weak-stable-RDF} using the fact that the push-forward is also a Ricci-DeTurck flow so that 
\begin{equation}\label{eqn:harmon-map-esti}
    \begin{split}
       (1-\Lambda \bar\e_0)  g_{\b}(t)
       \leq  (\Psi_{\a,\b,i}(t))^* g_{\a}(t)
        \leq (1+\Lambda\bar\e_0)  g_{\b}(t)
    \end{split}
\end{equation}
for all $t\in [t_i,S\wedge T_{max}]$. This implies high order regularity of $\Psi_{\a,\b,i}$ by \cite[Proposition 2.1]{Lee2025} and $T_{max}\geq S$. In particular, we might pass $i\to +\infty$ and obtain a solution $\Psi_{\a,\b}$ to \eqref{eqn:harmonicmapheatflow}. The convergence as $t\to 0$ follows from a similar argument in \eqref{eqn:degree}:
\begin{equation}
    \begin{split}
        d_{g_\a(s)} \left(\Psi_{\a,\b}(x,t),\Psi_{\a,\b}(x,s) \right)&\leq C\sqrt{t}
    \end{split}
\end{equation}
for $0<s<t$. By letting $s\to 0$ using \cite[Lemma 3.1]{SimonTopping2021}, we show the convergence to the initial data as $t\to 0$. Moreover, the convergence is smooth outside $\mathtt{S}_\b\cup \Psi_{\a,\b}(0)^{-1}(\mathtt{S}_\a)$ by interior estimate of Ricci-DeTurck flow.  This proves the existence of solution to \eqref{eqn:harmonicmapheatflow} for the case of $k=0$.

We now want to show that \eqref{eqn:harmon-map-esti} can be improved after taking $i\to +\infty$. Since the initial data is smooth outside $ \mathtt{S}_\b\cup \Psi_{\a,\b}(0)^{-1}(\mathtt{S}_\a)$, and the push-forward of the Ricci-harmonic map heat flow is a solution to the Ricci-DeTurck flow, by Theorem~\ref{thm:unique-general} it suffices to show that the harmonic map heat flow is initially of better estimate. Since outside $\mathtt{S}_\b\cup \Psi_{\a,\b}(0)^{-1}(\mathtt{S}_\a)$, we have 
\begin{equation}\label{eqn:improved-initial}
    \begin{split}
    (\Psi_{\a,\b}(0))^* g_\a(0)-g_\b(0)&=(\psi_{0,\b})^* \left(g_{0,\a}-g_{0,\b}\right).
    \end{split}
\end{equation}
This follows that for any $\e>0$, there exists  $\delta'(j,n,\varepsilon)>0$ such that if $|\a-\b|<\delta'$ and $\a,\b\in \mathcal{Z}_0$, then 
\begin{equation}
(1-\e) g_\b(0)\leq (\Psi_{\a,\b}(0))^* g_\a(0)\leq (1+\e) g_\b(0)
\end{equation}
outside $\mathtt{S}_\b\cup \Psi_{\a,\b}(0)^{-1}(\mathtt{S}_\a)$. 
We remark here that $\mathtt{S}_\b\cup \Psi_{\a,\b}(0)^{-1}(\mathtt{S}_\a)$ is of vanishing $(n-1)$-Hausdorff measure with respect to $d_{0,\b}$, by the same reasoning as showing $\mathcal{H}_{d_{0,\a}}(\mathtt{S}_\a)=0$. Therefore, \begin{equation}
(1-\Lambda\e) g_\b(t)\leq (\Psi_{\a,\b}(t))^* g_\a(t)\leq (1+\Lambda\e) g_\b(t)
\end{equation}
for $t\in (0,S]$ by Theorem~\ref{thm:unique-general}. This proves the case of $k=0$.
\bigskip

We assume the claim is true on $\mathcal{Z}_k$ for $0\leq k\leq N$. We now prove the case when $\a,\b\in \mathcal{Z}_{k+1}$. We might assume $N_\a,N_\b=k+1$. By construction, $g_\a(t)=(\Phi_{\a,0}(t))^*\,\hat g_{\a,\hat\a_{k+1}}(t)$ and $g_\b(t)=(\Phi_{\b,0}(t))^*\,\hat g_{\b,\hat\b_{k+1}}(t)$ where $\hat \a_{k+1},\hat \b_{k+1}\in \mathcal{Z}_k$ so that induction hypothesis implies the existence of $ \Psi_{\hat\a_{k+1},\hat\b_{k+1}}(t)$ on $(0,S]$ since $|\hat\a_{k+1}-\hat\b_{k+1}|\leq |\a-\b|$ where we might assume to be small. 

We let $t_i\to 0^+$ and consider the equation:
\begin{equation}
    \left\{
    \begin{array}{ll}
        \partial_t \Psi_{\a,\b,i} =\Delta_{g_\b(t),g_\a(t)}\Psi_{\a,\b,i} ,\;\;\text{for}\; t\in [t_i,S] \\[1mm]
         \Psi_{\a,\b,i}(t_i)=\Phi_{\a,0}^{-1}(t_i)\circ  \Psi_{\hat\a_{k+1},\hat\b_{k+1}}(t_i)\circ  \Phi_{\b,0}(t_i).
    \end{array}
    \right.
\end{equation}
Thanks to the induction hypothesis, at $t=t_i$ we have 
\begin{equation}
    \begin{split}
    \left(\Psi_{\a,\b,i}(t_i)\right)^* g_\a(t_i)&=(\Phi_{\b,0}(t_i))^*\circ ( \Psi_{\hat\a_{k+1},\hat\b_{k+1}}(t_i))^* \hat g_{\a}(t_i) 
    \\
        &\leq (1+\bar\e_0)^{1/3}\cdot (\Phi_{\b,0}(t_i))^*\circ ( \Psi_{\hat\a_{k+1},\hat\b_{k+1}}(t_i))^* g_{\hat\a_{k+1}}(t_i)\\
       &\leq (1+\bar\e_0)^{2/3}\cdot (\Phi_{\b,0}(t_i))^* g_{\hat\b_{k+1}}(t_i)\\
       &\leq (1+\bar\e_0)  (\Phi_{\b,0}(t_i))^* \hat g_{\b}(t_i)\\
        &= (1+\bar\e_0)  g_{\b}(t_i).
    \end{split}
\end{equation}
The lower bound is similar, so that 
\begin{equation}
   (1-\bar \e_0) g_\b(t_i)\leq   \left(\Psi_{\a,\b,i}(t_i)\right)^* g_\a(t_i)\leq (1+\bar \e_0) g_\b(t_i)
\end{equation}
and hence, Theorem~\ref{thm:weak-stable-RDF} can be applied to obtain a solution to \eqref{eqn:harmonicmapheatflow} on $\mathbb{S}^{n-1}\times (0,S]$, as in the case of $k=0$, with
\begin{equation}\label{eqn:rough-harm}
   (1-\Lambda\bar \e_0) g_\b(t)\leq   \left(\Psi_{\a,\b}(t)\right)^* g_\a(t)\leq (1+\Lambda\bar \e_0) g_\b(t)
\end{equation}
for $t\in (0,S]$. The initial data is given by
\begin{equation}\label{eqn:ini}
\begin{split}
    \lim_{t\to 0}\Psi_{\a,\b}(t)&=\Phi_{0,\a,0}^{-1}\circ  \Psi_{\hat\a_{k+1},\hat\b_{k+1}}(0)\circ  \Phi_{0,\b,0}\\
    &=\Phi_{0,\a,0}^{-1}\circ \psi_{0,\hat\a_{k+1}}^{-1}\circ \psi_{0,\hat\b_{k+1}}\circ  \Phi_{0,\b,0}\\
    &=\psi_{0,\a}^{-1}\circ \psi_{0,\b}.
    \end{split}
\end{equation}
This proves the existence of solution $\Psi_{\a,\b}$ to \eqref{eqn:harmonicmapheatflow} for the case of $k+1$. It remains to show the improved estimate. This follows from the same argument as in the case of $k=0$, using \eqref{eqn:ini} that outside $\mathtt{S}_\b\cup \Psi_{\a,\b}(0)^{-1}(\mathtt{S}_\a)$, \eqref{eqn:improved-initial} holds and hence Theorem~\ref{thm:unique-general} applies to show that for any $\e>0$, there exists $\delta'(n,j,\varepsilon)>0$ such that if  $|\a-\b|<\delta'$ and $\a,\b\in \mathcal{Z}_{k+1}$, then 
\begin{equation}
(1-\Lambda\e) g_\b(t)\leq (\Psi_{\a,\b}(t))^* g_\a(t)\leq (1+\Lambda\e) g_\b(t)
\end{equation}
for $t\in (0,S]$. This proves the case of $k+1$ and hence completes the proof by induction. 
\end{proof}
(iii) follows from the claim, and this finishes the proof of Theorem~\ref{thm:const-2}. 
\end{proof}

\bigskip

We end this subsection by presenting another application of our stability result. 
Using the method in the proof of Theorem~\ref{thm:const-2}, we are able to show that the expander constructed by Deruelle \cite{Deruelle2016} is  stable with respect to $L^\infty$ variation of the links. Quantitatively, we have the following:

\begin{thm}\label{thm:stable-expander}
For any $n\in \mathbb{N}$ and $v>0$, there exists $\delta_0(n,v),\Lambda(n,v)>0$ such that the following holds. Suppose $X_i:=(\mathbb{S}^{n-1},h_i)$ where $i=1,2$ are two smooth manifolds such that 
\begin{enumerate}
    \item[(i)] $\mathrm{Vol}_{h_1}(X_1)\geq v>0$; 
    \item[(ii)] $\mathrm{Rm}(h_i)>1$ for $i=1,2$;
    \item[(iii)] $(1-\e) h_2\leq h_1\leq (1+\e)h_2$ for some $0< \e<\delta_0$.
\end{enumerate}
 Let $(M_i,g_i,f_i,p_i)$ be the unique expanding Ricci solitons coming out of $C(X_i)$, constructed by Deruelle \cite{Deruelle2016} for $i=1,2$, then there exists a diffeomorphism $\Phi:M_1\to M_2$ such that 
\begin{equation}
    (1-\Lambda\e) g_1\leq \Phi^*g_2 \leq (1+\Lambda \e)g_1 ,\;\;\text{on}\; M_1.
\end{equation}
Furthermore for any $k\in \mathbb{N}$, there exists $C_k(n,v)>0$ so that $$\sup_{M_1}|\nabla^{k,g_1}(\Phi^*g_2-g_1)|\leq C_k\e.$$
\end{thm}
\begin{proof}
Since the proof follows from modification of that in claim~\ref{claim:const-RDF}, we only give a sketch. It suffices to show the quantitative stability. To do this, we construct $(M_2,g_2,f_2,p_2)$ using an alternative way. 

We let $g_1(t)$ be the Ricci flow associated to $(M_1,g_1,f_1,p_1)$. By identifying the distance limit of $d_{g_1(t)}$ as $t\to 0$ with $C(X_1)$ using \cite[Lemma 3.1]{SimonTopping2021}, we might assume $M_1\equiv C(X_1)$ where $g_{2,0}:=dr^2+r^2 h_2$ is also a $L^\infty$ metric on $M_1$ such that 
\begin{equation}
    (1-\e) g_{2,0}\leq  dr^2+r^2h_1\leq (1+\e) g_{2,0}
\end{equation}
on $M_1$. We follow the method in the proof of  claim~\ref{claim:const-RDF}. For $\sigma\to 0^+$, we consider the $L^\infty$ metric 
\begin{equation}
    g_{2,0,\sigma}:=\phi_\sigma g_1(\sigma)+ (1-\phi_\sigma) g_{2,0}
\end{equation}
where $\phi_\sigma(x):=\phi( d_{g_1(\sigma)}(x,o_{tips})/\sigma^{1/4})$. 

As $\sigma\to 0$, $g_{2,0,\sigma}\to g_{2,0}$ smoothly outside the tips. Arguing in a similar way as the construction of $\hat g_{\b,\hat \b,\tau,\sigma}$ in the proof of claim~\ref{claim:const-RDF},  this will enable us to apply Theorem~\ref{thm:weak-stable-RDF} (with rescaling) to $g_{2,0,\sigma}$ and obtain a Ricci-DeTurck flow $\hat g_{2,\sigma}(t)$ with respect to $g_1(\sigma+t)$ starting from $g_{2,0,\sigma}$ such that 
\begin{equation}\label{eqn:Lip-sigma}
   \left (1-\Lambda\e\right) g_{1}(\sigma+t)\leq \hat g_{2,\sigma}(t)\leq \left(1+\Lambda\e\right) g_{1}(\sigma+t)
\end{equation}
on $M_1\times (0,+\infty)$, for all $\sigma\to0^+$. By passing $\sigma\to 0$ subsequently and Lemma~\ref{lma:RDF-higher-ord}, we obtain a smooth solution to the Ricci-DeTurck flow $\hat g_2(t)$ with respect to $g_1(t)$ on $M_1\times (0,+\infty)$ satisfying
\begin{equation}\label{eqn:Lip-0}
   \left (1-\Lambda\e\right) g_{1}(t)\leq \hat g_{2}(t)\leq \left(1+\Lambda\e\right) g_{1}(t).
\end{equation}
Moreover  by \cite[Proposition 2.2]{ChuLee2025}, $\hat g_2(t)$ is smooth up to $t=0$ with $\hat g_2(0)=g_{2,0}$ outside the tips.

Now we might use the argument in the proof of claim~\ref{claim:const-RDF} to show that indeed $0< \mathrm{Rm}(\hat g_2(t))\leq \a t^{-1}$ for some $\a(n,v)>0$, using crucially the fact that $\delta_0$ is small and the singularity is isolated.  By considering the Ricci-DeTurck ODE as in \eqref{eqn:RDF-ODE}, we obtain a Ricci flow $\tilde g_2(t)$  coming out of $C(X_2)$ in the pointed Gromov-Hausdorff sense, thanks to \cite[Lemma 3.1]{SimonTopping2021}. By \cite[Theorem 6.1]{ChanLeePeachey2024}, $\tilde g_2(1)$ is a gradient expanding Ricci soliton coming out of $C(X_2)$. Since the link $X_2$ is smooth, by \cite[Theorem 1.3]{Deruelle2016}, the new solution $(M_1,\tilde g_2(1))$ is isometric to the expander $(M_2,g_2)$ constructed by Deruelle. The claimed estimates follow from \eqref{eqn:Lip-0} and Lemma~\ref{lma:RDF-higher-ord}.
\end{proof}

\subsection{Preservation of symmetry group}

We will show that the symmetry group of singular metric  will be preserved under Ricci flow smoothing, see also \cite{Brendle2013, LuTian} for the case of smooth initial data. 
\begin{prop}\label{prop:sym-group}
Suppose $g_0$ is a $L^\infty$ metric on  a compact smooth manifold $M^n$ and $g(t),t\in (0,T]$ is the Ricci flow on $M$ such that 
\begin{enumerate}
\item $|\Rm(g(t))|\leq \a t^{-1}$;
\item $\mathrm{inj}(g(t))\geq \sqrt{\a^{-1}t}$;
\item $\mathrm{Rm}(g(t))\geq 0$;
\item $d_{g(t)}\to d_{0}$ as $t\to 0$, for some distance function $d_0$;
\item $g(t)\to g_0$ in $C^\infty_{loc}(M\setminus \mathcal{S})$ as $t\to 0^+$, for some set $\mathcal{S}$ with upper Minkowski dimension $\leq n-2$ with respect to $d_0$.
\end{enumerate}
If $X_0$ is a (bounded) vector field on $M$  with respect to $g_0$ such that $\mathcal{L}_{X_0}g_0=0$ on $M\setminus \mathcal{S}$, then $\mathcal{L}_{X_0}g(t)=0$ on $M^n\times (0,T]$.
\end{prop}
\begin{proof}
Let $g(t)$ be the Ricci flow on $M^n\times (0,T]$. Let $t_i\to 0^+$ and we consider the following evolution of vector field:
\begin{equation}
\left\{
\begin{array}{ll}
\displaystyle\frac{\partial}{\partial t}X_i=\Delta_{g(t)} X_i +\Ric(X_i);\\[3mm]
X_i(t_i)=X_0
\end{array}
\right.
\end{equation}
on $M^n\times [t_i,T]$. Since it is a linear PDE, its solution exists up to $t=T$. We first claim that $X_i$ sub-converges to some time-evolving vector field on $(0,T]$. It suffices to show the $L^\infty$ estimate of  $X_i$, the smooth convergence then follows from the parabolic Schauder estimates.
\begin{claim}\label{claim:X-bdd}
There is a constant $C_0>0$ such that for all $i\in \mathbb{N}$, $|X_i|\leq C_0$ on $[t_i,T]$.
\end{claim}
\begin{proof}
We compute the evolution equation of $|X_i|$:
\begin{equation}
\begin{split}
 \heat |X_i|^2=-2|\nabla X_i|^2\leq 0.
\end{split}
\end{equation}
In particular, the maximum principle shows that $|X_i|_{g(t)}\leq \sup_{M} |X_0|_{g(t_i)}$ on $[t_i,T]$. Since the Ricci flow is non-increasing using $\Ric\geq 0$, it follows from the construction that $g(t_i)\leq Cg_0$ almost everywhere. Since $g_0$ is a $L^\infty$ metric on $M^n$, the conclusion follows.
\end{proof}

On the other hand, by the Ricci identity we have
\begin{equation}
\begin{split}
\heat |\nabla X|^2&\leq -2|\nabla^2 X|^2+C_n|\Rm| |\nabla X|^2+C_n|X||\nabla \Rm||\nabla X|.
\end{split}
\end{equation}

By the standard Bernstein trick, curvature estimates and Claim~\ref{claim:X-bdd}, it follows that there exists $C_1>0$ such that  $|\nabla X_i|\leq C_1 (t-t_i)^{-1/2}$ for all $t\in [t_i,T]$. Using this, we obtain by letting $t_i\to 0$ a vector field $X$ on $M^n\times (0,T]$ such that $\partial_t X=\Delta_{g(t)} X+\Ric(X)$ and $X(t)\to X_0$ on $M\setminus \mathcal{S}$ as $t\to 0$. In particular, the Lie derivative $\eta:=\mathcal{L}_X g$  satisfies: $
\partial_t \eta=\Delta_L \eta $
where $\Delta_L$ is the Lichnerowicz Laplacian with respect to $g(t)$, for instances see the proof of \cite[Theorem 4.1]{Brendle2013}. 

We apply a trick in \cite[Lemma 3.2]{ChanLeePeachey2024} that if $\varphi$ denotes either the negative part of the lowest  eigenvalues $\eta_-$ or the positive part of the highest  eigenvalues  $\eta_+$, then it satisfies $\heat \varphi\leq R\cdot \varphi$ in the sense of barrier. Particularly, we have 
\begin{equation}
\varphi(x,t)\leq \int_{M^n} G(x,t;y,s) \, \varphi(y,s)\,d\mathrm{vol}_{g(s),y}
\end{equation}
for $0<s<t\leq T$. Here $G(x,t;y,s)$ is the heat kernel with respect to $\partial_t-\Delta_{g(t)}-R_{g(t)}$. Furthermore, $\varphi(0)=0$ on $M\setminus \mathcal{S}$ and $\varphi\leq Ct^{-1/2}$ on $(0,T]$ since $\mathcal{L}_{X_0}g_0=0$ and $|\eta|\leq C_n |\nabla X|$. 

By assumption on the Minkowski dimension, we can apply the same argument as the derivation of $\mathrm{Rm}(g(t))\geq 0$ in the proof of Theorem~\ref{thm:RF-existence} to show that $\varphi(t)\equiv 0$ for $t\in (0,T]$. That said $\eta(t)=\mathcal{L}_{X(t)} g(t)\equiv 0$. Since
\begin{equation}
\mathrm{div}(\eta)-\frac12 \nabla \tr \eta=\Delta X+\Ric(X),
\end{equation}
we see that $\partial_t X=0$ and thus $X(t)\equiv X_0$ and $\mathcal{L}_{X_0}g(t)\equiv 0$ on $(0,T]$. This completes the proof of Proposition \ref{prop:sym-group}.
\end{proof}

\begin{proof}[Proof of Theorem~\ref{thm: links}]
By Theorem~\ref{thm:const-2}, the Ricci flow smoothing of $g_{0,\b}$ exists. By \cite[Theorem 1.6]{DeruelleSchulzeSimon2022}, we might assume $g(t)$ converges to $g_{0,\b}$ as $t\to 0$, outside a set of upper Minkowski dimension $\leq n-2$. By Proposition~\ref{prop:sym-group}, the Ricci flow smoothing preserved the symmetry from the initial singular metrics. By choosing $t$ sufficiently small, the conclusion follows from Hamilton-Perelman distance distortion, for instances see \cite[Lemma 3.1]{SimonTopping2021}. The family $g_\b(t)$ is continuous by (iii) and Lemma~\ref{lma:RDF-higher-ord}, using the fact that the push-forward of Ricci flow by the Ricci-harmonic map heat flow is a solution to the Ricci-DeTurck flow for $n-1\geq 3$. For $n-1=2$, the higher order stability follows from the explicit ODE construction in \cite{Lai2024}. 
\end{proof}

\section{Steady solitons with prescribed eigenvalues}

In this section, we prove our main Theorem. This immediately implies Theorem \ref{thm:non-collapsed}.

We define the convex hulls which we will frequently use in the proof.
 First, let
    $$\Omega:=\{(\beta_1,\cdots,\beta_{n-1})\in[0,1]^{n-1}:\beta_i=0 \textnormal{ for some }i=1,\cdots,n-1\}.$$
For any $k=0,\cdots,n-2$ and $1\le i_1<\cdots<i_{k+1}\le n-1$, we define 
    \begin{align}\label{Omega-k}
            \Omega_{k}(i_1,\cdots,i_{k+1}):=\{\beta\in\Omega:\beta_i=1\textnormal{ if }i\neq i_j,j=1,\cdots,k+1\}.
    \end{align}
In particular,  $\Omega_{0}(k)=(1,\cdots,1,\overset{k}{0},1,\cdots,1)$ for any $k=1,\cdots,n-1$.
We can view $\Omega$ as an $(n-2)$-dimensional convex hull spanned by vertices $\{\Omega_0(k)\}_{k=1}^{n-1}$, and the $k$-faces are given by $\Omega_{k}(i_1,\cdots,i_{k+1})$.
We also have
\begin{equation*} \label{eqn:partial-Omega-k}
\begin{split}
\partial\Omega_{k}(i_1,\cdots,i_{k+1})&={\Omega_{k-1}(\hat{i_1},i_2,\cdots,i_{k+1})\cup\cdots\cup\Omega_{k-1}(i_1,i_2,\cdots,\hat{i_{k+1}})},
\end{split}
\end{equation*}
where $\hat{i_j}$ means skipping $i_j$.

Next, we let 
\[\Delta^*:=\{(\lambda_1,\cdots,\lambda_{n-1})\in[0,1]^{n-1}: \lambda_1+\lambda_2+\cdots+2\lambda_{n-1}=1\}.\]
For any $k\ge0$, and $1\le i_1<\cdots<i_{k+1}\le n-1$, we define $i_{k+2}=n$ by convention, and let
\[\Delta_{k}^*(i_1,\cdots,i_{k+1}):=\{(\lambda_1,\cdots,\lambda_{n-1})\in\Delta^*:\lambda_{i_{j}}=\cdots=\lambda_{i_{j+1}-1},\,j=1,\cdots,k+1\},\]
if $i_1=1$; {and if $i_1>1$ we furthermore assume $\lambda_1=\cdots=\lambda_{i_1-1}=0$.} 

Now we consider 
$$\Delta:=\{ (\lambda_1,\cdots,\lambda_{n-1})\in \Delta^*: \lambda_1\leq\cdots\leq\lambda_{n-1}\}.$$
Then it is easy to see $\Delta$ is an $(n-2)$-dimensional convex hull spanned by the $(n-1)$ vertices 
$$\Delta_0(k) :=\left\{ (0,\cdots,0,\overset{k}{\tfrac{1}{n+1-k}},\cdots,\tfrac{1}{n+1-k})\right\},$$ 
and the $k$-faces are given by 
$$\Delta_{k}(i_1,\cdots,i_{k+1}):=\Delta_{k}^*(i_1,\cdots,i_{k+1})\cap \Delta.$$ We also have
\begin{equation} \label{eqn:partial-Omega-k-1}
\begin{split}
        \partial\Delta_{k}(i_1,\cdots,i_{k+1})&={\Delta_{k-1}(\hat{i_1},i_2,\cdots,i_{k+1})\cup\cdots\cup\Delta_{k-1}(i_1,i_2,\cdots,\hat{i_{k+1}})}.
\end{split}
\end{equation}

\begin{figure}[h]
\centering
\begin{minipage}[b]{0.45\textwidth}
    \centering
       \begin{tikzpicture}[line cap=round, line join=round, >=stealth, thick, scale=2]\label{f:Omega}

\newcommand{\pt}[4]{%
  \pgfmathsetmacro{\X}{-0.5*#1 + 1.0*#2} 
  \pgfmathsetmacro{\Y}{-0.3*#1 + 1.2*#3} 
  \coordinate (#4) at (\X,\Y);
}

\pt{0}{0}{0}{O}
\pt{1}{0}{0}{X}
\pt{0}{1}{0}{Y}
\pt{0}{0}{1}{Z}
\pt{1}{1}{0}{XY}
\pt{1}{0}{1}{XZ}
\pt{0}{1}{1}{YZ}
\pt{1}{1}{1}{XYZ}

\fill[gray!25,opacity=0.7] (O) -- (X) -- (XY) -- (Y) -- cycle;   
\fill[gray!50,opacity=0.7] (O) -- (X) -- (XZ) -- (Z) -- cycle;   
\fill[gray!70,opacity=0.7] (O) -- (Y) -- (YZ) -- (Z) -- cycle;   

\draw (X) -- (XY) -- (Y) -- (YZ) -- (Z) -- (XZ) -- (XYZ) -- (XY);
\draw (YZ) -- (XYZ);
\draw (XZ) -- (X);

\draw[dashed] (O) -- (X);
\draw[dashed] (O) -- (Y);
\draw[dashed] (O) -- (Z);

\draw[line width=1.8pt] (X) -- (XZ);   
\draw[line width=1.8pt] (Z) -- (XZ);  
\draw[line width=1.8pt] (Y) -- (YZ);   
\draw[line width=1.8pt] (Z) -- (YZ);  
\draw[line width=1.8pt] (X) -- (XY);   
\draw[line width=1.8pt] (Y) -- (XY);

\draw[->] (X) -- ++(-0.5/1.2,-0.3/1.2) node[below left] {$x$};  
\draw[->] (Y) -- ++(0.5,0)     node[right] {$y$};       
\draw[->] (Z) -- ++(0,0.5)   node[above] {$z$};       

\node[circle,fill,inner sep=1.8pt] at (XZ) {};
\node[circle,fill,inner sep=1.8pt] at (YZ) {};
\node[circle,fill,inner sep=1.8pt] at (XY) {};

\node at ($(XZ)+(-0.15,0)$) {$B$};
\node at ($(YZ)+(0.15,0.05)$) {$A$};
\node at ($(XY)+(0.1,-0.12)$) {$C$};
    \end{tikzpicture}

    \caption{}\label{fig:Omega}
\end{minipage}%
\hfill
\begin{minipage}[b]{0.45\textwidth}
    \centering
      \begin{tikzpicture}[line cap=round, line join=round, >=stealth, thick, scale=2]

\newcommand{\pt}[4]{%
  \pgfmathsetmacro{\X}{-0.5*#1 + 1.0*#2} 
  \pgfmathsetmacro{\Y}{-0.3*#1 + 1.2*#3} 
  \coordinate (#4) at (\X,\Y);
}

\pt{0}{0}{0}{O}
\pt{1}{0}{0}{X}
\pt{0}{1}{0}{Y}
\pt{0}{0}{1/2}{Z}
\pt{1/2}{1/2}{0}{XY}
\pt{1/3}{0}{1/3}{XZ}
\pt{0}{1/3}{1/3}{YZ}
\pt{1/4}{1/4}{1/4}{XYZ}

\fill[gray!60,opacity=0.7] (XYZ) -- (YZ) -- (Z) -- cycle;   

\draw (X) -- (XY) -- (Y) -- (YZ) -- (Z) -- (XZ) -- (XYZ) -- (XY);
\draw (YZ) -- (XYZ);
\draw (XZ) -- (X);
\draw (XYZ) -- (X);
\draw (XYZ) -- (Y);
\draw (XYZ) -- (Z);

\draw[dashed] (O) -- (X);
\draw[dashed] (O) -- (Y);
\draw[dashed] (O) -- (Z);

\draw[line width=1.8pt] (X) -- (YZ);   
\draw[line width=1.8pt] (Z) -- (XY);  
\draw[line width=1.8pt] (Y) -- (Z);   

\draw[->] (X) -- ++(-0.5/1.2,-0.3/1.2) node[below left] {$x$};  
\draw[->] (Y) -- ++(0.5,0)     node[right] {$y$};       
\draw[->] (Z) -- ++(0,0.5)   node[above] {$z$};       

\node[circle,fill,inner sep=1.8pt] at (Z) {};
\node[circle,fill,inner sep=1.8pt] at (YZ) {};
\node[circle,fill,inner sep=1.8pt] at (XYZ) {};

\node at ($(XYZ)+(-0.05,-0.15)$) {$A$};
\node at ($(YZ)+(0.15,0.05)$) {$B$};
\node at ($(Z)+(0.1,0.1)$) {$C$};

\end{tikzpicture}
         \caption{}\label{fig:Delta}
\end{minipage}
\end{figure}

Figure \ref{fig:Omega} illustrates $\Omega$ when $n=4$, where $A=\Omega_0(1)$, $B=\Omega_0(2)$, $C=\Omega_0(3)$, the three segments connecting $AB,AC,BC$ are $\Omega_1(1,2),\Omega_1(1,3),\Omega_1(2,3)$, and the union of the three shaded faces on the coordinates planes is $\Omega_2(1,2,3) \equiv\Omega$. 
We can identify $\Omega$ as the triangle spanned by $A,B,C$ under a face-preserving diffeomorphism.

Figure \ref{fig:Delta} illustrates $\Delta$ when $n=4$, where $\Delta^*$ is the triangle spanned by the three points on the coordinates, $\Delta$ is the shaded triangle spanned by $A=\Delta_0(1),B=\Delta_0(2),C=\Delta_0(3)$, and $\Delta_1^*(1,2),\Delta_1^*(2,3),\Delta_1^*(1,3)$ are the three thicken edges.
In the proof of Theorem \ref{thm:soliton-polyhedron}, the vertices $A,B,C$ corresponds to the three steady solitons: 4D Bryant soliton, $\mathbb R\times\textnormal{3D Bryant soliton}$, and $\mathbb R^2\times\textnormal{2D Cigar soliton}$.

Now we prove Theorem \ref{thm:soliton-polyhedron}.

\begin{proof}[Proof of Theorem \ref{thm:soliton-polyhedron}]

    By Theorem \ref{thm: links}\eqref{i:Rm ge 1}, for any $\mathbf x\in\Omega$, the metric $\mathcal S_j(\mathbf x)$ on $\mathbb{S}^{n-1}$ satisfies $\Rm\geq 1$. 
    So by \cite{Deruelle2016} we can define a smooth map $\mathcal E_j$ from $\Omega$ to the space of expanding gradient solitons with $\Rm\geq 0$ and $R=1$ at the critical point, so that $\mathcal E_j(\mathbf x)$ is the unique expanding soliton on $\mathbb R^n$ asymptotic to the metric cone over $\mathcal S_j(\mathbf x)$. 
Since the expanding soliton $\mathcal E_j(\mathbf x)$ is $O(2)$-symmetric by Theorem \ref{thm: links}\eqref{i:symmetry}, it follows that the eigenvalues of the Ricci curvature at the critical point are $\lambda_1,\cdots,\lambda_{n-1}=\lambda_n$, and $(\lambda_1,\cdots,\lambda_{n-1})\in\Delta^*$. 
This induces a smooth map $$\Lambda_j^*:\Omega\to\Delta^*$$ by composing $\mathcal E_j|_\Omega$ with the smooth map from the expanding solitons to the eigenvalue vector $(\lambda_1,\cdots,\lambda_{n-1})\in\Delta^*$.
 
\begin{claim}\label{claim:phi_j}
Let $\Sigma=\Delta^*_{n-3}(2,\cdots,n-1)$.
There exist $C>0$, a sequence of $\eps_j\to0$ and continuous maps $\Phi_j:\Delta^*\to \Delta^*$
such that
\begin{enumerate}
\item\label{two} $\Phi_j=\textnormal{id}$ on {$\Delta^*\setminus U(\Sigma,\eps_j)$,}
\item\label{three} $\Phi_j\big( U(\Sigma,\eps_j/C)\big)\subset \Sigma$, $\Phi_j|_{\Sigma}=\textnormal{id}$, and $|\Phi_j-\textnormal{id}|_{C^0}\le C\eps_j$,
\item\label{Omega-to-Delta*} $\Phi_j\big(\Delta^*_{k}(i_1,\cdots,i_{k+1})\big)\subseteq \Delta^*_{k}(i_1,\cdots,i_{k+1})$, for any $k$ and $i_1,\cdots,i_{k+1}$,
\end{enumerate}
where $U(V,r)$ denotes the $r$-neighborhood of a subset $V$ {in $\Delta^*$}.
\end{claim}

\begin{proof}
We omit the index $j$ for convenience, and let $\eps>0$ denote a generic small constant.
First, we observe the following fact: Given any $\ell$-simplex $\sigma$, and a subsimplex $\sigma_0$ of dimension smaller than $\ell$, assume there is a continuous map 
$\phi:\partial\sigma\to \partial\sigma$ such that $\phi|_{\sigma_0}=\mathrm{id}$ and satisfies assertion (1)(2) when replacing $\Sigma$ by $\sigma_0$.
Then we can find a continuous map $\tilde\phi:\sigma\to\sigma$ such that $\tilde\phi|_{\partial \sigma}=\phi$, and satisfies the same properties with (1)(2) with $\Sigma$ replaced by $\sigma$.

First, we take a triangularization of $\Delta^*$, so that each subsimplex of this triangularization is contained in some $\Delta^*_k(i_0,\cdots,i_{k+1})$.
We call the union of all $k$-dimensional subsimplices a $k$-skeleton $\mathsf{S}_k$.
In the following we define a continuous map $\phi_k$ on $\mathsf{S}_k$ by induction, which satisfies assertion (1)(2) with $\Sigma$ replaced by $\Sigma\cap\mathsf S_{k}$.

First, for $k=0$, let $\phi_0=\textnormal{id}$ on $\mathsf{S}_0$.
Next, assume we have constructed the desired map $\phi_k$ on $\mathsf{S}_k$. Then take a minimal $(k+1)$-subsimplex $\sigma$. First, if $\sigma\subset \Sigma$, we define $\phi_{k+1}=\textnormal{id}$ on $\sigma$, then it is clear $\phi_{k+1}|_{\partial\sigma}=\phi_k$. Next, assume $\sigma$ is not in $\Sigma$. Then
by the observation at the beginning, we can extend $\phi_k|_{\partial\sigma}$ to $\phi_{k+1}$ on $\sigma$ with $\phi_{k+1}|_{\partial\sigma}=\phi_k|_{\partial\sigma}$, and it satisfies assertion (1)(2) when replacing $\Sigma$ by $\sigma\cap \Sigma$ (note that this is 
a subsimplex of $\sigma$ of dimension at most $k$).
Repeating this process to every minimal $(k+1)$-subsimplex $\sigma$, using the fact that extended maps on different $\sigma$ agree with each other on their intersection, we can glue  together the maps $\phi_{k+1}$ defined on all minimal subsimplices to get a continuous function $\phi_{k+1}$ on $\mathsf S_{k+1}$, which satisfies (1)(2) with $\Sigma$ replaced by $\Sigma\cap\mathsf S_{k+1}$.
By induction, we obtain a map $\phi=\phi_{n-2}$ defined on the entire simplex $\Delta^*=\mathsf S_{n-2}$.
Since $\phi$ preserves each $k$-subsimplex which is contained in some $\Delta_k^*(i_1,i_2,\dots, i_{k+1})$ in the triangularization, it also satisfies $(3)$. This proves the claim.

\end{proof}

\begin{claim}\label{claim:r}
There exists a smooth contraction map $r:\Delta^*\ri\Delta$ satisfying:
\begin{enumerate}
\item\label{i:r_first} $r(\mathbf x)=\mathbf x$ for any $\mathbf x\in\Delta$.
\item\label{contraction-partial} $r\left(\Delta^*_{k}(i_1,\cdots,i_{k+1})\setminus\Delta_{k}(i_1,\cdots,i_{k+1})\right)\subseteq\partial \Delta_{k}(i_1,\cdots,i_{k+1})$.
\end{enumerate}
\end{claim}

\begin{proof}[Proof of Claim \ref{claim:r}]
For any $w\in \Delta^*_k(i_1,\cdots,i_{k+1})$,
we define $r(w)$ by induction: First, assume for some $\ell=1,\cdots,k$ we obtain 
    $$\mathbf x_\ell=(0,\cdots,\overset{i_1}{\mu_1},\cdots,\overset{i_2}{\mu_2},\cdots,\overset{i_{k+1}}{\mu_{k+1}},\cdots)\in\Delta^*_k(i_1,\cdots,i_{k+1})$$ 
with $0<\mu_1\leq\cdots\leq\mu_\ell$. 
    If $\mu_{\ell}\leq\mu_{\ell+1}$, we take $\mathbf x_{\ell+1}=\mathbf x_\ell$. Otherwise, we take
    $\nu>0$ so that $(i_2-1)\tfrac{\mu_1}{\nu}+\cdots+(i_{\ell+2}-i_\ell)\tfrac{\mu_\ell}{\nu}=\sum_{j=1}^{\ell+1}(i_{j+1}-i_j)\mu_j$, and let
    \[\mathbf x_{\ell+1}=(\cdots,\overset{i_1}{\tfrac{\mu_1}{\nu}},\cdots,\overset{i_\ell}{\tfrac{\mu_\ell}{\nu}},\cdots,\overset{i_{\ell+1}}{\tfrac{\mu_{\ell}}{\nu}},\cdots,\overset{i_{\ell+2}}{\mu_{\ell+2}},\cdots,\overset{i_{k+1}}{\mu_{k+1}},\cdots).\]
    Then the choice of $\nu$ guarantees that $\mathbf x_{\ell+1}\in \Delta^*_k(i_1,\cdots,i_{k+1})$, and the first $(i_{\ell+2}-1)$ entries are non-decreasing.
Repeating this by induction we obtain $\mathbf x_{k+1}\in \Delta_k(i_1,\cdots,i_{k+1})$. 
    It is not hard to see $r$ satisfies all assertions.

\end{proof}

Consider the map $$\Lambda_j=r\circ  \Phi_j\circ\Lambda_j^*:\Omega\ri\Delta.$$ 

Note that $\Lambda_j^*\big(\Omega_{k}(i_1,\cdots,i_{k+1})\big)\subset \Delta^*_{k}(i_1,\cdots,i_{k+1})$ when $i_1=1$, and $\Lambda_j^*\big(\Omega_{n-3}(2,\cdots,n-1)\big)\subset U(\Sigma,\eps_j/2)$, by Claim \ref{claim:phi_j} we see
$$\Phi_j\circ\Lambda_j^*\big(\Omega_{k}(i_1,\cdots,i_{k+1})\big)\subseteq \Delta^*_{k}(i_1,\cdots,i_{k+1})$$ for any $k$ and $i_1,\cdots,i_{k+1}$.
Then it follows from $r|_{\Delta}=\textnormal{id}$ that
\begin{align}\label{eq:Omega}
\Lambda_j(\Omega_{k}(i_1,\cdots,i_{k+1}))\subset \Delta_{k}(i_1,\cdots,i_{k+1}).
\end{align}
Omitting the same indices $i_1,\cdots,i_{k+1}$ in $\Omega_{k}(i_1,\cdots,i_{k+1})$, $\Delta^*_{k}(i_1,\cdots,i_{k+1})$, $\Delta_{k}(i_1,\cdots,i_{k+1})$ for simplicity, then
we claim 
\begin{equation} \label{Delta-Delta*}
\Lambda_j(\Omega_{k})\cap(\Delta_{k}\setminus\partial\Delta_{k})\subset \Phi_j\circ\Lambda_j^*(\Omega_{k}).
\end{equation}
To see this, assume $\mathbf y=r( \Phi_j\circ\Lambda_j^*(\mathbf x))\in \Delta_{k}\setminus\partial\Delta_{k}$ for some $\mathbf x\in \Omega_k$. Then it suffices to show $\mathbf y= \Phi_j\circ\Lambda_j^*(\mathbf x)$. Suppose not, since $r|_{\Delta}=\textnormal{id}$ we must have $\Phi_j\circ\Lambda_j^*(\mathbf x)\notin \Delta$ and thus $ \Phi_j\circ\Lambda_j^*(\mathbf x)\in \Delta^*_{k}\setminus\Delta_{k}$. By \eqref{contraction-partial}  in 
Claim \ref{claim:r}, this implies
$\mathbf y\in \partial \Delta_{k}$, a contradiction.

Since by Claim \ref{claim:phi_j} we have $\Phi_j\circ\Lambda_j^*(\Omega_{0}(k))=\Delta_{0}^*(k)=\Delta_{0}(k)$, it follows by $r|_{\Delta}=\textnormal{id}$ that
    \begin{align*}\label{vertex-surjective2}
        \Lambda_j(\Omega_{0}(k))=\Delta_{0}(k),
    \end{align*}
and by \eqref{eq:Omega} the map $\Lambda_j:\Omega\to \Delta$ preserves all $k$-faces.
 Thus, we can apply Lemma \ref{surjective} and deduce that $\Lambda_j$ is surjection on any $\Delta_{k}(i_1,\cdots,i_{k+1})$:
    \begin{align*}
        &\Lambda_j(\Omega_{k}(i_1,\cdots,i_{k+1}))= \Delta_{k}(i_1,\cdots,i_{k+1}).
          \end{align*}
Thus \eqref{Delta-Delta*} implies $\Delta_k\setminus\partial\Delta_k\subset \Phi_j\circ\Lambda_j^*(\Omega_k)$.
So by induction on $k$ {using \eqref{eqn:partial-Omega-k-1}}, it is easy to see $$\Delta_k\subset \Phi_j\circ\Lambda_j^*(\Omega_k).$$ 

Next, by \eqref{two} and \eqref{three} in Claim \ref{claim:phi_j} we can see
\[\Delta_k\setminus U(\Sigma,\eps_j)\subset \big(\Phi_j\circ \Lambda_j^*(\Omega_k)\big)\setminus U(\Sigma,\eps_j) \subset \Lambda_j^*(\Omega_k).\]
Since $\eps_j\to0$, this implies
\[\Delta_k\subset \overline{\cup_{j=1}^{\infty}\Lambda_j^*(\Omega_k)}.\]
So for any $\mathbf y\in \Delta_k$, for all sufficiently large $j$ we can find $\mathbf x_j\in \Omega_k$
 such that $\mathbf y=\lim_{j\to\infty}\Lambda_j^*(\mathbf x_j)$.
By the same argument as in \cite[Lemma 2.3]{Lai2024}, after passing to a subsequence, the expanding solitons $\mathcal E_j(\mathbf x_j)$ converge to a steady gradient soliton with $\Rm\ge0$ and $R=1$ at the critical point as $j\to\infty$. Moreover, the eigenvalues of the steady soliton at the critical point are given by the vector $\mathbf y$. 
This proves the main theorem.
\end{proof}

\begin{rem}\label{eig}When $n=5$, by Berger’s holonomy classification and a result of Deng-Zhu \cite{DengZhu2020}, one may choose $\lambda_1>0$ (i.e. $\Ric>0$) so that the 2-parameter family of $O(3)$-symmetric non-collapsed steady solitons in Theorem \ref{thm:non-collapsed} has positive curvature operator. More generally, Theorem \ref{thm:soliton-polyhedron} provides a wealth of new examples of positively curved steady solitons in higher dimensions. For instances,  
we can consider the examples from Theorem \ref{thm:soliton-polyhedron} with $0<\lambda_1<\lambda_2<\dots<\lambda_k<\lambda_{k+1}=\cdots=\lambda_n$ and the following symmetry
\[
\underbrace{O(1)\times\cdots\times O(1)}_\text{$k$ times}\times O(n-k),
\]
where $n\ge k+2$. With such a choice of $\lambda_i$, it follows from the symmetry that there exists a new $2$-parameter family of $O(1)\times O(1)\times O(n-2)$ symmetric steady solitons with $\Rm>0$ and 3 distinct eigenvalues of $\Ric$ at the critical point of $f$, for each $n\ge 4$. Similarly by \cite{Lai2022_O(2)}, Theorem \ref{thm:soliton-polyhedron} also gives a new $3$-parameter family of $O(1)\times O(1)\times O(1)\times O(n-3)$ symmetric steady solitons with $\Rm>0$ and $4$ distinct eigenvalues of $\Ric$ at the critical point of $f$, for each $n\ge 5$. Moreover, the $O(1)\times O(1)\times O(n-2)$ and $O(1)\times O(1)\times O(1)\times O(n-3)$ symmetric examples are noncollapsed when $n\ge 5$ and $n\ge 6$ respectively.
\end{rem}

\appendix

\section{Maps preserving cells}

For any $n$,
we denote by $\Delta_n(A_0,\cdots, A_n)$ be the interior of the convex hull spanned by $(k+1)$ linearly independent points $A_0,\cdots,A_n$.
For any for any $0\leq i_0< \cdots< i_k\leq n$ and all $k=0,\cdots,n$, we say 
$\Delta_k(A_{i_0},\cdots,A_{i_k})$ is a $k$-face of $\Delta_n(A_0,\cdots, A_n)$.
In the next lemma, we show that if a continuous map on $\Delta_n$ maps each face to itself, then it must be surjective restricted on every face.

\begin{lma}\label{surjective}
Let $A_0,A_1,\cdots,A_n\in \mathbb{R}^n$, let $A_0=\mathbf{0}$, and $A_i=(0,\cdots,\overset{i}{1},\cdots,0)$, for each $i=1,\cdots,n$. 
Suppose $f:\Delta_n(A_{0},\cdots,A_{n})\rightarrow\Delta_n(A_{0},\cdots,A_{n})$ is a continuous map, and for any $0\leq i_0< \cdots< i_k\leq n$ and all $k=0,\cdots,n$ that
    $$f(\Delta_k(A_{i_0},\cdots,A_{i_k}))\subset\Delta_k(A_{i_0},\cdots,A_{i_k}).$$
   Then $\textnormal{deg}\, f|_{\partial \Delta_k(A_{i_0},\cdots,A_{i_k})}=1$, and 
    $$f(\Delta_k(A_{i_0},\cdots,A_{i_k}))=\Delta_k(A_{i_0},\cdots,A_{i_k}).$$
\end{lma}

\begin{proof}
    We prove this Lemma by induction on $k$. 
    First, it holds trivially for $k=0$. Next, assume this holds for some $k\ge0$.
    Next, let $\Delta_{k+1}$ be an arbitrary $(k+1)$-face, and denote $X=\partial \Delta_{k+1}$. Let $A$ be some $k$-face of $X$, and let $U=X\setminus A$.

Since $f_*:H_{k-1}(\partial A)\to H_{k-1}(\partial A)$ is an isomorphism by the inductive assumption,
it follows by the following long exact relative homology sequence of $(A,\partial A)$, which commute with $f_*$, that $f_*:H_{k}(A,\partial A)\to H_{k}(A,\partial A)$ is an isomorphism.
\[\cdots \longrightarrow 0\longrightarrow H_{k}(A)=\mathbb Z \overset{\cong}{\longrightarrow} H_{k}(A,\partial A)=\mathbb Z\longrightarrow 0\longrightarrow\cdots\]
Note $H_k(X,\overline U)=H_k(A,\partial A)$ by the excision lemma, together with the long exact sequence which commute with $f_*$,
\[\cdots \longrightarrow 0\longrightarrow H_{k}(X)=\mathbb Z \overset{\cong}{\longrightarrow} H_{k}(X,\overline U)=\mathbb Z\longrightarrow 0\longrightarrow\cdots,\]
 we see $f_*:H_{k}(X)\to H_{k}(X)$ is an isomorphism. So $\textnormal{deg} \, f|_{X}=1$, and in particular $f: \Delta_{k+1}\to \Delta_{k+1}$ is surjective.
This proves the lemma by induction.
\end{proof}

\section{Warped product metric and curvature}

\begin{lma}\label{lma:Rm-sphere-warp}
Let $m\ge 1$ and $\beta\in(0,1]$.
Suppose $(N,h)$ is an $m$-dimensional Riemannian manifold (not necessarily complete) such that either one of the following holds
\begin{itemize}
    \item $m=1$ and $h=\b_2^2dx_2^2$ for some constant $\b_2>0$;
    \item $m\ge 2$ and with $\Rm(h)\ge 1$.
\end{itemize}
Let $$g=\b^2(dx^2+\sin^2(x)h),$$ where $x\in [0,\pi]$.
Then $\Rm(g)\ge \b^{-2}$ for $x\in (0,\pi)$.
\end{lma}

\begin{proof}
When $m=1$, upon a reparametrization, the metric can be written as 
\[
dr^2+\b^2\b_2^2\sin^2\left(\frac{r}{\b}\right)dx_2^2.
\]
It has constant Gauss curvature given by $\b^{-2}\ge 1$. Thus the lemma is true when $m=1$. Suppose $(N,h)$ is an $m$-dimensional Riemannian manifold (not necessarily complete) with $\Rm(h)\ge 1$ and $m\ge 2$. Let $g=\b^2(dx^2+\sin^2(x)h)$, where $\b\in (0,1]$ and $x\in [0,\pi]$. Using the change of variables $r=\b x$, $g$ can be rewritten as
\[
g=dr^2+\b^2\sin^2\left(\frac{r}{\b}\right)h.
\]
Let $\{e_i\}_{i=0}^{m}$ be a local orthonormal frame on $(0,\pi)\times N$ with respect to $g$ such that $e_0=\partial_r$. Denote its dual frame by $\{e_i^*\}_{i=0}^{m}$. By the curvature formula of warped product metric \cite[Appendix A]{Li2012}, for $i,j,k,l \ge 1$, 
\begin{eqnarray*}
    \Rm(e_0^*\,\wedge e_j^*, e_k^*\,\wedge e_l^* )&=&0,\\
    \Rm(e_0^*\,\wedge e_j^*, e_0^*\,\wedge e_k^* )&=&\frac{1}{\beta^2}\delta_{jk}\left(\csc^2\left(\frac{r}{\b}\right)-\cot^2\left(\frac{r}{\b}\right)\right)=\frac{1}{\beta^2}\delta_{jk},\\
    \Rm(e_i^*\,\wedge e_j^*, e_k^*\,\wedge e_l^* )&=&\frac{1}{\beta^6}\csc^6\left(\frac{r}{\b}\right)\Rm(h)(e_i^*\,\wedge e_j^*, e_k^*\,\wedge e_l^*)\\
    &&-\frac{1}{\beta^2}\cot^2\left(\frac{r}{\b}\right)(\delta_{ik}\delta_{jl}-\delta_{il}\delta_{jk}).
\end{eqnarray*}
Together with our assumption $\Rm(h)\ge 1$, we have $\Rm(g)\ge \beta^{-2}\ge 1$.
\end{proof}

\end{document}